\documentclass[11pt]{amsart}

\usepackage{mathtools}
\usepackage{amsmath}
\usepackage{amssymb}
\usepackage{yhmath}
\usepackage{graphicx}
\usepackage{ mathrsfs }
\usepackage{bbm}
\usepackage{xcolor}

\newtheorem{theorem}{Theorem}[section]

\newtheorem*{claim*}{Claim}

\newtheorem{lemma}[theorem]{Lemma}
\newtheorem{lem}[theorem]{Lemma}
\newtheorem{corollary}[theorem]{Corollary}

\newtheorem{cor}[theorem]{Corollary}
\newtheorem{proposition}[theorem]{Proposition}
\newtheorem{Prop}[theorem]{Proposition}
\newtheorem{prop}[theorem]{Proposition}
\newtheorem{Thm}[theorem]{Theorem}
\newtheorem{thm}[theorem]{Theorem}

\theoremstyle{definition}
\newtheorem{definition}[theorem]{Definition}\newtheorem{Def}[theorem]{Definition}
\newtheorem{example}[theorem]{Example}

\theoremstyle{remark}
\newtheorem{remark}[theorem]{Remark}\newtheorem{rmk}[theorem]{Remark}
\newtheorem{Rmk}[theorem]{Remark}

\numberwithin{equation}{section}

\newcommand{\norm}[1]{\lVert#1\rVert}

\newcommand{\M}{\mathcal{M}}
\newcommand{\op}{\operatorname}
\newcommand{\bb}{\mathbb}
\newcommand{\hu}{\op{hull}}

\newcommand{\Ga}{\Gamma}
\newcommand{\s}{\mathsf{s}}
\newcommand{\ga}{\gamma}

\newcommand{\la}{\lambda}
\newcommand{\La}{\Lambda}
\newcommand{\ba}{\backslash}

\newcommand{\cl}{\overline}
\newcommand{\til}{\tilde}
\newcommand{\cal}{\mathcal}
\newcommand{\br}{\mathbb R}
\newcommand{\SO}{\op{SO}}
\newcommand{\core}{\op{core}}
\newcommand{\coree}{\op{core}_\epsilon(\M)}
\newcommand{\bH}{\mathbb H}
\newcommand{\BFM}{\op{BF} \M}
\newcommand{\RFM}{\op{RF} \M}
\newcommand{\RFPM}{\op{RF}_+ \M}
\newcommand{\RF}{\op{RF}}
\newcommand{\FM}{\op{F} \M}
\newcommand{\hull}{\op{hull}}
\newcommand{\q}{\mathbb Q}
\newcommand{\lan}{\langle}
\newcommand{\be}{\begin{equation}}
\newcommand{\ee}{\end{equation}}
\newcommand{\ran}{\rangle}
\newcommand{\G}{\Gamma}
\newcommand{\m}{\mathsf{m}}
\renewcommand{\O}{\mathcal O}
\newcommand{\mS}{\mathscr{S}}
\newcommand{\mG}{\mathscr{G}}\newcommand{\mH}{\mathscr{H}}

\newcommand{\T}{\op{T}}\newcommand{\F}{\op{F}}
\renewcommand{\frak}{\mathfrak}
\newcommand{\mT}{\mathsf T}
\newcommand{\e}{\varepsilon}\renewcommand{\epsilon}{\varepsilon}
\newcommand{\BR}{\op{BR}}
\newcommand{\NC}{\op{N}_G(H_{nc})}

\begin{document}

\title[Orbit closures]{Orbit closures of Unipotent flows for hyperbolic manifolds with Fuchsian ends}

\author{Minju Lee}

\address{Department of Mathematics, Yale University, New Haven, CT 06520}
\address{current: Department of Mathematics, University of Chicago, Chicago, IL 60637}
\email{minju1@uchicago.edu}

\author{Hee Oh}
\address{Mathematics department, Yale university, New Haven, CT 06511 and Korea Institute for Advanced Study, Seoul, Korea}
\email{hee.oh@yale.edu}

\thanks{Oh was supported in part by NSF Grants \#1900101.}



\begin{abstract} We establish an analogue of Ratner's orbit closure theorem for any connected closed subgroup
generated by unipotent elements in $\SO(d,1)$ acting on the space $\Gamma\ba \SO(d,1)$, 
assuming that the associated hyperbolic manifold $\M=\Gamma\ba \bH^d$ is a convex cocompact manifold with Fuchsian ends.
For $d=3$, this was proved earlier by McMullen, Mohammadi and Oh.
In a higher dimensional case, the possibility of accumulation on closed orbits of intermediate subgroups causes serious  issues, but
in the end, all orbit closures of unipotent flows are relatively homogeneous.
Our results imply the following: for any $k\ge 1$,\begin{enumerate}
\item the closure of any $k$-horosphere in $\M$  is
  a properly immersed submanifold;
  \item the closure of any geodesic $(k+1)$-plane in $\M$
  is  a properly immersed submanifold;
\item an infinite sequence of maximal properly immersed geodesic $(k+1)$-planes 
intersecting  $\text{core } \M$ becomes dense in $\M$.
\end{enumerate}

\end{abstract}

\maketitle

\tableofcontents

\section{Introduction}
Let $G$ be a connected simple linear Lie group and $\Gamma <G$ be a discrete subgroup.
An element $g\in G$ is called {\it unipotent} if all of its eigenvalues are one, and
a closed subgroup of $G$ is called unipotent if all of its elements are unipotent. Let $U$ be a connected unipotent subgroup of $G$, or more generally, any connected
closed  subgroup of $G$ generated by unipotent elements in it. We are interested in the action of $U$ on the homogeneous space $\Gamma\ba G$ by right translations.

If the volume
of the homogeneous space $\Gamma\ba G$ is finite, i.e., if $\Gamma$ is a lattice in $G$,
then Moore's ergodicity theorem  says that for almost all $x\in \Gamma\ba G$,
$xU$ is dense in $\Gamma\ba G$ \cite{Moo}.
While this theorem does not provide any information for a given point $x$, the celebrated Ratner's orbit closure theorem,
which was a  conjecture of Raghunathan,  states that 
\begin{equation} \label{ho} \text{the closure of every $U$-orbit is homogeneous},\end{equation}  that is,
for any $x\in \Gamma\ba G$, $\overline{xU}=xL$
 for some connected closed subgroup $L<G$ containing $U$  \cite{Ra2}.  Ratner's proof 
 is based on her classification  of all $U$-invariant ergodic probability measures \cite{Ra1} and the  work of Dani and Margulis \cite{DM2} on the non-divergence of unipotent flow. Prior to her work,
 some important special cases of \eqref{ho} were established  by Margulis \cite{Mar}, Dani-Margulis (\cite{DM1}, \cite{DM3})
 and Shah (\cite{Sh}, \cite{Sh2}) by topological methods. 
 This theorem is a fundamental result with numerous applications.

It is natural to ask if there exists a family of homogeneous spaces of infinite volume where 
an analogous orbit closure theorem holds. 
When the volume of $\Gamma\ba G$ is infinite, the geometry of the associated locally symmetric space turns out to play an important role in this question. The first orbit closure theorem in the infinite volume case was established by
McMullen, Mohammadi, and Oh (\cite{MMO1}, \cite{MMO2}) 
for a class of homogeneous spaces  $\Gamma\ba \SO(3,1)$ which arise as the frame bundles of convex cocompact hyperbolic $3$-manifolds with Fuchsian ends.

Our goal in this paper is to show that a similar type of orbit closure theorem holds in the higher dimensional analogues of these manifolds.
We present a complete hyperbolic $d$-manifold $\M=\Gamma\ba \bH^d$ as the quotient
of the hyperbolic space by the action of a discrete subgroup
$$\Gamma < G=\SO^\circ (d, 1)\simeq \op{Isom}^+(\bH^d) $$
where $\SO^\circ(d,1)$ denotes the identity component of $\SO(d,1)$.
The geometric boundary of $\bH^d$ can be identified with the sphere ${\mathbb S}^{d-1}$.
The limit set $\Lambda\subset \mathbb S^{d-1}$ of $\Gamma$ is the set of all accumulation points of an orbit $\Gamma x$ in the compactification
$\bH^d\cup \mathbb S^{d-1}$
for $x\in \bH^d$. 

The convex core of $\M$ is a submanifold of $\M$ given by the quotient
$$\core {\M}=\Gamma\ba \op{hull}(\Lambda)$$
where $\hull(\Lambda) \subset \bH^d$ is the smallest convex subset containing all geodesics in $\bH^d$ connecting points in $\Lambda$. 
When $\core {\M}$ is compact, $\M$ is called {\it convex cocompact}.

\begin{figure}[h]
\centering
      \includegraphics[totalheight=5cm]{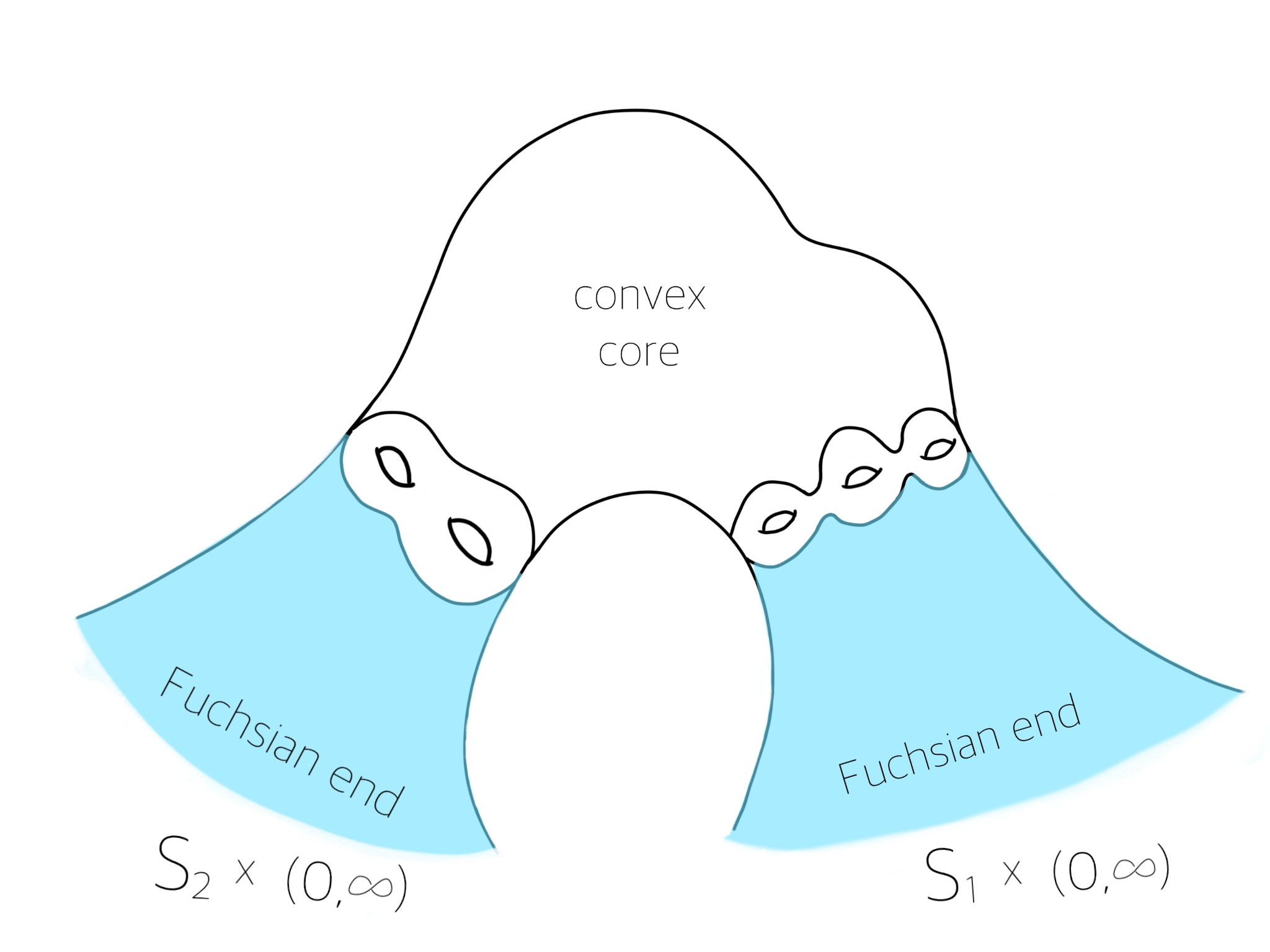}
\caption{A convex cocompact hyperbolic manifold with non-empty Fuchsian ends}
\label{Fend}
\end{figure} 
\subsection*{Convex cocompact manifolds with Fuchsian ends}
Following the terminology introduced in \cite{KS}, we define:
 \begin{Def} A convex cocompact hyperbolic $d$-manifold $\M$ is said to have Fuchsian ends
 if  $\core {\M}$ has non-empty interior and has  totally geodesic boundary.
 \end{Def}
The term {\it Fuchsian ends} reflects the fact that each component of the boundary of $\core {\M}$
is a $(d-1)$-dimensional closed hyperbolic manifold, and each component of the complement $\M-\core \M $  is diffeomorphic 
to the product $S\times (0,\infty)$ for some closed hyperbolic $(d-1)$-manifold $S$ (see Figure \ref{Fend}).

Convex cocompact hyperbolic $d$-manifolds with non-empty Fuchsian ends can also be characterized
as convex cocompact hyperbolic manifolds whose limit sets satisfy: $${\mathbb S}^{d-1}-\Lambda =\bigcup_{i=1}^\infty B_i$$
where $B_i$'s are round balls with mutually disjoint closures (see Figure \ref{limit}). Hence for $d=2$, any non-elementary convex cocompact hyperbolic surface has Fuchsian ends.
The double of the  core of a convex cocompact hyperbolic $d$-manifold with non-empty Fuchsian ends
 is a closed hyperbolic $d$-manifold.

\begin{figure}[h]
\centering
        \includegraphics[totalheight=4cm]{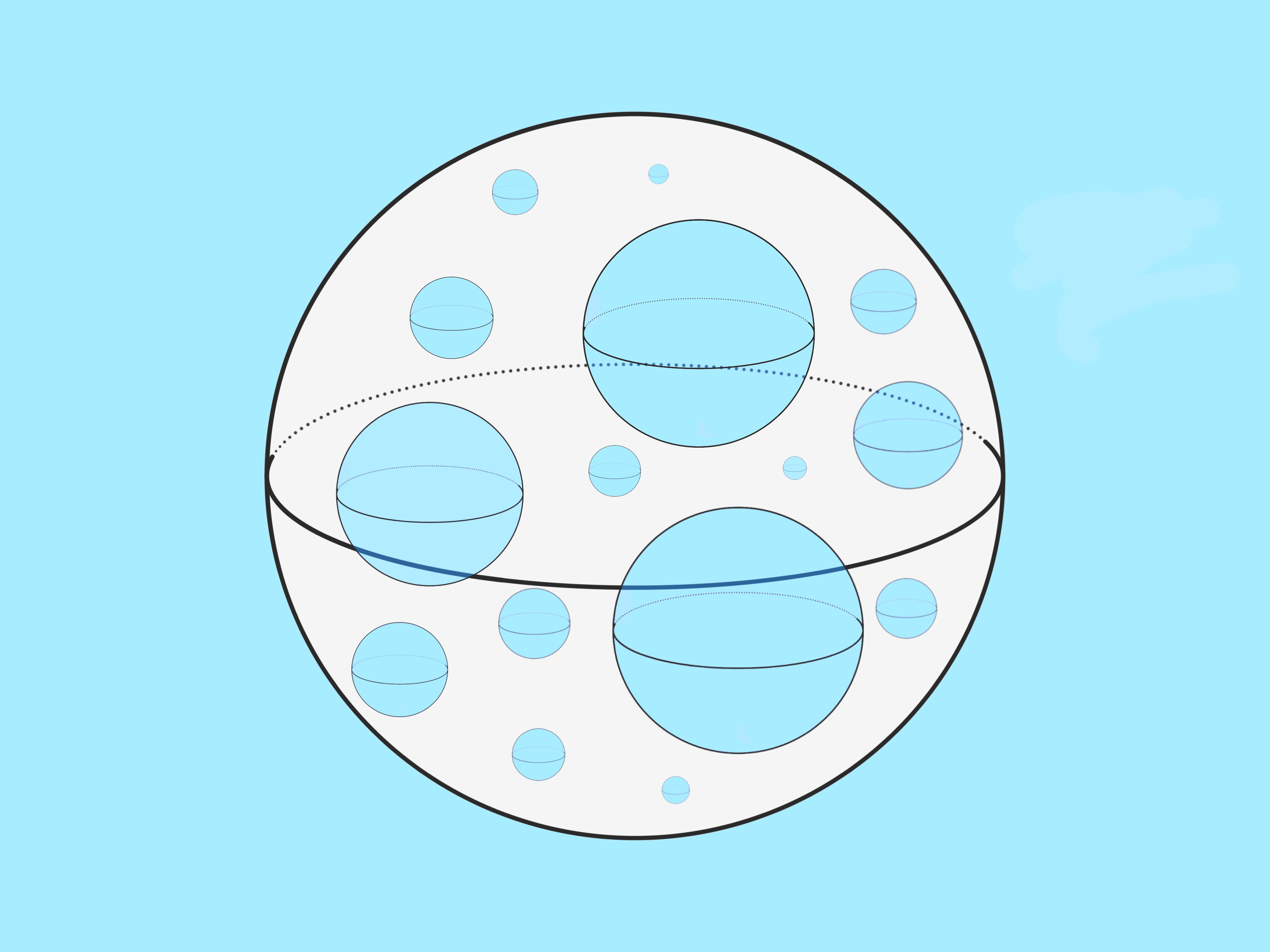} \caption{Limit set of a convex cocompact hyperbolic $4$-manifold with non-empty Fuchsian ends}
\label{limit}
\end{figure}

Any convex cocompact hyperbolic manifold with non-empty Fuchsian ends is constructed in the following way.
 Begin with a closed hyperbolic $d$-manifold $\cal N_0$ with a fixed collection of finitely many, mutually disjoint,
 properly embedded totally geodesic hypersurfaces. Cut $\cal N_0$ along those hypersurfaces to obtain a compact hyperbolic manifold $W$ with totally geodesic boundary hypersurfaces. 
There is a canonical procedure of
extending each boundary hypersurface to a Fuchsian end, which results in  a convex cocompact hyperbolic manifold $\M$ (with Fuchsian ends) which is diffeomorphic to the interior of $W$.

By  Mostow rigidity theorem, there are only countably infinitely many convex cocompact
hyperbolic manifolds with Fuchsian ends of dimension at least $3$. On the other hand, for a fixed closed hyperbolic $d$-manifold $\cal N_0$ with infinitely many properly immersed geodesic hypersurfaces,\footnote{Any closed arithmetic hyperbolic manifold
has  infinitely many properly immersed geodesic hypersurfaces provided  it has at least one. This is due to the
presence of Hecke operators \cite{Re}.}
 one can produce  infinitely many non-isometric convex compact hyperbolic $d$-manifolds with non-empty Fuchsian ends;
for each properly immersed geodesic hypersurface $f_i(\bH^{d-1})$ for a totally geodesic immersion $f_i:\bH^{d-1}\to \cal N_0$, there is a finite covering $\cal N_i$ of $\cal N_0$ such that $f_i$ lifts to $\bH^{d-1}\to \cal N_i$ with image $S_i$ being properly imbedded in $\cal N_i$ \cite{MR}. Cutting and pasting $\cal N_i$ along $S_i$ as described above produces
a  hyperbolic manifold $\M_i$ with Fuchsian ends. When the volumes of $S_i$ are distinct, $\M_i$'s are not isometric to each other.

\subsection*{Orbit closures}  
In the rest of the introduction, we assume that for $d\ge 2$,
$$\text{ $\M$ is a convex cocompact hyperbolic $d$-manifold with Fuchsian ends.}$$

The homogeneous space $\Gamma\ba G$ can be regarded as the bundle $\FM$ of oriented frames over $\M$.
Let $A=\{a_t: t\in \br\} <G$ denote the one parameter subgroup of diagonalizable elements whose right translation actions on
$\Gamma\ba G$
correspond to the frame flow. Let $N\simeq \br^{d-1}$ denote the contracting horospherical subgroup: 
$$N=\{g\in G:
a_{-t} ga_{t}\to e \text{ as $t\to +\infty$}\} .$$

 We denote by $\RFM$ the renormalized frame bundle of $\M$:
$$\RFM:=\{x\in \Gamma\ba G: \text{ $xA$ is bounded}\},$$
and also set $$ \RFPM:=\{x\in \Gamma\ba G: \text{$xA^+$ is bounded}\} $$
where $A^+=\{a_t: t\ge 0\}$.
When $\op{Vol}(M)<\infty$, we have 
 $$\RFM=\RFPM=\Gamma\ba G.$$ In general, $\RFM$ projects into $\core {\M}$ (but not surjective in general)
 and $\RFPM$ projects onto $\M$ under the basepoint projection $\Gamma\ba G\to \M$.
The sets $\RFM$ and $\RFPM$ are precisely non-wandering sets for the actions of $A$ and $N$ respectively \cite{Win}.

For a connected closed subgroup $U<N$,
we denote by $H(U)$
the smallest  closed simple Lie subgroup of $G$
which contains both $U$ and $A$. If $U\simeq \br^{k}$, then $H(U)\simeq \SO^\circ(k+1,1)$.
A connected closed subgroup of $G$ generated by  one-parameter unipotent subgroups is, up to conjugation, of the form
 $U<N$ or $H(U)$ for some $U<N$ (Corollary \ref{dic}).

We set $F_{H(U)}:=\RFPM \cdot H(U)$, which is a closed subset.  
It is easy to see that if $x\notin \RFPM$ (resp. $x\notin F_{H(U)}$),  then $xU$ (resp. $xH(U)$)
 is closed in $\Gamma\ba G$. 
 On the other hand, for almost all $x\in \RF_+\M$,
$xU$ is dense in $\RFPM$, with respect to  a unique $N$-invariant  locally finite measureon $\RFPM$, called the Burger-Roblin measure;
this was shown by Mohammadi-Oh \cite{MO} for $d=3$ and by Maucourant-Schapira for general $d\ge 3$ \cite{MS} (see section \ref{s:err}).
\subsection*{Orbit closures are relatively homogeneous}
We define the following collection of closed connected subgroups of $G$:
\begin{equation*}
\mathcal L_U:=\left\{L=H(\widehat U) C: 
\begin{array}{c}
\text{for some $z\in \RFPM$, $zL $ is closed in $\Gamma\ba G$  }\\  \text{ and $\op{Stab}_L(z)$ is Zariski
dense in $L$}
\end{array}
\right\}.
\end{equation*}
 where $U <\widehat U<N$ and $C$ is a closed subgroup of the centralizer of $H(\widehat U)$.
We also define: $$\cal{Q}_U:=\{vLv^{-1} :  L\in\cal{L}_U\text{ and } v\in N\}.$$

In view of the previous discussion, the following theorem gives a classification of orbit closures for all connected closed subgroups of $G$ generated by unipotent one-parameter subgroups:

\begin{theorem}\label{mtp} Let $\M=\Gamma\ba \bH^d$ be a convex cocompact hyperbolic manifold with Fuchsian ends, and let
 $U<N$ be a non-trivial connected closed subgroup.
\begin{enumerate}
\item {\rm ($H(U)$-orbit  closures)} For any $x\in \RFM \cdot H(U)$,

\begin{equation*}
\cl{xH(U)}=xL\cap F_{H(U)}
\end{equation*}
where $xL$ is a closed orbit of  some $L\in \mathcal L_{U}$.

\item{\rm ($U$-orbit closures)}
  For any $x\in  \RFPM$, 
 \begin{equation*}
\cl{xU}=x  L\cap\op{RF}_+\M
\end{equation*}
where  $x L$ is a closed orbit of  some $L\in \cal{Q}_U$.

\item {\rm (Equidistributions)} Let $x_iL_i $ be a sequence of  closed orbits  intersecting $\RFM$, where $x_i\in\RFPM$ and
 $L_i\in\cal{Q}_U$. 
Assume that  no infinite subsequence of $x_iL_i$ is contained in a subset
of the form $y_0L_0D$ where  $y_0L_0$ is a closed orbit of $L_0\in \cal L_U$ with $\op{dim}L_0<\op{dim} G$  and $D$ is a compact subset of the normalizer $\op{N}(U)$ of $U$.
Then\footnote{For a sequence of subsets $Y_n$ in a topological space $X$ such that $Y=\limsup_n Y_n=\liminf Y_n$,  we write $Y=\lim_{n\to \infty} Y_n$, where
$\limsup_n Y_n= \bigcup_n \cl{\bigcap_{m\geq n} Y_m }$ and   
$\liminf_n Y_n= \bigcap_n \cl{\bigcup_{m\geq n} Y_m }.
$ }
$$\lim_{i\to\infty}
\text{ } x_i L_i\cap\op{RF}_+\M= \op{RF}_+\M.
$$
\end{enumerate}
\end{theorem}

\begin{rmk}\label{EEnd}
\begin{enumerate}
\item If $x\in F_{H(U)}- \RFM \cdot H(U)$, then
$xH(U)$ is contained in an end component of $\M$ under the projection $\Gamma\ba G\to \M$,
 and its closure is not relatively homogeneous in $F_{H(U)}$.  More precisely,
 $$\overline{xH(U)}= xL V^+ H(U) $$
 for some $L\in \mathcal L_U$, and some one-parameter semigroup $V^+<N$  (see Theorem \ref{b-ratner}).

 \item If $\M$ has empty ends, i.e., if $\M$ is compact,
 Theorem \ref{mtp}(1) and (2) are special cases of Ratner's theorem \cite{Ra2}, also proved by Shah \cite{Sh0} independently, and
  Theorem \ref{mtp}(3) follows from  Mozes-Shah equidistribution theorem \cite{MSh}.
\end{enumerate}
\end{rmk}

Theorem \ref{mtp}(1) and (2)  can  also be presented as follows in a unified manner:
\begin{cor} \label{mt2} Let
$H<G$ be a connected closed subgroup generated by unipotent elements in it.
 Assume that $H$ is normalized by $A$.
For any $x\in \RFM$, the closure of $xH$ is homogeneous in $\RFM$, that is,
\begin{equation}\label{uc} \overline{xH}\cap \RFM=xL \cap \RFM\end{equation}
where $xL$ is a closed orbit of some $L\in \mathcal Q_{H\cap N}$.
\end{cor}

\begin{Rmk} If $\Gamma$ is  contained in $G(\q)$ for some $\q$-structure of $G$,
and $[g]L$ is a closed orbit  appearing in Corollary \ref{mt2},
then $L$ is defined by the condition that
$gLg^{-1}$ is the smallest connected $\q$-subgroup of $G$ containing $gHg^{-1}$.
\end{Rmk}

\subsection*{Generic points} Denote by $\mG(U)$ the set of all points $x\in \RFPM$ such that $x$ is not contained in any closed orbit
 of a proper  reductive algebraic subgroup of $G$ containing $U$.
Theorem \ref{mtp}(2) implies that for any $x\in \mG(U)$, 
$$\cl{xU}=\RFPM.$$

\subsection*{Geodesic planes, horospheres and spheres}We state implications of our main theorems on the closures of 
geodesic planes and horospheres of the manifold $\M$, as well as on the $\Gamma$-orbit closures of  spheres in $\mathbb S^{d-1}$.

A geodesic $k$-plane $P$ in $\M$ is the image of a totally geodesic immersion $f:\bH^k\to \M$, or equivalently, the image of a geodesic
$k$-subspace of $\bH^d$ under the covering map $\bH^d\to \M$. If $f$ factors through the covering map 
$\bH^k \to \Gamma_0\ba \bH^k$
for a convex cocompact hyperbolic $k$-manifold with Fuchsian ends, we call $P=f(\bH^k)$ a convex cocompact geodesic $k$-plane with Fuchsian ends.

\begin{thm} \label{geo-intro1} Let $\M=\Gamma\ba \bH^d$ be a convex cocompact hyperbolic manifold with Fuchsian ends, and
let $P$ be a geodesic $k$-plane of $\M$ for some $k\ge 2$.
\begin{enumerate}
\item If $P$ intersects $\core {\M}$, then
$\cl P$
is a properly immersed  convex cocompact geodesic $m$-plane  with Fuchsian ends for some $m\ge k$.
\item Otherwise, $P$ is contained in some Fuchsian end $E=S_0\times (0,\infty)$ of $\M$, and either $P$ is properly immersed or
$\cl{P}$ is diffeomorphic to the product $S\times [0,\infty)$
for a closed geodesic $m$-plane $S$ of $ S_0$ for some $k\le m\le d-1$.
\end{enumerate}
In particular, the closure of a geodesic plane of dimension at least $2$ is a properly immersed submanifold of $\M$ (possibly with boundary).
\end{thm}

We also obtain:
\begin{thm} \label{geo-intro2}
\begin{enumerate}

\item Any infinite sequence of maximal properly immersed geodesic planes $P_i$ of $\op{dim}P_i\ge 2$ intersecting $\core {\M}$ becomes dense in $\M$, i.e.,
$$\lim_{i\to\infty} P_i= \M $$
where the limit is taken in the Hausdorff topology on the space of all closed subsets in $\M$.
\item There are only countably many properly immersed geodesic
planes of dimension at least $2$ intersecting $\core {\M}$.

\item If $\op{Vol}(\M)=\infty$, there are only finitely many maximal properly immersed bounded
geodesic planes of dimension at least $2$.
\end{enumerate}

\end{thm}

In fact, Theorem \ref{geo-intro2}(3) holds for any convex cocompact hyperbolic $d$-manifold (see Remark 
 \ref{nochange}).

A  $k$-horosphere in $\bH^d$ is a Euclidean sphere of dimension $k$ which is tangent to a point in $\mathbb S^{d-1}$. A $k$-horosphere in $\M$ is simply the image of a $k$-horosphere in $\bH^d$ under the covering map $\bH^d\to \M=\Gamma\ba \bH^d$.

\begin{thm} \label{geo-intro3} Let $\chi$ be a $k$-horosphere of $\M$ for $k\ge 1$.
Then either
\begin{enumerate}
\item $\chi$ is properly immersed; or
\item
$\cl{\chi}$ is
 a properly immersed $m$-dimensional submanifold, parallel to a convex cocompact geodesic $m$-plane of $\M$ with Fuchsian ends for some $m\ge k+1$.
\end{enumerate}
\end{thm}

 By abuse of notation, let $\pi$ denote both base point projection maps
$G\to \bH^d$  and $\Gamma\ba G\to \M$ where we consider an element $g\in G$ as an oriented frame
over $\bH^d$. Let $H'=\SO^\circ (k+1,1)\SO(d-k-1)$, $1\le k\le d-2$. The quotient space $G/H'$ parametrizes all oriented $k$-spheres in $\mathbb S^{d-1}$, which we denote by $\mathcal C^{k}$.
For each $H'$-orbit $gH'\subset G$,  the image $\pi(gH')\subset \bH^d$ is an oriented geodesic $(k+1)$-plane
 and the boundary $\partial(\pi(gH'))\subset \mathbb S^{d-1}$  is an oriented  $k$-sphere.  Passing to the quotient space $\Gamma\ba G$,
 this gives bijections among:
\begin{enumerate}
\item
the space of all closed $H'$-orbits  $xH'\subset \Gamma\backslash G$ for $x\in \RFM$;
\item
the space of all oriented properly immersed geodesic $(k+1)$-planes $P$ in $\M$ intersecting $\core {\M}$;
\item
the space of all closed $\Gamma$-orbits of oriented $k$-spheres $C\in \mathcal C^{k}$ with $\# C\cap \Lambda \ge 2$
\end{enumerate}
If $U:=H'\cap N$, then any $k$-horosphere in $\M$
 is given by $\pi(xU)$ for some $x\in \Gamma\ba G$. 
 
 In view of these correspondences, Theorems \ref{geo-intro1}, \ref{geo-intro2} and \ref{geo-intro3} follow from Theorem \ref{mtp}, Theorem \ref{b-ratner}, and Corollary \ref{rcount}.

 We also obtain the following  description on $\Gamma$-orbits of spheres of any positive dimension. 
 
  \begin{cor}\label{cor.sphere interpretation}\label{sphere} Let $1\le k\le d-2$.
\begin{enumerate}
\item 
Let $C\in \mathcal C^{k}$ with
$\# C\cap \Lambda \ge 2$.
Then there exists a sphere $S\in \cal C^m$
such that $\Gamma S$ is closed in $\cal C^m$ and
$$\cl{\Ga C}=\{D\in \cal C^k :D\cap \Lambda \ne \emptyset,\; D \subset  \Gamma S\}.$$

\item Let $C_i\in \cal C^k$ be an infinite sequence of spheres  with $\#C_i\cap \Lambda \ge 2$ such that
$\Gamma C_i $ is closed in  $\mathcal C^k$. 
Assume that $\Ga C_i$ is maximal in the sense that there is no 
  proper sphere $S\subset \bb S^{d-1}$ which properly contains $C_i$ and that $\Ga S$ is closed.
Then as $i\to\infty$, 
$$\lim_{i\to \infty} \Ga C_i= \{D\in \cal C^k :D\cap \Lambda \ne \emptyset\}$$
where the limit is taken in the Hausdorff topology on the space of all closed subsets in $\mathcal C^k$.
\item If $\Lambda \ne \bb S^{d-1}$, there are only finitely many maximal closed $\Gamma$-orbits of spheres of positive dimension
contained in $\Lambda$. 
\end{enumerate}
\end{cor}

\begin{Rmk}

\begin{enumerate}
\item  The main results of this paper for $d=3$ were proved  by McMullen, Mohammadi, and Oh  (\cite{MMO1}, \cite{MMO2}). We refer to \cite{MMO1} for counterexamples to Theorem \ref{mtp} for a certain family of quasi-Fuchsian $3$-manifolds.
\item A convex cocompact hyperbolic $3$-manifold with Fuchsian ends (which was referred to as {\it a rigid acylindrical hyperbolic $3$-manifold} in \cite{MMO1})
 has a huge deformation space parametrized by the product of the Teichmuller spaces of the boundary components of $\core {\M}$ (cf. \cite{Ma}).  Any
convex cocompact acylindrical hyperbolic $3$-manifold is a quasi-conformal conjugation of a rigid acylindrical hyperbolic $3$-manifold \cite{Mc}.
An analogue of Theorem \ref{mtp}(1) was obtained for all convex cocompact acylindrical hyperbolic $3$-manifolds in \cite{MMO3} 
and  for all geometrically finite acylindrical hyperbolic $3$-manifolds in \cite{BO}.

\item For $d\ge 4$, Kerckhoff and Storm showed that a convex cocompact
hyperbolic manifold $\M=\Gamma\ba \bH^d$ with non-empty Fuchsian ends does not allow any non-trivial deformation, in the sense that
  the representation of $\Gamma$
into $G$ is infinitesimally rigid
\cite{KS}.
\end{enumerate}
\end{Rmk}

\begin{Rmk}
  We discuss an implication of Theorem \ref{mtp}(2)
on the classification question on $U$-invariant ergodic locally finite measures on $\RFPM$.
There exists a canonical geometric $U$-invariant measure on each closed orbit $xL$ in Theorem \ref{mtp}(2):
we write $L=v^{-1} H(\widehat U)Cv $. As $v$ centralizes $U$, let's assume $v=e$ without loss of generality.
 Denoting by $p: L\to H(\widehat U)$
 the canonical projection, the subgroup $p(\op{Stab}_L(x))$
 is a convex cocompact Zariski dense subgroup of $H(\widehat U)$, and hence
there exists a unique  $\widehat U$-invariant locally finite
measure on $ p(\op{Stab}_L(x)) \ba H(\widehat U) $,  called the Burger-Roblin measure (\cite{Bu}, \cite{Ro},
\cite{OS}, \cite{Win}).
 Now its $C$-invariant lift  to $ (L\cap \op{Stab}_L(x))\ba L$ defines a
unique $\widehat U C$-invariant locally finite measure, say $m^{\BR}_{xL}$, 
 whose support is equal to $xL\cap \RFPM$.
Moreover  $m^{\BR}_{xL}$ is $U$-ergodic (cf. section \ref{s:err}).
A natural question is the following:
\begin{multline*}\text{ \it is every ergodic $U$-invariant locally finite Borel measure in $\RFPM$} \\ \text{\it
proportional to some $m^{\BR}_{xL}$?}\end{multline*}
 An affirmative answer would provide an analogue of Ratner's measure classification \cite{Ra1} in this setup.
 Theorem \ref{mtp}(2) implies that the answer is yes, at least in terms of the support of the measure.
\end{Rmk}

\noindent{\bf Acknowledgement} We would like to thank Nimish Shah for making his unpublished notes, containing
 most of his proof of Theorems \ref{mtp}(1) and (2)  for the finite volume case, available to us. We would also like to thank Gregory Margulis, Curt McMullen, and Amir Mohammadi
for useful conversations. Finally, Oh would like to thank  Joy Kim  for her encouragement.

\section{Outline of the proof}
We will explain the strategy of our proof of Theorem \ref{mtp} with an emphasis on the difference between finite and infinite volume case
and the difference between dimension $3$ and higher case.

\subsection*{Thick recurrence of unipotent flows} Let $U_0=\{u_t:t\in \br\}$ be a one-parameter subgroup of $N$.
The main obstacle of carrying out unipotent dynamics in a homogeneous space of {\it  infinite }volume
is the scarcity of recurrence of unipotent flow.
In a compact homogeneous space, every $U_0$-orbit stays in a compact set for the obvious reason.  Already in a {\it noncompact} homogeneous space of finite volume, understanding the recurrence of $U_0$-orbit is a non-trivial issue. Margulis showed that
any $U_0$-orbit is recurrent to a compact subset \cite{Ma1}, and Dani-Margulis  showed that
for any $x\in \Gamma\ba G$, and for any $\epsilon>0$, there exists a compact subset $\Omega\subset \Gamma\ba G$ such that
$$\ell \{t\in [0, T]: xu_t\in \Omega\} \ge (1-\epsilon) T$$
for all large $T\gg 1$, where $\ell$ denotes the Lebesgue measure on $\br$ \cite{DM2}.
This non-divergence of unipotent flows is an important ingredient of Ratner's orbit closure theorem \cite{Ra2}.

In contrast, when $\Gamma\ba G$ has infinite volume, for any compact subset $\Omega\subset \Gamma\ba G$, and for almost all $x$
(with respect to any Borel measure $\mu$ on $\br$),
$$\mu\{t\in [0, T]: xu_t\in \Omega\} =o(T)$$
for all $T\gg 1$ \cite{Aa}.

 Nonetheless, the pivotal reason that we can work
with convex cocompact hyperbolic manifolds of non-empty Fuchsian ends
 is the  following {\it thick} recurrence property that they possess: there exists $k>1$, depending only on the systole of the double of $\core {\M}$, such that
for any $x\in \RFM$,
the return time $$\mathsf T (x):=\{t\in \br: xu_t\in \RFM\}$$ is $k$-thick, in the sense that
for any $\la > 0$, 
\begin{equation}\label{window} \mathsf T (x)\cap \left( [-k\la,-\la]\cup [\la, k\la]\right)\ne \emptyset .\end{equation}

This recurrence property was first observed by McMullen, Mohammadi and Oh \cite{MMO1}  in the case of dimension $3$
 in order to get an additional invariance of a relative $U_0$-minimal subset with respect to $\RFM$
by studying the polynomial divergence property of  $U_0$-orbits of two nearby $\RFM$-points.

\subsection*{Beyond $d=3$} In a higher dimensional case, 
the possible presence of closed orbits of intermediate subgroups introduces a variety of serious  hurdles. 
Roughly speaking, calling the collection of all such closed orbits  as the singular set and its complement as the generic set,
one of the main new ingredients of this paper is the {\it avoidance of the singular set} along the $k$-thick recurrence of $U_0$-orbits to $\RFM$ for a sequence of $\RFM$-points limiting at a generic point.  Its analogue in the finite volume case was proved by Dani-Margulis \cite{DM} and also independently by Shah \cite{Sh1} based on the
 {\it linearization methods}.

\subsection*{Road map for induction} 
Roughly speaking,\footnote{To be precise, we need to carry out  induction on the co-dimension of $U$ in $\widehat L\cap N$
whenever $xU$ is contained in a closed orbit $x_0\widehat L$ for some $\widehat L\in \mathcal L_U$  as formulated  in Theorem \ref{mainth}.} 
Theorem \ref{mtp} is proved by induction on the co-dimension of $U$ in $N$. For each $i=1,2,3$, let us say that $(i)_m$ holds, if Theorem \ref{mtp}$(i)$ is true for all $U$ satisfying $\op{co-dim}_N(U)\leq m$.
We  show that
the validity of  $(2)_m$ and $(3)_m$ implies that of $(1)_{m+1}$,  
the validity of $(1)_{m+1}$, $(2)_m$, and $(3)_m$ implies that of $(2)_{m+1}$ and
the validity of $(1)_{m+1}$, $(2)_{m+1} $, and $(3)_m$ implies that of $(3)_{m+1}$.
In order to give an outline of the proof of $(1)_{m+1}$, we 
suppose  that $\op{co-dim}_N(U)\leq m+1$. Let 
$$F:= \RFPM \cdot H(U),\;\;  F^*:=\text{Interior}( F),\text{ and }\, \partial F:=F-F^* .$$ 
 
Let $x\in F^*\cap \RFM$, and consider
 $$X:=\overline{xH(U)}\subset F.$$

The strategy in proving $(1)_{m+1}$ for $X$ consists of two steps:
\begin{enumerate}
\item  (Find) Find a closed $L$-orbit $x_0L$ {\it with $x_0\in F^*\cap \RFM$} such that $x_0L\cap F$ contained in $X$ for some $L\in \mathcal L_U$;

\item (Enlarge) If $X\not\subset x_0L \op{C}(H(U))$,\footnote{The notation $\op{C}(S)$ denotes the identity component of the
centralizer of $S$} then  enlarge $x_0L$ to a bigger closed orbit $x_1 \widehat L$
so that $x_1\widehat L\cap F\subset X$ where  $x_1\in F^*\cap \RFM$ and $\widehat L\in\cal L_{\widehat U}$ for some $\widehat U<N$ 
containing $L\cap N$ properly. \end{enumerate}

The enlargement process must end after finitely many steps because of dimension reason.  
Finding a closed orbit as in (1) is based on the study of the relative $U$-minimal sets and  the unipotent blow up argument using the polynomial divergence of $U$-orbits of nearby $\RFM$-points. 
To explain  the enlargement step, suppose that we are
given an intermediate closed $L$-orbit with $x_0L\cap F\subset X$ by the step (1), and a one-parameter subgroup $U_0=\{u_t\}$ of $U$ such that $x_0U_0$ is dense in $x_0L\cap \RFPM$.  
As $L$ is reductive, the Lie algebra of $G$ can be decomposed into the $\op{Ad}(L)$-invariant subspaces 
$\mathfrak l\oplus \mathfrak l^\perp$ where $\mathfrak l$ denotes the Lie algebra of $L$.
Suppose that we could arrange a sequence $x_0g_i \to x_0$ in $ X$ for some $g_i \to e$
such that writing $g_i= \ell_i r_i$ with $\ell_i\in L$ and $r_i \in  \exp (\mathfrak l^\perp)$, 
the following conditions are satisfied:
\begin{itemize}
\item  $x_0\ell_i \in \RFM$;
\item $r_i \notin   \op{N}(U_0) $.
   \end{itemize}
 
Then the $k$-thick return property of $x_0\ell_i\in \RFM$ along $U_0$ yields 
a sequence $u_{t_i}\in U_0 $ such that  $$x_0\ell_i u_{t_i}\to x_1\in \RFM\cap x_0L\;\;\text{ and }\;\; u_{t_i}^{-1} r_iu_{t_i}\to v$$ 
 for some element $v \in N-L$.  This gives us a point 
 $$x_1 v \in X .$$
  \begin{equation}\label{prog} \text{{\it If we could guarantee that 
$x_1$ is a {generic} point  for $U$ in $x_0L$,}}\end{equation} 
  then $\overline{x_1U}$ must be equal to $x_0L\cap \RFPM$ by induction hypothesis $(2)_m$, since 
the co-dimension of $U$ inside $L\cap N$ is at most $m$.
Then  $$\overline{x_1vU}=\overline{x_1U} v= x_0L v \cap \RF_+\M \subset X .$$ 

Using the $A$-invariance of $X$ and the fact that
the double coset $AvA$ contains a one-parameter unipotent
subsemigroup $V^+$, we can put $x_0LV^+ \cap F$ inside $ X$. 
\be \label{star}\text{\it Assuming that $x_0\in F^*\cap \RFPM$,}\ee
 we can promote $V^+$ to a one-parameter subgroup $V$, and
find an orbit of a bigger unipotent subgroup  $\widehat U:=(L\cap N)V$  inside $ X$. This enables us to use the induction hypothesis $(2)_m$ 
to complete the enlargement step. Note that if $x_1$ is not generic for $U$ in $x_0L$, the closure of $x_1U$ may be stuck in a smaller closed orbit inside
$x_0L$, in which case $\overline{x_1U}v$ may not be bigger than $x_0L$ in terms of the dimension, resulting in no progress.

We now explain how we establish \eqref{prog}.\footnote{For the dimension $d=3$,  $L$ is either the entire $\SO^{\circ}(3,1)$ in which case we are done, or
 $L= H(U)=\SO^\circ(2,1)$. In the latter case, \eqref{prog} is automatic as $U$ is a horocyclic subgroup of $L$.}

\subsection*{Avoidance of the singular set along the thick return time} Let $U_0=\{u_t\}$ be a one parameter subgroup of $U$. We denote by $\mathscr{S}(U_0)$
the union of all closed orbits $xL$ where $x\in \RFPM$ and $L\in \mathcal Q_{U_0}$ is a {\it proper} 
subgroup of $G$.  This set is called the {\it singular set} for $U_0$.
Its complement in $\RFPM$
 is denoted by $\mathscr{G}(U_0)$, and called the set of {\it generic} elements of
$U_0$. 
We have
$$\mS(U_0)=\bigcup_{H\in \mH}\Gamma\ba \Gamma X(H, U_0)$$
where $\mH$ is the countable collection of all proper connected closed subgroups $H$ of $G$ containing a unipotent element such that
$\Gamma\ba \Gamma H$ is closed and $H\cap \Gamma$ is Zariski dense in $H$,
and $X(H, U_0):=\{g\in G: gU_0g^{-1}\subset H\}$ (Proposition \ref{sudef}).
We define $\cal E=\cal E_{U_0}$ to be the collection of all subsets of $\mS(U_0)$ which are
of the form
$$\bigcup \Gamma\ba \Gamma H_i D_i \cap \RFM$$
where $H_i\in \mH$ is a finite collection, and $D_i$ is a compact subset of $X(H_i, U_0)$.
The following avoidance theorem is one of the main ingredients of our proof:  let $k$ be given by
\eqref{window} for $\M=\Gamma\ba \bH^d$:

\begin{Thm}[Avoidance theorem] \label{mtt2} 
There exists a sequence of compact subsets $E_1\subset E_2\subset \cdots $ in 
$\cal E$ with $$\mathscr{S}(U_0)\cap \RFM=\bigcup\limits_{j=1}^\infty E_j $$ satisfying the following:
for each $j\in \mathbb N$ and  for any compact subset $F\subset \op{RF}\M- E_{j+1}$, 
there exists an open neighborhood $\cal O_j=\cal O_j(F)$ of $ E_j$ such that for any $x\in F$, 
the following set \begin{equation}\label{omj} 
\{t\in \bb{R} : xu_t\in\op{RF}\M-\cal O_j\}
\end{equation}
is $2k$-thick.
\end{Thm}
It is crucial that the thickness size of the set \eqref{omj}, which is given by $2k$ here, can be
controlled independently of the compact subsets $E_j$ for applications in the orbit closure theorem.
If $E_j$  does not intersect any closed orbit of a proper subgroup of $G$, then obtaining $E_{j+1}$ and $\O_j$ is much simpler. In general, $E_j$ may 
intersect  infinitely many intermediate
closed orbits, and our proof
is based on a careful analysis on the graded intersections of those closed orbits 
 and a combinatorial argument, which we call an {\it inductive search argument}.
 This process is quite delicate, compared to the finite volume case treated in
(\cite{DM}, \cite{Sh1}) in which case the set $\{t: xu_t\in \RFM\}$, being equal to $\br$, possesses the Lebesgue measure which can be used to measure the time outside of a neighborhood of $E_j$'s.

We deduce  the following from Theorem \ref{mtt2}:
\begin{Thm}[Accumulation on a generic point]  \label{mtl} 
Suppose that $(2)_m$ and $(3)_m$ hold in Theorem \ref{mtp}. Then the following holds for
any connected closed subgroup $U<N$ with $\op{co-dim}_{N}(U)= m+1$:
Let  $U_0=\{u_t:t\in \br\}$ be a one-parameter subgroup of $U$, and let $x_i\in \op{RF}\M$  be a sequence converging to $x_0\in\mathscr G(U_0) $ as $i\to\infty$.
Then for any given sequence $T_i\to\infty$, 
\begin{equation}\label{dmggg}
\limsup_{i\to\infty}\{ x_i u_{t_i}\in \RFM: T_i\le |t_i| \le 2kT_i \}
\end{equation}
contains a sequence $\{y_j:j=1,2, \cdots \}$ such that $\limsup_{j\to \infty} y_jU $ contains a point in $\mathscr{G}(U_0)$.\footnote{Here we allow a constant sequence $y_j=y$ in which case
 $\limsup_{j\to \infty} y_jU$ is understood as $\overline{yU}$ and hence $y\in \mG(U_0)$.}
\end{Thm}

 Again, it is important that $2k$ is independent of $x_i$ here. 
We prove two independent but related versions of Theorem \ref{mtl} in section \ref{s:li} depending
on the relative location of $x_i$ for the set $\RFM$; we use Proposition \ref{cor.lin} for the proof of $(1)_{m+1}$ and Proposition \ref{lem.lin} for the proofs of $(2)_{m+1}$
and $(3)_{m+1}$.


\subsection*{Comparison with the finite volume case}
 If $\Gamma\ba G$ is compact, the approach of Dani-Margulis   \cite{DM} shows that
if $x_i$ converges to $x \in \mG(U_0)$,  
then for any $\e>0$, we can find a sequence of compact subsets $E_1\subset E_2\subset \cdots $ in $\cal E$, and neighborhoods $\cal O_j$
of $E_j$  such that 
$\mS(U_0)=\bigcup E_j$, $x_i\notin \bigcup_{j\le i+1} \cal O_{j}$ and for all $i\ge j$ and $T>0$,
$$ \ell \{t\in [0,T] : x_iu_t\in \cal O_j\} \le \tfrac{\e }{2^i}T.$$
This implies that for all $i>1$,
\begin{equation}\label{dmg}
 \ell \{t\in [0,T] : x_iu_t\in \bigcup_{j\le i} \cal O_j\} \le \e T. \end{equation}

In particular, the limsup set in \eqref{dmggg} always contains an element of $\mG(U_0)$, without using induction hypothesis.
This is the reason why $(3)_m$ is not needed in obtaining $(1)_{m+1}$ and $(2)_{m+1}$
in Theorem \ref{mainthlattice} for the finite volume case.\footnote{We give a summary of our proof for the case
when $\Gamma\ba G$ is compact and has at least one $\SO^\circ(d-1,1)$ closed orbit in the appendix
to help readers understand the whole scheme of the proof.}

In comparison, we are able to get a generic point in Theorem \ref{mtl}
only with the help of the induction hypothesis $(2)_m$ and $(3)_m$ and after taking the limsup of the $U$-orbits of all accumulating points from the $2k$-thick sets obtained in Theorem \ref{mtt2}.

\subsection*{Generic points in $F^*$ as limits of $\RFM$-points}
In the inductive argument, it is important to find a closed orbit $x_0L$ based at a point $x_0\in F^*$ in order to
 promote a semi-group $V^+$ to a group $V$ as described following \eqref{star}. Another reason why this is critical is the following:
implementing  Theorem \ref{mtl} (more precisely, its versions Theorems \ref{cor.lin} and \ref{lem.lin}) requires having a sequence of
{\it $\RFM$-points of $X$} accumulating on a generic point of $x_0L$ with respect to $U_0$.

The advantage of having a closed orbit $x_0L$ with $x_0\in F^*\cap \RFPM\cap \mG(U_0)$ is that $x_0$ can be  approximated by a sequence of $\RFM$-points in $F^*\cap X$ (Lemmas \ref{lem.frameshift} and \ref{xfss}). 

We also point out that we use the ergodicity theorem
obtained in \cite{MO} and \cite{MS} to guarantee that there are many $U_0$-generic points in any closed orbit $x_0L$ as above.

\subsection*{Existence of a compact orbit in any noncompact closed orbit}
In our setting, $\Gamma\ba G$ always contains a closed orbit $xL$ for some $x\in \RFM$ and
a proper subgroup $L\in \mathcal L_U$; namely those compact orbits of
$\SO^\circ (d-1,1)$ over the boundary of $\core {\M}$. 
Moreover, if $x_0\widehat L$ is a noncompact closed orbit for some $x_0\in \RFM$ and $\op{dim}(\widehat L\cap N)\ge 2$,
then $x_0\widehat L$ contains a compact
orbit $xL$ of some $L\in\mathcal L_U$ (Proposition \ref{SAVE}).
This fact was crucially used in deducing $(2)_{m+1}$ from $(1)_{m+1}, (2)_m$ and $ (3)_m$ in Theorem \ref{mtp}
(more precisely, in Theorem \ref{mainth}).

\subsection*{Organization of the paper}
\begin{itemize}
\item In section \ref{s:ba}, we set up notations for certain Lie subgroups of $G$, review some basic facts and gather preliminaries
about them and geodesic planes of $\M$. 
\item In section \ref{s:ri},  for each unipotent subgroup $U$ of $G$, we
define the minimal $H(U)$-invariant closed subset $F_{H(U)}\subset \Gamma\ba G$ containing $\RFPM$ and
study  its properties  for a
  convex cocompact hyperbolic manifold of non-empty Fuchsian ends. 
\item In section \ref{s:st}, we define the singular set  $\mS(U, x_0L)$ for a closed orbit $x_0L\subset \Gamma\ba G$,
 and prove a structure theorem and a countability
theorem  for a general convex cocompact manifold.

\item  In section \ref{s:ig}, we prove Proposition \ref{prop.2kthick}, based on a combinatorial lemma \ref{lem.search}, called an
{\it inductive search lemma}. This proposition is used in the proof of Theorem \ref{avoid1} (Avoidance theorem).
\item In section \ref{s:lt}, we construct families of triples of intervals which satisfy the hypothesis of Proposition \ref{prop.2kthick},
by making a careful analysis of the  graded intersections of the singular set and the linearization, and
prove Theorem \ref{avoid1} from which Theorem \ref{mtt2} is deduced.
\item In section \ref{s:ge}, we prove several geometric lemmas which are needed to modify a sequence
limiting on a generic point to a sequence of $\RFM$-points which still converges to a generic point. 

\item In section \ref{s:un}, we study the unipotent blowup lemmas using quasi-regular maps and properties
of thick subsets.
\item In section \ref{s:tr}, we study the translates of relative $U$-minimal sets $Y$ into the orbit closure
of an $\RFM$ point; the results in this section are used in the step of finding a closed orbit in a given $H(U)$-orbit closure.
\item In section \ref{s:cl}, we describe closures of orbits contained in the boundary of $F_{H(U)}$.
\item In section \ref{s:err}, we review the ergodicity theorem of \cite{MO} and \cite{MS} and deduce the density of almost all orbits
of a connected unipotent subgroup in $\RFPM$.
\item In section \ref{s:um}, the minimality of a horospherical subgroup action is obtained in the presence of compact factors.

\item  In section \ref{s:or},
we  begin to prove Theorem \ref{mtp}; the base case $m=0$ is addressed and the orbit closure
of a singular $U$-orbit is classified under the induction hypothesis.
\item In section \ref{s:ge2} we prove two propositions on how to get an additional invariance from Theorem \ref{avoid1}; the results in this section are used in the step of
enlarging a  closed orbit to a larger one inside a given $U$-invariant orbit closure in the proof of Theorem \ref{mtp}.

\item We prove $(1)_{m+1}$, $(2)_{m+1}$
and $(3)_{m+1}$ respectively in sections \ref{s:1}, \ref{s:2} and \ref{s:3}.

\item In the appendix, we give an outline of our proof
in the case when $\Gamma\ba G$ is compact with at least one $\SO^\circ (d-1,1)$-closed orbit.

\end{itemize}

\section{Lie subgroups and geodesic planes}\label{sec.notation}\label{s:ba}
Let $G$ denote the connected simple Lie group $\op{SO}^\circ (d,1)$ for $d\ge 2$.
In this section, we fix  notation and recall some background about  Lie subgroups of $G$ and geodesic planes of a hyperbolic $d$-manifold.

As a Lie group, we have $G\simeq \op{Isom}^+(\bH^d)$.
In order to present a family of subgroups of $G$ explicitly,
we fix a quadratic form $Q(x_1, \cdots, x_{d+1})=
2x_1x_{d+1}+x_2^2+ x_3^2 +\cdots + x_{d}^2$, and identify $G=\op{SO}^\circ(Q)$.
The Lie algebra of $G$ is then given as:
\begin{equation*}
\frak{so}(d,1)=\{X\in\frak{sl}_{d+1}(\bb{R}) : X^tQ+QX=0\}
\end{equation*}
where
\begin{equation*}
Q=\left(
\begin{array}{ccc}
0&0&1\\
0&\op{Id}_{d-1}&0\\
1&0&0
\end{array}
\right).
\end{equation*}

A subset $S\subset G$ is said to be Zariski closed  if $S$ is defined as the zero set
$\{(x_{ij})\in G: p_1(x_{ij})= \cdots = p_\ell(x_{ij})=0\}$ for a finite collection of polynomials with real coefficients
in variables $(x_{ij})\in\op{ M}_{d+1}(\br)$. 
The Zariski closure of a subset $S\subset G$ means the smallest Zariski closed subset of $G$ containing $S$.
A connected sugroup $L<G$ is algebraic if $L$ is equal to the identity component of its Zariski closure.
\subsection*{Subgroups of $G$}

Inside $G$, we have the following subgroups:
\begin{align*}
K&=\{g\in G: g^t g=\op{Id}_{d+1}\} \simeq \SO(d), \\
A&= \left\{a_s=\begin{pmatrix}
e^s&0&0\\
0&\op{Id}_{d-1}&0\\
0&0&e^{-s}
\end{pmatrix}  :s\in \br \right\},\\
M&=\text{ the centralizer of $A$ in $K$} \simeq \SO(d-1), \\
N^- &=\{\exp  u^{-}(x): x\in \br^{d-1}\},\\
N^{+} &=\{\exp  u^{+}(x): x\in \br^{d-1}\},
\end{align*}
where
 $$u^-(x)= \begin{pmatrix} 0&x^t&0\\
0&0&-x\\
0&0&0
\end{pmatrix}\quad\text{and}\quad u^+(x)=\begin{pmatrix}
0&0&0\\
x&0&0\\
0&-x^t&0
\end{pmatrix} .$$
The Lie algebra of $M$ consists of
matrices of the form $$m(C)=\begin{pmatrix}
0&0&0\\
0&C&0\\
0&0&0
\end{pmatrix} $$
where $C\in \op{M}_{d-1}(\br)$ is a skew-symmetric matrix, i.e., $C^t=-C$.

The subgroups $N^{-}$ and $N^+$ are respectively the contracting  and expanding horospherical subgroups of $G$ for the action of $A$.
We have the Iwasawa decomposition $G=KAN^{\pm}$.
As we will be using the subgroup $N^-$ frequently, we simply write $N=N^-$.
We often identify the subgroup $N^{\pm}$ with $\br^{d-1}$ via the map $\exp u^{\pm}(x) \mapsto x$.
For a connected closed subgroup $U<N$, we use the notation $U^\perp$ for the orthogonal complement of $U$ in $N$ as a vector subgroup of $N$,
and $U^t=U^+$ for the transpose of $U$.
 We use the notation $B_U(r)$ to denote the ball of radius $r$ centered at $0$ in $U$
for the Euclidean metric on $N=\br^{d-1}$.

We consider the upper-half space model of $\bH^d=\br^+\times \br^{d-1}$, so that
its boundary is given by $\bb S^{d-1}=\{\infty\} \cup ( \{0\}\times \br^{d-1})$. Set $o=(1, 0,\cdots, 0)$, and
fix a standard basis $e_0, e_1, \cdots, e_{d-1}$ at $\T_o(\bH^d)$.  The map
\be \label{map} g \mapsto (ge_0, \cdots, g e_{d-1})_{g(o)}\ee
gives an identification of $G$ with the oriented frame bundle $\op F \bH^d$.
The stabilizer of $o$ and $e_0$ in $G$ are equal to $K$ and $M$ respectively, and hence the map \eqref{map} induces the identifications of
the hyperbolic space $\bH^d$ and the unit tangent bundle $\op{T}^1\bH^d$  with $G/K$ and $G/M$ respectively.
The action of $G$  on the hyperbolic space $\bH^d=G/K$ extends continuously to the compactification
$\mathbb {\mathbb S}^{d-1}\cup \bH^{d}$.

If $g\in G$ corresponds to a frame $(v_0, \cdots, v_{d-1})\in \F \bH^d$, we define $g^+, g^-\in {\mathbb S}^{d-1}$ to be the forward and backward end points of the directed geodesic tangent to $v_0$ respectively.
The right translation action of $A$ on  $G=\op{F}\bH^d $ defines the frame flow
and we have
$$g^{\pm}=\lim_{t\to\pm\infty} \pi(ga_t)$$
where $\pi:G=\F \bH^d\to \bH^d$ is the basepoint projection.

For the identity element $e=\op{Id}_{d+1}\in G$, note that
 $e^{+}=\infty$, and $e^-=0$, and hence $g^+=g(\infty)$ and $g^-=g(0)$. The subgroup $MA$ fixes both points $0$ and $\infty$, and
the horospherical subgroup $N$ fixes $\infty$, and the restriction of the map $g\mapsto g(0)$
to $N$ defines an isomorphism $N\to \br^{d-1}$ given by $u^-(x)\mapsto x$.

For each non-trivial connected subgroup $U<N$, we denote by 
$$H(U)$$
 the smallest  simple
closed Lie subgroup of $G$ containing $A$ and $U$.
It is generated by $U$ and the transpose of $U$. 

For a subset $S\subset G$, we denote by $\op{N}_G(S)$ and $\op{C}_G(S)$ the normalizer
of $S$ and the centralizer of $S$ respectively.
We denote by $\op{N}(S) $ and $\op{C}(S)$ the identity components of $\op{N}_G(S)$ and $\op{C}_G(S)$ respectively.

\begin{example} \label{exam} Fix the standard basis $e_1,\cdots,e_{d-1}$ of $\br^{d-1}$.
For $1\le k\le d-1$, define $U_{k}$ to be the connected subgroup of $N$ spanned
by $ e_1,\cdots, e_{k}$.  

The following can be checked directly:
\begin{align*} H(U_{k}) &=\lan U_{k}, U_{k}^t\ran \simeq\SO^\circ (k+1,1);\\
\op{C}(H(U_{k}))&\simeq\SO(d-k-1);\\
 \op{N}_G(H(U_{k}))&\simeq\mathrm{O}(k+1,1) \mathrm{O}(d-k-1) \cap G; \\
 \op{N}( H(U_{k}))&\simeq\mathrm{SO}^\circ (k+1,1) \mathrm{SO}(d-k-1) .
\end{align*}
\end{example}
We set $$H'(U):= \op{N}(H(U))=H(U) \op{C}(H(U)),$$
which is a connected reductive algebraic subgroup of $G$ with compact center.

The adjoint action of $M$ on $N$ corresponds to the standard
action of $\SO(d-1)$ on $\br^{d-1}$. It follows that  any connected closed subgroup $U< N$ is conjugate to $U_k$
and $H(U)$ is conjugate to $H(U_k)$ by an element of $M$, where $k=\op{dim}(U)$.

 We set
 \begin{equation}\label{conetwo}\op{C}_1(U):= \op{C}(H(U))= M\cap \op{C}(U),\text{ and } \op{C}_2(U):=M\cap \op{C}(U^\perp)\subset H(U).\end{equation}

\begin{lemma} \label{nu}We have  $$\op{N}(U)=NA \op{C}_1(U) \op{C}_2(U) \text{ and } \op{C}(U)=N \op{C}_1(U). $$ \end{lemma}
\begin{proof}
For the first claim, it suffices to show that for $U=U_k$, $\op{N}(U)=NA \op{SO}(k) \op{SO}(d-1-k)$.
It is easy to check that $Q:=NA \op{C}_1(U) \op{C}_2(U)$ normalizes $U$.
Let $g\in\op{N}(U)$. We claim that $g\in Q$.
Using the decomposition $G=KAN$, we may assume that $g\in K$.
Then $Ug(\infty)=gU(\infty)=g(\infty)$ since $U(\infty)=\infty$.
Since $\infty\in\bb{S}^{d-1}$ is the unique fixed point of $U$, it follows $g(\infty)=\infty$. As $M=\op{Stab}_K(\infty)$, we get  $g\in M$.
Now $gU(0)=Ug(0)=U(0)$.
As  $U(0)=\bb{R}^k$, $g\bb{R}^k=\bb{R}^k$. Therefore, as $g\in M$,
we also have $g\bb{R}^{d-1-k}=\bb{R}^{d-1-k}$, and consequently $g\in \op{O}(k) \op{O}(d-1-k)$.
This shows that $NA \op{SO}(k)\op{SO}(d-1-k)  \subset \op{N}(U)\subset NA \op{O}(k) \op{O}(d-1-k) $.
As $\op{N}(U)$ is connected, this implies the claim.

For the second claim, note first that $N\op{C}_1(U)<\op{C}(U)$.
Now let $g\in\op{C}(U)$.
Since $\op{C}(U)<\op{N}(U)=AN\op{C}_1(U)\op{C}_2(U)$, we can write $g=ac_2nc_1\in A\op{C}_2(U)N\op{C}_1(U)$.
Since $nc_1$ commutes with $U$, it follows $ac_2\in\op{C}(U)$.
Now observe that the adjoint action of $a$ on $U$ is a dilation and the adjoint action of $c_2$ on $U$ is a multiplication by an orthogonal matrix.
Therefore we get  $a=c_2=e$, finishing the proof. 
\end{proof}

Denote by $\frak g= \op{Lie}(G)$ the Lie algebra of $G$. By a one-parameter subsemigroup of $G$, we mean a set of the form $\{\exp (t\xi)\in G:t\ge 0\}$ for some 
non-zero $\xi\in \frak g$.  Note that the product $AU^\perp \op{C}_2(U)$ is a subgroup of $G$.
\begin{lemma}\label{rmk.1psg}
An unbounded one-parameter subsemigroup $S$ of $AU^\perp \op{C}_2(U)$ is one of the following form:
\begin{align*}
&\{\exp(t\xi_A)\exp(t\xi_C)\ : t\geq 0\};\\
&\{\left(v\exp(t\xi_A)v^{-1}\right)\exp(t\xi_C) : t\geq 0\}; \\
&\{\exp(t\xi_V)\exp(t\xi_C) : t\geq 0 \}
\end{align*}
for some $\xi_A\in\op{Lie}(A)-\{0\}, \xi_C\in\op{Lie}(\op{C}_2(U)), v\in U^\perp-\{e\}$, and $\xi_V\in \op{Lie}(U^\perp)-\{0\}$.
\end{lemma}
\begin{proof}
Let $\xi\in\op{Lie}(AU^\perp \op{C}_2(U))$ be such that $S=\{\exp(t\xi) :t\geq 0\}$.
Write $\xi=\xi_0+\xi_C$ where $\xi_0\in\op{Lie}(AU^\perp)$ and $\xi_C\in\op{Lie}(\op{C}_2(U))$.
Since $AU^\perp$ commutes with $\op{C}_2(U)$, $\exp(t\xi)=\exp(t\xi_0)\exp(t\xi_C)$ for any $t\in \br$.
Hence we only need to show that either $\xi_0\in\op{Lie}(U^\perp)$ or
\begin{equation}\label{eq.u2}
\{\exp(t\xi_0) : t\geq 0\}=\{v\exp(t\xi_A)v^{-1} : t\geq0\}
\end{equation}
for some $v\in U^\perp$ and $\xi_A\in\op{Lie}(A)$.
Now if $\xi_0\not\in\op{Lie}(U^\perp)$, then writing 
\begin{equation*}
\xi_0=
\left(
\begin{array}{ccc}
a&x^t &0 \\
0&0_{d-1} &-x\\
0&0 &-a
\end{array}
\right)
\in\op{Lie}(AU^\perp)
\end{equation*}
with $a\neq0$, a direct computation shows that $\xi_0=v\xi_A v^{-1}$ where
\begin{equation*}
\log v=
\left(
\begin{array}{ccc}
0&-x^t/a &0 \\
0&0_{d-1} &x/a\\
0&0 &0
\end{array}
\right)\text{ and }
\xi_A=
\left(
\begin{array}{ccc}
a&0 &0 \\
0&0_{d-1} &0\\
0&0 &-a
\end{array}
\right),
\end{equation*}
proving \eqref{eq.u2}.
\end{proof}

\begin{lemma}\label{vAv}
If $v_i\to\infty$ in $U^\perp$, then $\limsup_{i\to \infty} v_iAv_i^{-1}$ contains one-parameter subgroup of $U^\perp$.
\end{lemma}
\begin{proof}
Writing $v_i=\exp u^-(x_i)$ for $x_i\in \br^{d-1}$, we have
$$v_ia_sv_i^{-1}=
\begin{pmatrix}
e^s&(1-e^s)x_i^t & -\norm{(e^{s/2}-e^{-s/2})x_i}^2/2\\
0&\op{Id}_{d-1} & (1-e^{-s})x_i\\
0&0 &e^{-s} 
\end{pmatrix}.
$$
Passing to a subsequence, $x_i/\norm{x_i}$ converges to some unit vector $x_0$ as $i\to\infty$.
For any $r\in\bb{R}$,  if we set $s_i:=\log(1-{r}{\norm{x_i}^{-1}})$, then
the sequence $v_ia_{s_i}v_i^{-1}$ converges to $ \exp u^-(rx_0)$. Therefore
the set $V:=\{\exp u^-(rx_0):r\in\bb{R}\}<U^\perp$ gives the desired subgroup.
\end{proof}

\noindent{\bf The complementary subspaces $\frak h_U^\perp$ and $\frak h^\perp$.} 
 If $L$ is a reductive Lie subgroup of $G$ with $\frak l=\op{Lie}(L)$,
the restriction of the adjoint representation of $G$ to $L$ is completely reducible, and hence
  there exists an $\op{Ad} (L)$-invariant complementary subspace $\frak l^\perp$ so that
  $$\frak g=\frak l\oplus \frak l^\perp.$$ 
  It follows from the inverse function theorem that
the map $L\times \frak l^\perp \to G$  given by $(g, X)\mapsto g \exp X$ is a local diffeomorphism onto
an open neighborhood of $e$ in $G$.

  Let $U=U_k$. Denote by  $\frak h_U\subset \frak g$ the Lie algebra of $H(U)$,
  by $\frak u^\perp$ the subspace $\op{Lie}(U^\perp)$, and by $(\frak u^\perp)^t$ its transpose.
Then $\frak h_U^\perp$ can be given explicitly as follows:

\be\label{perp0}
\frak h_U^\perp = \frak u^\perp \oplus (\frak u^\perp)^t \oplus \frak m_0
\ee
where $\frak m_0$ is given by $$\left\{m(C) :  C=
\begin{pmatrix}
0&Y\\
-Y^t&Z
\end{pmatrix}, Z\in \op{M}_{d-1-k}(\br), Z^t=-Z, Y\in \op{M}_{k\times (d-1-k)}(\br)
\right\};$$
to see this, it is enough to check that
$\op{dim}(\frak g)=\op{dim}(\frak h_U)+\op{dim}(\frak h_U^\perp)$ and 
that $\frak h_U^\perp$ is $\op{Ad}(H(U))$-invariant, which can be done by direct computation.

Similarly, setting $\frak h:=\op{Lie}(H'(U))$, $\frak h^\perp$
is given by 

\be\label{perp}
\frak h^\perp = \frak u^\perp \oplus (\frak u^\perp)^t \oplus \frak m_0'
\ee
where $$\frak m_0':=\left\{m(C) : C=
\begin{pmatrix}
0&Y\\
-Y^t&0 
\end{pmatrix} 
\right\}.$$

\begin{lem}\label{hen} If $r_i\to e$ in $\exp \frak h^\perp - \op{C}(H(U))$,
then either $r_i\notin \op{N}(U)$ for all $i$, or $r_i\notin \op{N}(U^+)$ for all $i$, by passing to a subsequence.
\end{lem}
 \begin{proof}
 By Lemma \ref{nu} and \eqref{perp}, there exists a neighborhood $\cal O$ of $0$ in $\frak g$ such that 
  $$  \op{N}(U)\cap \op{N}(U^+)   \cap \exp (\frak h^\perp \cap\cal O)\subset \op{C}(H(U)).$$
Hence the claim follows.  \end{proof}
 
\noindent{\bf Reductive subgroups of $G$.}
\begin{Def}
For a connected reductive algebraic subgroup $L<G$, we denote by $L_{nc}$ the maximal connected normal semisimple subgroup of $L$ with no compact factors.
\end{Def}

A connected reductive algebraic subgroup $L$ of $G$ is an almost direct product
\be\label{lnc} L=L_{nc} C T\ee where 
$C$ is a connected semisimple compact normal subgroup of $L$ and  $T$ is the central torus of $L$.
 If $L$ contains a unipotent element, then $L_{nc}$ is non-trivial, and simple,
  containing a conjugate of $A$, and  the center of $L$ is compact.

\begin{prop}\label{str} If $L<G$ is a connected reductive algebraic subgroup normalized by $A$ and containing a unipotent element, then 
$$L=H(U)C$$ where $U<N$ is a non-trivial connected subgroup and $C$ is a closed subgroup of $\op{C}(H(U))$. 
In particular, $L_{nc}$ and $\op{N}(L_{nc})$ are equal to $H(U)$ and $H'(U)$ respectively.
\end{prop}
\begin{proof} 
If $L$ is normalized by $A$, then so is $L_{nc}$. 
Therefore it suffices to prove that
a connected  non-compact simple Lie subgroup $H<G$ normalized by $A$ is of the form $H=H(U)$ where
 $U<N$ is a non-trivial connected subgroup. 

 First, consider  the case when $A<H$. Let $\mathfrak{h}$ be the Lie algebra of $H$, and $\mathfrak{a}$ be the Lie algebra of $A$. Since $\mathfrak h$ is simple, its root space decomposition for the adjoint action of $\mathfrak a$
 is of the form $\mathfrak{h}=\mathfrak{z}(\mathfrak{a})\oplus\mathfrak{u^+}\oplus\mathfrak{u^-}$ where $\mathfrak u^{\pm}$ are the sum  of
 all positive and negative root subspaces respectively and $\mathfrak{z}(\mathfrak{a})$ is the centralizer of $\mathfrak a$.
Since the sum of all negative root subspaces for the adjoint action of $\mathfrak{a}$ on $\frak g$ is $\op{Lie}(N^-)$, it follows that $U:=\exp(\mathfrak{u}^-)<N^-$ and $H=H(U)$.

Now for the general case, $H$ contains a conjugate $gAg^{-1}$ for some $g\in G$. Hence
$g^{-1}Hg=H(U)$. Since $H(U)$ contains both $A$ and $g^{-1}Ag$, they must be conjugate within $H(U)$,
so $A=h^{-1}g^{-1}A gh$ for some $h\in H(U)$. Hence $gh\in \op{N}_G(A)=AM$. Therefore $H=gH(U)g^{-1}$ is equal to
$mH(U)m^{-1}$ for some $m\in M$. Since $m$ normalizes $N$ and $mH(U)m^{-1}=H(mUm^{-1})$, the claim follows. \end{proof}

\begin{cor}\label{dic}
Any connected closed subgroup $L$ of $G$ generated by unipotent elements is conjugate to either $U$ or $H(U)$
for some non-trivial connected subgroup $U<N$.
\end{cor}
\begin{proof}
The subgroup $S$ admits a Levi decomposition $L=S V$ where $S$ is a connected semisimple subgroup with no compact factors
  and $V$ is the unipotent radical of $S$ \cite[Lemma 2.9]{Sh1}.
If $S$ is trivial, the claim follows since any connected unipotent subgroup can be conjugate into $N$.
Suppose that $S$ is not-trivial. 
Then $S=H(U)$ for some non-trivial $U<N$ by Proposition \ref{str}. Unless $V$ is trivial,
the normalizer of $V$ is contained in a conjugate of $NAM$, in particular, it cannot contain $H(U)$. Hence $V=\{e\}$.
\end{proof}

\subsection*{Totally geodesic immersed planes}
Let $\Ga$ be a discrete, torsion free, non-elementary, subgroup of $G$, and consider the associated
hyperbolic manifold $$\M=\Ga\ba\bb{H}^d=\Gamma\ba G/K .$$
We refer to \cite{Rt} for basic properties of hyperbolic manifolds.
As in the introduction, we denote by $\La$ the limit set of $\Ga$ and by $\core\M$ the convex core of $\M$.
Note that $\core {\M}$ contains all bounded geodesics in $\M$.

We denote by $\op{F}\M\simeq \Gamma\ba G$ the bundle of all oriented orthonormal frames over $\M$.
We denote by \be\label{pi1} \pi:\Gamma\ba G\to \M=\Gamma \ba G/ K \ee the base-point projection.
By abuse of notation, we also denote by 
\be\label{pi2} \pi:G\to \bH^d=G/K\ee  the base-point projection. For $g\in G$,
$[g]$ denotes its image under the covering map $G\to\Gamma\ba G$.

Fix $1\le k\le d-2$ and let \begin{equation}\label{standardh} H=H(U_k)\simeq\SO^\circ (k+1, 1)\quad \text{and} \quad H'=H(U_k)\simeq \SO^\circ (k+1,1)\cdot \SO(d-k-1).\end{equation}

Let $C_0:=\br^k\cup\{\infty\}$ denote the unique oriented $k$-sphere in ${\mathbb S}^{d-1}$ stabilized by $H'$.
Then
$\tilde S_0:=\hull(C_0)$ is the unique oriented totally geodesic subspace of $\bH^d$  stabilized by $H'$, and
$\partial \tilde S_0=C_0$.
We note that $H'$ (resp. $H$) consists of all oriented frames $(v_0,\cdots, v_{d-1})\in G$
(resp. $(v_0, \cdots, v_k, e_{k+1}, \cdots, e_{d-1}) \in G$)
such that the $k$-tuple $(v_0, \cdots, v_k)$ is tangent to $ \tilde S_0$, compatible with the orientation of $\tilde S_0$.
The group $G$ acts transitively on the space of all oriented $k$ spheres in ${\mathbb S}^{d-1}$
giving rise to the isomorphisms of $G/H'$ with
 $$\cal C^{k}=\text{ the space of all oriented  $k$-spheres in ${\mathbb S}^{d-1}$}$$
and with
 $$  \text{ the space of all oriented totally geodesic  $(k+1)$-planes of $\bH^d$. }$$

 We discuss the fundamental group of an immersed geodesic $k$-plane $S\subset \M$. Choose a totally geodesic subspace $\tilde S$
of $\bH^d$ which covers $S$. Then $\tilde S=g \tilde S_0$ for some $g\in G$,
and the stabilizer of $\tilde S$ in $G$ is equal to $gH'g^{-1}$.
We have $$\Gamma_{\tilde S}=\{\gamma\in \Gamma: \gamma \tilde S =\tilde S\}=\Gamma\cap gH'g^{-1}$$
and get
an immersion $\tilde f: \Gamma_{\tilde S} \ba \tilde S \to \M$ with image $S$.
Consider the projection map 
\be\label{defp} p: gH'g^{-1} \to gHg^{-1}.\ee
Then $p$ is injective on $\Gamma_{\tilde S}$ and $$ \Gamma_{\tilde S} \ba \tilde S\simeq  p(\Gamma_{\tilde S}) \ba \tilde S$$
is an isomorphism, since $g\op{C}(H)g^{-1}$ acts trivially on $\tilde S$.
Hence $\tilde f$ gives an immersion \be\label{f} f: p(\Gamma_{\tilde S}) \ba \tilde S\to \M\ee with image
 $S$. We say $S$ properly immersed if $f$ is a proper map.

\begin{Prop} \label{proper} 
Let $x\in \Gamma\ba G$, and set $S:=\pi(x H')\subset \M$.
Then \begin{enumerate}
\item  $xH'$ is closed in $\Gamma\ba G$ if and only if $S$ is properly immersed in $\M$.
\item
If $M$ is convex cocompact and $S$ is properly immersed,
then $S$ is convex cocompact and 
$$\partial \tilde S\cap \Lambda=\Lambda(p(\Gamma_{\tilde S}))$$ for any geodesic subspace $\tilde S\subset \bH^d$ which covers  $S$.
\end{enumerate} 
\end{Prop}
\begin{proof} Choose a representative $g\in G$ of $x$ and consider the totally geodesic subspace $\tilde S:=g \tilde S_0$.
Then $S=\op{Im}(f)$ as $f$ given by \eqref{f}. 
Now   the closedness of $xH'$ in $\Gamma\ba G$ is equivalent to
the properness of the map $ ( H' \cap  g^{-1}\Gamma g) \ba H' \to \Gamma\ba G$ induced from  map $h\mapsto xh$.
This in turn is equivalent to the properness of the induced map
$ (H' \cap  g^{-1}\Gamma g ) \ba H'/ (H'\cap K) \to \Gamma\ba G/K$. If $\Delta$ is the image of $H'\cap g^{-1}\Gamma g$ under the projection map $H'\to H$,
then the natural injective map $\Delta\ba H/H\cap K \to (H'\cap g^{-1} \Gamma g) \ba H'/H'\cap K$ is an isomorphism.
Since
$$p(\Gamma_{\tilde S}) \ba \tilde S= p(\Gamma_{\tilde S}) \ba gH/(H\cap K)\simeq \Delta\ba H/(H\cap K),$$
the first claim follows.
The second claim follows from \cite[Theorem 4.7]{OS}.

\end{proof}

\section{Hyperbolic manifolds with Fuchsian ends and  thick return time}\label{s:ri}

In this section, we study the closed $H(U)$-invariant subset $F_{H(U)}:=\RFPM \cdot H(U)$ when $\M=\Gamma\ba \bH^d$ is a convex cocompact manifold with Fuchsian ends. At the end of the section, we address
 the global thickness of the return time of any one-parameter
subgroup of $N$ to $\RFM$.

\begin{Def}\label{d:rigid} A convex cocompact hyperbolic manifold $\M=\Gamma\ba \bH^d$ is said to have non-empty
{\it Fuchsian ends}
if one of the following equivalent conditions holds:
\begin{enumerate}
\item its convex core  has non-empty interior and non-empty totally geodesic boundary.
\item the domain of discontinuity of  $\Gamma$ 
$$\Omega:={\mathbb S}^{d-1}-\Lambda=\bigcup_{i=1}^{\infty} B_i$$
is a dense union of infinitely many round balls with mutually disjoint closures.
\end{enumerate}
\end{Def}

In the whole section, let $\M$ be a  convex cocompact hyperbolic manifold of non-empty Fuchsian ends.
 \subsection*{Renormalized frame bundle} 
The renormalized frame bundle
 $\RFM\subset \FM$ is defined as the following $AM$-invariant subset
 $$\RFM=\{[g]\in \Gamma\ba G: g^{\pm}\in \Lambda\}=\{x\in \Gamma\ba G: xA\text{ is bounded}\}$$
i.e., the closed set
consisting of all oriented frames $(v_0, \cdots, v_{d-1})$ such that the complete geodesic through $v_0$ is contained in
$\core {\M}$.

Unless mentioned otherwise\footnote{At certain places, we use notation $A^+$ for
any subsemigroup of $A$},
we set 
$ A^+=\{a_t: t\ge 0\}.$ We define
$$\RFPM=\{[g]\in \Gamma\ba G: g^{+}\in \Lambda\}=\{x\in \Gamma\ba G: xA^+\text{ is bounded}\}$$
which is a closed $NAM$-invariant subset.
As $\pi(xNA)=\pi(xG)=\M$ for any $x\in \G\ba G$, we have $\pi(\RFPM)=\M$.

\begin{lem}\label{boundedA} For $x\in \RFPM$, $\overline{xA^+}$ meets $\RFM$.
\end{lem}
\begin{proof}
Take any sequence $a_i\to \infty$ in $A^+$. 
Since $xA^+$ is bounded, 
$xa_i$ converges to some $x_0\in \overline{xA^+}$ by passing to a subsequence.
On the other hand, as $A=\liminf a_i^{-1}A^+$, we have
 $x_0A\subset \limsup  (xa_i) (a_i^{-1}A^+)\subset \overline{xA^+}$.
Since $x\in \RFPM$, $\cl{xA^+}$ is bounded, so is $x_0A$. Hence $x_0\in \RFM$ as desired.
\end{proof}

\subsection*{$H(U)$-invariant subsets: $F_{H(U)}, F_{H(U)}^*$, $\partial F_{H(U)}$}
Fix a non-trivial connected subgroup $U<N$, and consider the associated
subgroups $H(U)$ and $H'(U)$ as defined in section \ref{s:ba}.

We define
\begin{equation}\label{fhu} F_{H(U)}:=\op{RF}_+\M\cdot H(U). \end{equation}

We denote by $F_{H(U)}^*$ the interior of $F_{H(U)}$ and by $\partial {F_{H(U)}}$ the boundary of $F_{H(U)}$.
When there is no room for confusion, we will omit the subscript $H(U)$ and simply write $F, F^*$ and $\partial F$.

If $C\subset \bb S^{d-1}$ denotes the oriented $k$-sphere stabilized by $H(U)$, then $g\in F_{H(U)}$ if and only if $gC\cap \La\neq\emptyset$. Therefore the closedness of $F_{H(U)}$  follows from the compactness of $\La$. The set $F_{H(U)}$ is also $H'(U)$-invariant, since $\RFPM$ is $M$-invariant and $\op{C}(H(U))$ is contained in $ M$.
For $g\in G$, we denote by $C_g=C_{gH(U)}\subset \mathbb S^{d-1}$  the sphere given by the boundary of the geodesic plane $\pi(gH(U))$. Then
$\hull C_g=\pi(g(H(U))$, and  
$C_g= gH(U)^+= gH(U)^- $ where $H(U)^{\pm}=\{h^{\pm}: h\in H(U)\}$.
It follows that
\be\label{fhu2} F_{H(U)}=\{[g]\in \Gamma\ba G: C_g\cap \La \ne \emptyset\}.\ee

\begin{lem} Fix $x=[g]\in \Gamma\ba G$. Let $L$ be a closed subgroup of $G$ such that
the closure of  $\pi(gL)$ in $\bH^d\cup \mathbb S^{d-1}$ does not meet $\La$.
Then the map $L\to xL\subset \Gamma\ba G$ given by $\ell \mapsto  x\ell $ is a proper map, and hence $xL$ is closed.
\end{lem}
\begin{proof} Suppose that   $x\ell_i$ converges to some $[g_0]\in \Gamma\ba G$ for some sequence  $\ell_i\to \infty$ in $L$. Then there exists $\gamma_i\in\Gamma$ such that $d(\gamma_i \pi(g\ell_i), \pi(g_0))=d(\pi(g\ell_i), \gamma_i \pi(g_0))\to 0$ as $i\to \infty$.
As $g\ell_i \to \infty$, $\gamma_i \pi(g_0)$ converges to a limit point $\xi\in \La$, after passing to a subsequence.
Hence $\overline{\pi(gL)}\cap \La\ne \emptyset$.
\end{proof}

This lemma implies that if $x\notin \RFPM$ (resp. $x\notin F_{H(U)}$) , then $xU$ (resp. $xH(U)$) is closed for any closed subgroup $U<N$.

\begin{lem}\label{fhu3} If $\M$ is a  convex cocompact hyperbolic manifold of non-empty Fuchsian ends, then
$$F_{H(U)}=\{x \in \Gamma\ba G: \pi(\overline{xH(U)})\cap \core \M\ne \emptyset\} .$$
\end{lem}

\begin{proof}
Denote by $Q$  the subset on the right-hand side of the above equality. To show $F_{H(U)}\subset Q$,
let $x\in F_{H(U)}$. By modifying it using an element of $H(U)$, we may assume that $x\in \RFPM$.
By Lemma \ref{boundedA}, $\overline{xA^+}$ contains $x_0\in \RFM$. Since $x_0A$ is bounded, $\pi(x_0A)$
is a bounded geodesic, and hence
$$\pi(x_0A) \subset  \pi(\overline{xH(U)})\cap \core {\M}$$ because $\core {\M}$ contains all bounded geodesics.
 Therefore $x\in Q$.
To show the other inclusion $Q\subset F_{H(U)}$, we use the hypothesis on $\M$.
 Suppose $x=[g]\notin F_{H(U)}$. Then $C_g\cap \La=\emptyset$, and hence
$C_g$ must be contained in
a connected component, say $B_i$, of $\Omega$. Hence $\pi (gH(U))=\op{hull}(C_g) $ is contained in the interior
of $\hull (B_i)$, which is disjoint from $\hull(\Lambda)$,
by the convexity of $B_i$. Therefore the orbit
 $\Gamma \pi (gH(U)) $ is a closed subset of $\bH^d$, disjoint from
$\hull(\Lambda)$. Hence $x\notin Q$, proving the claim. 
\end{proof}

Note also that
\begin{align}\label{convex} \RFM\cdot H(U)&=\{[g]\in \Gamma\ba G: \# C_g\cap \La\ge 2\}
\\&=\{x \in \Gamma\ba G: \pi({xH(U)})\cap \core {\M}\ne \emptyset\} \notag.\end{align}
This can be seen using the fact that for any two distinct points $\xi^+, \xi^-\in C_g$,
there exists $h\in H(U)$ such that $gh(\infty)=\xi^+$ and $gh(0)=\xi^-$; this fact is clear if $H(U)=H(U_k)$
for some $k$, and a general case follows since $H(U)=mH(U_k)m^{-1}$ for some $m\in M$, and $M$ fixes both $0$ and $\infty$.

Denote by $\M^*$ the interior of the core of $\M$ and by $F_{H(U)}^*$ the interior of $F_{H(U)}$.
 Then  $$F_{H(U)}^*=\{x\in \Gamma\ba G: \pi(x H(U))\cap \M^* \ne \emptyset\}.$$
 To see this, note that the right-hand side is equal to
\begin{align}\label{fffs} &\{ [g]\in F_{H(U)} : \hull C_g \cap \text{Interior}( \hull (\La))\ne \emptyset\} \notag
\\&=\{ [g]\in F_{H(U)} : C_g \not\subset \cl B_i\text{ for any $ i$}\}
\end{align}
which can then be seen to be equal to $F_{H(U)}^*$ in view of \eqref{fhu2}.
Note that \eqref{fffs} implies that for $[g]\in F_{H(U)}^*$, $\#C_g\cap \La \ge 2$ and hence 
\be \label{FFF}F_{H(U)}^*\subset \RFM \cdot H(U).\ee
In particular, $\RFM \cdot H(U)$ is dense in $F_{H(U)}$.

\begin{lemma} \label{lem.R1}\label{ru}
We have $$\RFPM\cap F_{H(U)}^*\subset \RFM\cdot U .$$
\end{lemma}
\begin{proof}  
Let $y\in\op{RF}_+\M\cap F_{H(U)}^*$. We need to show that $yU\cap\op{RF}\M\neq\emptyset.$
Choose $g\in G$ so that $[g]=y$.   As $y\in \RFPM$, $g^+=g(\infty)\in \La$,
and hence $C_g\cap \Lambda\ne \emptyset$.
If $\# C_g\cap \Lambda =1$, then $C_g$ must be contained in $\overline{B_i}$ for some $i$, which implies
$[g]\notin  F^*_{H(U)}$. Therefore $\# C_g\cap \Lambda \ge 2$. We note that
 $gU(0) \cup \{g(\infty)\}= C_g$; this is clear when $U=U_k$ for some $k\ge 1$ and $g=e$, to which a general case is reduced.
Hence
there exists $u\in U$ such that $gu(0)\in \Lambda$. Since $gu(\infty) =g(\infty)\in \La$, we have $yu=[g]u\in \RFM$.
\end{proof}

We denote by  $\partial {F_{H(U)}}$ the boundary of $F_{H(U)}$, that is,
$$\partial F_{H(U)}=F_{H(U)} - F^*_{H(U)}=\{[g]\in F_{H(U)}: C_g\subset \cl{B_i}\text{ for some $i$}\} .$$
When there is no room for confusion, we will omit the subscript $H(U)$ and simply write $F, F^*$ and $\partial F$.

We call an oriented frame $g=(v_0, \cdots, v_{d-1})\in \FM=G$ a boundary frame
if the first $(d-1)$ vectors $v_0, \cdots, v_{d-2}$ are tangent to the boundary of $\core {\M}$.
Set $$\check H:=H(U_{d-2})= \SO^\circ (d-1,1) ,$$
and denote by $\check V$ the one-dimensional subgroup $\br e_{d-1} $ of $N=\br^{d-1}$; note that $ \check V=(\check H\cap N)^\perp$.

We denote by $\op{BF}\M$ the set of all boundary frames of $\M$;
it is a union of  compact $\check H$-orbits:
\begin{equation}\label{eq.BFMDEF}
\op{BF}\M=\bigcup\limits_{i=1}^k z_i\check H
\end{equation}
such that $\pi(z_i \check H)=\Gamma \ba \Gamma \hull(B_i)$ for some component $B_i$ of $\Omega$.

\noindent{\bf The boundary $\partial F_{H(U)}$ for $U<\check H\cap N$.}
 Suppose that $U$ is contained in $\check H\cap N=\br^{d-2}$. Then  there exists a one-parameter semigroup $\check V^+$ of $\check V$
 such that
\begin{equation*}
\partial F= \BFM \cdot  \check V^+ \cdot H'(U). \end{equation*}
We use the notation $\check V^-=\{v^{-1}:v\in \check V^+\}$.
Note that
\begin{equation}\label{bfmz}
\partial F \cap \RFM =\BFM \cdot \op{C}(H(U))\;\; \text{and}\;\; 
\partial F \cap \RFPM =\BFM \cdot \check V^+ \cdot \op{C}(H(U)).
\end{equation}

For a general proper connected closed subgroup $U<N$, $mUm^{-1}\subset \check H\cap N$ for some $m\in M$, and
$$\partial F\cap \RFM = \BFM \cdot m \cdot \op{C}(H(U))$$
where $\BFM \cdot m$ is now a union of finitely many $m^{-1}\check H m$-compact orbits.

\begin{lemma}\label{lem.R0} 
Let  $U <\check H\cap N$, $z\in\op{BF}\M$ and $v\in \check{V}-\{e\}$.
If $zv\in\op{RF}\M$, then $zv\in F^*$.
\end{lemma}
\begin{proof} Let $z=[g]\in \BFM$. Then $\partial (\pi(g\check H))=\partial B_j$ for some $j$.
Let $v\in \check{V}-\{e\}$ be such that $zv\in\op{RF}\M$.
Suppose $zv\in \partial F_{H(U)}$. Then $C_{gv}\subset \overline{B_i}$ for some $i$.
Since the sphere $C_{gv}=\{ gv h(\infty): h\in H(U)\}$ contains $g(\infty)$ which belongs to $\partial B_j$,
we have $i=j$, as $\overline B_i$'s are mutually disjoint.
 As $zv\in \RFM$, $C_{gv}\subset \partial B_j$. Hence $gvH(U)^+\subset g\check H^+$.
It follows that $gv H(U)\subset g \check H$, and hence $vH(U)\cap \check H\ne \emptyset$, which is a contradiction since
$v\notin \check H$, and $H(U)\subset \check H$.
\end{proof}

\subsection*{Properly immersed geodesic planes}
 
Let $H=H(U_k)$ and $H'=H'(U_k)$ be as in \eqref{standardh}, and $p$ be the map in \eqref{defp}.
In \eqref{f}, if $ p(\Gamma_{\tilde S})\ba \tilde S$ is a convex cocompact hyperbolic $k$-manifold with Fuchsian ends  and $f$ is proper, then 
the image $S=\op{Im}(f)$ is referred to as a properly immersed  convex cocompact geodesic $k$-plane of  Fuchsian ends.

\begin{proposition} \label{PPP}
If $xH'$ is closed for $x\in \RFM$, then
  $S=\pi(x H')$ is a properly immersed convex cocompact  geodesic plane with  (possibly empty) Fuchsian ends.
 \end{proposition}

\begin{proof} Choose $g\in G$ so that $x=[g]$. Let $\tilde S$ and $\G_{\tilde S}$ be as in
  Proposition \ref{proper}. 
 Set $C=\partial \tilde S$.
  By loc. cit., $S$ is properly immersed, and $C\cap \La =\La(p(\Gamma_{\tilde S}))$.
  Write  
\be\label{cex}C-(C\cap\La)=\bigcup\limits_{i\in I} (C\cap B_i)\ee where $I$
 is the collection of all $i$ such that $C\cap B_i\neq\emptyset$. 
If $C\cap \Lambda$ contains a non-empty open subset of $C$, then the limit set of $p(\Gamma_{\tilde S})$  is equal to $C$.
Since $p(\Gamma_{\tilde S})$ is convex cocompact by Proposition \ref{proper}, it is a uniform lattice in $gHg^{-1}$, and hence $S$ is compact.
In the other case, $I$ is an infinite set and $\bigcup_{i\in I} (C\cap B_i)$ is dense in $C$; so $S$ is a convex cocompact hyperbolic submanifold with Fuchsian ends by Definition \ref{d:rigid}(2).
\end{proof}

\begin{lemma}\label{Zd}
For any sphere $C$ in ${\mathbb S}^{d-1}$ with $\#  C\cap \Lambda \ge 2$,
the intersection $C\cap \Lambda$ is Zariski dense in $C$.
\end{lemma}
\begin{proof}
The claim is clear if $C\cap \Lambda$ contains a non-empty open subset of $C$.
If not, $C\cap\La$ contains infinitely many $C\cap\partial B_i$'s, each of which is an irreducible co-dimension one real subvariety of $C$. It follows that the Zariski closure of $C\cap\La$ has dimension strictly greater than $\op{dim}C -1$, hence is equal to $C$.
\end{proof}

We
let \be\label{p1} \pi_1:H'\to H,\quad\text{and }\quad \pi_2:H'\to \op{C}(H)\ee denote the canonical
projections.

\begin{proposition} \label{prop.countability} \label{count} \label{rs}\label{cor.HUminimal}
  Suppose that $xH'$ is closed for $x=[g]\in \RFM$, and set $\Gamma':=g^{-1}\Gamma g \cap H'$.
  Then
  \be\label{czero} \cl{xH}=xHC\ee where $C=\overline{\pi_2(\Gamma')}$ and $HC$
is equal to the identity component of the
 Zariski closure of $\Gamma'$. 
 \end{proposition}

\begin{proof} Without loss of generality, we may assume $g=e$.
As $H'$ is a direct product $H\times \op{C}(H)$, we write an element of $H'$ as
$(h,c)$ with $h\in H$ and $c\in \op{C}(H)$.
For all $\ga\in\Ga'$, 
\begin{equation*}
xH=[(e,e)]H=[(e,\pi_2(\ga))] H =[(e,e)] H \pi_2(\ga)
\end{equation*}
and hence $xH=xH \pi_2(\Ga')$.
It follows that $xH C\subset \cl{xH}$.

To show the other inclusion, let $(h_0,c_0)\in H\op{C}(H)$ be arbitrary.
If $[(h_0,c_0)]\in \cl{xH}=\cl{[(e,e)] H}$, then there exist
sequences  $\ga_i\in\Ga'$ and $h_i\in H$ such that $\ga_i(h_i,e)\to(h_0,c_0)$ in $H'$ as $i\to\infty$.
In particular, $\pi_2(\ga_i)\to c_0$ in $\op{C}(H)$ as $i\to\infty$ and hence $c_0\in C=\cl{\pi_2(\Ga')}$. 
This finishes the proof of \eqref{czero}. Let $W$ denote the identity component of the Zariski closure of $\Ga'$ in $H'$. 
Since any proper algebraic subgroup of $G$ stabilizes either a point, or a proper sphere
in ${\mathbb S}^{d-1}$, it follows from
 Proposition \ref{proper} and Lemma \ref{Zd} that $\pi_1(\G') $ is Zariski dense in $H$; so  $\pi_1(W)=H$.
So  the quotient $W\ba H'$ is compact.
This implies that
 $ W$  contains a maximal real-split connected solvable subgroup, say, $P$ of $H'$. 
Now $H\cap W$  is a normal subgroup of $H$, as $\pi_1(W)=H$.
Since $P< H\cap W$ and $H$ is simple, we conclude  that $H\cap W=H$, i.e., $H<W$. Hence $W=H \pi_2(W)$.
As any compact linear group is algebraic, $C$ is algebraic and hence
$C=\pi_2(W)=\cl{\pi_2(\Ga')}$. Therefore $W=HC$, finishing the proof. \end{proof}

\subsection*{Global thickness of the return time to $\RFM$}
We recall the various notions of thick subsets of $\br$,
following \cite{MMO1} and \cite{MMO3}. 
\begin{Def} \label{thickdef}  Fix $k>1$. 
\begin{itemize}
\item
 A closed subset $\mathsf T \subset \br$ is locally $k$-thick at $t$
if for any $\la>0$,
$$\mathsf T \cap \left(t \pm   [\la, k\la]  \right) \ne \emptyset .$$

\item  A closed subset $\mathsf T \subset \br$ is $k$-thick 
if $\mathsf T$ is locally $k$-thick at $0$.

\item  A closed subset $\mathsf T \subset \br$ is $k$-thick at $\infty$ if
$$\mathsf T\cap \left(\pm   [\la, k\la]  \right) \ne \emptyset$$
for all sufficiently large $\la \gg 1$.

\item A closed subset $\mathsf T \subset \br$ is globally $k$-thick 
if $\mathsf T \ne \emptyset$ and $\mathsf T$ is locally $k$-thick at every $t\in \mathsf T$.
\end{itemize}
\end{Def}

We will frequently use the fact that if $\mathsf T_i$ is a sequence of $k$-thick subsets,
then $\limsup \mathsf T_i$ is also $k$-thick, and that if $\mathsf T$ is $k$-thick, so is $-\mathsf T$.

The following proposition shows that $\op{RF}\M$ has a thick return property under the action of
any one-dimensional subgroup $U$ of $N$.

\begin{proposition}\label{lem.thickreturntime}\label{defk}
 There exists a constant $k>1$  depending only on the systole of the double of $\core {\M}$ such that 
 for any one-parameter subgroup $U=\{u_t: t\in \br \}$ of $N^{\pm}$, and any $y\in \RFM$,
 $$\mathsf T(y)\coloneqq\{t\in \bb{R} : yu_t\in\op{RF}\M\}$$ is globally $k$-thick. 
\end{proposition}
\begin{proof} Let $ \eta>0$ be the systole
of  the hyperbolic double of $\core \M$, which is a closed hyperbolic manifold.
Let  $k>1$ be given by the equation 
\begin{equation}\label{koh}
 d \big(\mathrm{hull}([-k,-1]),\text{ }\mathrm{hull}([1,k])\big)=\eta /4
\end{equation}
where $d$ is the hyperbolic distance in the upper half plane $\bH^2$.

Note that
\begin{equation}\label{min}
\inf\limits_{i\neq j} d (\mathrm{hull}B_i,\text{ }\mathrm{hull}B_j)\ge \eta /2
\end{equation}
as the geodesic realizing this distance is either a closed geodesic or half of a closed geodesic in the double of $\core {\M}$.

We first prove the case when $U<N$. 
Let $s\in \mathsf T(y)$ be arbitrary.
To show that $\mathsf T(y)$ is locally $k$-thick at $s$, we may assume that $s=0$, by replacing $y$ with $yu_s\in\RFM$.
We may  also assume that $y=[g]$ where $g(\infty) =\infty$ and $g(0)=0$. As $y\in \RFM$, 
this implies that $0, \infty\in \La$.
Since $gu_t(\infty)=g(\infty)\in \La$, we have
\begin{equation*}
\mathsf T(y)=\{t\in\bb{R} : gu_t(0)\in\La\}.
\end{equation*}

Suppose that $\mathsf T(y)$ is not locally $k$-thick at $0$. Then there exist $w\in U$ and $t>0$ such that 
\begin{equation*}
([-kt,-t]\cdot w \cup [t,kt]\cdot w)\cap\Lambda=\emptyset.
\end{equation*}
Since each component of $\Omega$ is convex and $0\not\in\Omega$, it follows that $[-kt,-t]\cdot w$ and $[t,kt]\cdot w$  lie in distinct components of $\Omega$, say $B_i$ and $B_j$, $(i\neq j)$.
But this yields
\be\label{dw}  d_w(\mathrm{hull}([-kt,-t]\cdot w),\text{ }\mathrm{hull}([t,kt]\cdot w))
\geq d(\mathrm{hull} B_i , \text{ }\mathrm{hull}B_j)
\geq\eta/2 .\ee
where $d_w$ denotes the hyperbolic distance of the plane above the line
$\br w$. Observe that the distance in \eqref{dw} 
is independent of $w\in\bb{R}^{d-1}$ and $t>0$, because both the
dilation centered at $0$ and the $(d-2)$-dimensional rotation with respect to the vertical axis above $0$ are hyperbolic isometries.
Therefore, we get a contradiction to \eqref{koh}. The case of $U<N^+$ is proved similarly, just replacing the role of $g^+$ and $g^-$ in the above arguments.
\end{proof}

\begin{Rmk}
It follows from the proof that $k$ is explicitly given by \eqref{koh}, equivalently,
$k+k^{-1}=e^{\eta/4}+2e^{\eta/8}-1$ where $\eta>0$ is the systole of the double of $\core {\M}$.
\end{Rmk}

\section{Structure of singular sets}\label{sec.correspondence}\label{s:st}

Let $\Gamma <G=\SO^\circ(d,1)$ be a convex cocompact torsion-free Zariski-dense subgroup.
Let $U<G$ be a connected closed subgroup of $G$ generated by unipotent elements in it.
In this section, we define the singular set $\mS(U)$ associated to $U$ and
study its structural property. The singular set $\mS(U)$ is defined
so that it contains all closed orbits of intermediate 
subgroups between $U$ and $G$.

\begin{definition}[Singular set]\label{sing}
We set \begin{equation*}
\mS(U)=\left\{x\in \Gamma\ba G : 
\begin{array}{c}
\text{there exists a proper connected }\\ \text {closed subgroup }W \supset U
     \text{ such that }xW \\ \text{ is closed and $\op{Stab}_W(x)$ is Zariski dense in
$W$.}
\end{array}
\right\}.
\end{equation*}
\end{definition}

\begin{Def} [Definition of $\mH$]
We denote by $\mH$  the collection of all  {\it proper} connected closed subgroups $H <G$ containing a unipotent element
such that
\begin{itemize}
\item  $ \Ga\ba\Ga H$  is closed, and 
\item $H\cap \Gamma$ is Zariski dense in $H$.
\end{itemize}

\end{Def}

\begin{prop}\label{sin1}
If  $H\in \mH$, then
 $H$ is a reductive subgroup of $G$, and hence
is of the form $ gH(U)Cg^{-1}$
 for some connected subgroup $U<N$, a closed subgroup $C<\op{C}(H(U))$ and $g\in G$
 such that $[g]\in \RFM$.
 
\end{prop}
\begin{proof} In order to prove that $H$ is reductive, suppose not. Then
its unipotent radical is non-trivial, which we can assume to be a subgroup $U$ of $N$, up to a conjugation.
Now we write $H=H_{nc} CT U$ where $C$ is a connected semisimple compact subgroup and $T$
is a  torus centralizing $H_{nc} C$. As $H$ is contained in $\op{N}(U)=NA \op{C_1}(U) \op{C_2}(U)$, which does not contain any
non-compact simple Lie subgroup,
it follows that $H_{nc}$ is trivial.
Now if $T$ were compact, then $H\cap \Gamma$ would consist of parabolic elements, which is a contradiction as $\Gamma$ is convex cocompact.
Hence $T$ is non-compact. Write $T=T_0S$ where $S$ is a split torus and $T_0$ is compact. Then $T_0$ is equal to a conjugate
of $A$, say, $g^{-1}A g$ for some $g\in G$.  As $T_0$ normalizes $U$, and $\op{N}(U)$ fixes $\infty$, we deduce that $g(\infty)$ is either $\infty$ or $0$.
Since $\op{Stab}_G(\infty)=NAM$,
$g(\infty)=\infty$ implies $g\in NAM$, and  $g(\infty)=0$ implies $jg\in NAM$ where $j\in G$ is an element of order $2$ such that $j(0)=\infty$.
In either case, $T_0=v^{-1}Av$ for some $v\in N$.
By replacing $H$ with $vHv^{-1}$, we may assume that $T_0=A$. Since $CS$ is a compact subgroup commuting with $A$,
$CS\subset M$. Therefore $H$ is of the form $M_0AU$ where $M_0$ is a closed subgroup of $M\cap \op{N}(U)$; note that we used the fact that
 $v$ commutes with $U$.
Now the commutator subgroup $[H,H]$ is equal to $[M_0,M_0]U$. Since $[H\cap \Gamma, H\cap \Gamma]$ must be Zariski dense
in $[H, H]$, we deduce that $\Gamma$ contains an element $m_0 u\in M_0 U$ with $u$ non-trivial.
Since $m_0u$ is a parabolic element of $\Gamma$, this is a contradiction to the assumption that $\Gamma$ is convex cocompact.
This proves that
 $H$ is reductive. 
 
 By Proposition \ref{str}, $H$ is of the form $ gH(U)Cg^{-1}$ for some $g\in G$ and $C<\op{C}(H(U))$.
 For some $m\in M$ and $1\le k\le d-2$, $H(U)=m H(U_k) m^{-1}$.
Hence $\Gamma\ba \Gamma gm H(U_k) C_0$ is closed where $C_0= m^{-1} C m$.
By Proposition \ref{proper},  the boundary of the geodesic plane
$\pi(gm H(U_k))$ contains uncountably many points of $\Lambda$, since $(gm) H(U_k) C_0(gm)^{-1}\cap \Gamma $ is Zariski dense in $(gm)H(U_k)C_0(gm)^{-1}$.
Using two such limit points,
we can find an element $h\in H(U_k)$ such that $(gmh)^{\pm}\in \La$.
Since $(gmhm^{-1})^{\pm}=(gmh)^{\pm}$ and $mhm^{-1}\in H(U)$,
it follows that $[g] H(U)\cap \RFM\ne \emptyset$, and hence we can take $[g]\in \RFM$
by modifying it with an element of $H(U)$ if necessary. This finishes the proof. \end{proof}

Therefore, for each $H\in \mH$,
the non-compact semisimple part $H_{nc}$ of $H$ is well defined.

\begin{prop}\label{sinc}
If $H\in\mathscr{H}$, then 
\begin{itemize}
\item  $H\cap\Ga$ is finitely generated;
\item $[\op{N}_G(H_{nc})\cap \Gamma; H\cap \Gamma]<\infty$.
\end{itemize}
\end{prop}
\begin{proof}
Let $p$ denote  the projection map $\op{N}_G(H_{nc})\to H_{nc}$.
Note that $p$ is an injective map on $\op{N}_G(H_{nc})\cap \Gamma$, as $\Ga$ is torsion free and the kernel of $p$ is a compact subgroup.
It follows from Proposition \ref{sin1} that $H_{nc}$ is co-compact in $\op{N}_G(H_{nc})$.
Since $H\in\mH$, the orbit $[e]H$ is closed and hence $[e]\op{N}_G(H_{nc})$ is closed.
 It follows that both $p(H\cap \Ga)$ and
 $p(\op{N}_G(H_{nc})\cap\Ga)$ are convex cocompact Zariski dense subgroups of $H_{nc}$ by Proposition \ref{proper}.
 As any convex cocompact subgroup is finitely generated \cite{Bow},
  $p(H\cap\Ga)$ is finitely generated. Hence $H\cap\Ga$ is finitely generated by the injectivity of $p|_{H\cap \Gamma}$.
  
 Since $p(H\cap \Ga)$ is a normal subgroup of $p(\op{N}_G(H_{nc})\cap\Ga)$, it follows
 that $p(H\cap \Ga)$ has finite index in $p(\op{N}_G(H_{nc})\cap\Ga)$ by Lemma \ref{finite} below.
 Since $p|_{\op{N}_G(H_{nc})\cap\Ga}$ is injective, it follows that $H\cap \Gamma$ has finite index in $ \op{N}_G(H_{nc})\cap \Gamma$.  \end{proof}

\begin{lem}\label{finite}
Let $\Gamma_1$ and $ \Gamma_2$ be non-elementary convex cocompact subgroups of $G$.
If $\Gamma_2$ is a normal subgroup of $\Gamma_1$, then $[\Gamma_1: \Gamma_2]<\infty$.
\end{lem}
\begin{proof} Let $\Lambda_i$ be the limit set of $\Gamma_i$ for $i=1,2$.
Since $\Ga_2<\Ga_1$, $\Lambda_2\subset \Lambda_1$.  As $\Gamma_2$ is normalized by $\Gamma_1$,
$\Lambda_2$ is $\Gamma_1$-invariant. Since $\Gamma_1$ is non-elementary,
$\Lambda_1$ is a minimal $\Gamma_1$-invariant closed subset. Hence $\Lambda_1=\La_2$.
Let $\M_i:=\Gamma_i\ba \bH^d$. Then the convex core  of $\M_1$ is equal to $\Gamma_1\ba \hull (\La_2)$ and covered by
$\core {\M}_2=\Gamma_2\ba \hull (\La_2)$. Since $\core {\M}_2$ is compact, it follows
that $[\Gamma_1: \Gamma_2]<\infty$.
\end{proof}

\begin{Def} [Definition of $\mH^\star$]
\be\label{hstar}\mH^\star:=\{ \op{N}_G(H_{nc}): H\in \mH\}.\ee
\end{Def}

\begin{corollary}[Countability] \label{counth}
The collection $\mathscr{H}$ is countable,
and the map $H\to \op{N}_G(H_{nc})$ defines a bijection between $\mH$ and $\mH^{\star}$.
\end{corollary}
\begin{proof} As $\Gamma$ is convex cocompact, it is finitely generated.
Therefore there are only countably many finitely generated subgroups of $\Ga$. By Proposition \ref{sinc},
there are only countably many possible $H\cap \Gamma$ for $H\in \mathscr{H}$. Since $H$ is determined by $H\cap \Gamma$,
being its Zariski closure, the first claim follows.

Since $H\cap \Gamma$ has finite index in $\op{N}_G(H_{nc})\cap \Gamma$ by Proposition \ref{sinc}, $H$ is determined as the identity component of the Zariski closure of
$\op{N}_G(H_{nc})\cap \Gamma$. This proves the second claim.
\end{proof}

 In the case of a  convex cocompact hyperbolic manifold of Fuchsian ends, there is a one to one correspondence between
$\mathscr{H}$ and the collection of all closed $H'(U)$-orbits of points in $\RFM$ for $ U<N$:
 if $H\in \mathscr H$,
then $H=gH(U) C g^{-1}$ for some $U<N$ and  $g\in G$ with
$[g]\in \RFM$ and $[g]H'(U)$ is closed.
Conversely, if $[g]H'(U)$ is closed for some $[g]\in \RFM$,
then the identity component of
the  Zariski closure of $\Gamma\cap gH'(U)g^{-1}$ is given by $ gH(U) Cg^{-1}$ for some
closed subgroup $C< \op{C}(H( U))$ by Proposition \ref{count},
and hence $gH(U) Cg^{-1}\in \mathscr H$. Moreover, since the normalizer of $H(U)C$ is contained in $H'(U)$,
 if $g_1H(U) Cg_1^{-1}
=g_2H( U) Cg_2^{-1}$, then $g_2^{-1}g_1 \in H'(U)$, so
$[g_1]H'(U)= [g_2] H'(U)$.
Therefore Corollary \ref{counth} implies
the following corollary  by Propositions \ref{proper} and \ref{count}.

\begin{cor}\label{rcount} Let $\M$ be a  convex cocompact hyperbolic manifold with Fuchsian ends.
Then \begin{enumerate}
\item  there are only countably many properly immersed geodesic
planes of dimension at least $2$ intersecting $\core {\M}$.
\item For each $1\le m\le d-2$, there are only countably many spheres $S\subset \mathbb S^{d-1}$ of dimension $m$,
such that $\# S\cap \La \ge 2$ and $\Gamma S$ is closed in the space $\mathcal C^m$.
\end{enumerate}
\end{cor}

\begin{Rmk}\label{Qcount}
In (2), we may replace the condition $\# S\cap \La \ge 2$ with $\# S\cap \La \ge 1$, because
if  $\#S\cap \La=1$, then $\Gamma S$ is not closed (see Remark \ref{NCL}).
\end{Rmk}

For a subgroup $H<G$, define
\begin{equation}\label{xhuttt}
X(H,U):=\{g\in G : gUg^{-1}\subset H\}.
\end{equation}

Note that $X(H,U)$ is left-$\op{N}_G(H)$ and right-$\op{N}_G(U)$-invariant, and for any $g\in G$,
\begin{equation}\label{eq.v5}
X(gHg^{-1},U)=gX(H,U) .
\end{equation}

For $H\in \mH$ and any connected unipotent subgroup $U<G$, observe that
\be\label{xhnc} X(H, U)=X(H_{nc}, U)=X(\NC, U);\ee
this follows since any unipotent element of $\NC$ is contained in $H_{nc}$.

\begin{proposition}\label{sudef}
We have
\begin{equation*}
\mathscr{S}(U )=\bigcup\limits_{H\in\mathscr{H}^\star}\Ga\ba\Ga X(H,U).
\end{equation*}
\end{proposition}

\begin{proof}
 If $x=[ g]\in\mathscr{S}(U)$, then there exists a proper 
 connected closed subgroup $W$ of $G$ containing $U$ such that $[g] W$ is closed and $\op{Stab}_W(x)$ is Zariski
 dense in $W$.
 This means $H:=gWg^{-1}\in \mathscr{H}$ and $g\in X(H, U)$. Since $X(H, U)= X(\op{N}_G(H_{nc}), U)$, and $\NC\in \mH^{\star}$,
   this proves the inclusion $\subset$. Conversely, let $g\in X(\NC,U)$ for some $H\in\mathscr{H}$.
Set $W:=g^{-1}H g$. Then
$U\subset W$, $[g] W=\Gamma H g$ is closed and $\op{Stab}_W([g])= g^{-1}(\Gamma\cap H) g$ is Zariski dense in $W$.
Hence  $[g]\in\mathscr{S}(U)$.
\end{proof}

\subsection*{Singular subset of a closed orbit}
Let $L<G$ be a connected reductive subgroup of $G$ containing unipotent elements.
For a closed orbit $x_0L$  of $x_0\in \RFM$, and a connected subgroup $U_0<L\cap N$,
  we define  the singular set $ \mathscr{S}(U_0, x_0L)$ by
\begin{equation}\label{defslu}
\mS(U_0, x_0L)=\left\{x \in x_0L : 
\begin{array}{c}
\text{there exists a connected closed } \text {subgroup $W <L$},\\ \text{containing $U_0$}
     \text{ such that $\op{dim} W_{nc}<\op{dim} L_{nc}$, }\\x W \text{ is closed and $\op{Stab}_W(x )$ is Zariski dense in
$W$}
\end{array}
\right\}.
\end{equation}

It follows  from Proposition \ref{sudef} and Proposition \ref{sin1}
 that the subgroup $W$ in  the definition \ref{sing} is conjugate to $H(\widehat U)C$ for some $\widehat U<N$. Hence $W$ being a proper subgroup of $G$
is same as requiring $\op{dim}W_{nc}<\op{dim} G$. Therefore
 $\mS(U_0)=\mS(U_0, \Gamma\ba G)$ and
$$\mS (U_0, x_0 L)=x_0L\cap \bigcup
\Gamma\ba \Gamma   X(H, U_0) $$
where 
the union is taken over all subgroups $H\in\mathscr{H}^{\star}$ such that 
$H$ is a subgroup of $g_0Lg_0^{-1}$ with $\op{dim} H_{nc}<\op{dim} L_{nc}$ and $x_0=[g_0]$.
Equivalently,
\be\label{hxo}\mS (U_0, x_0 L)= \bigcup
_{W\in \mH^{\star}_{x_0L}}
x_0  (L\cap X(W, U_0)) \ee
where 
$ \mH^{\star}_{x_0L}$ consists of all subgroups of the form $ W=g_0^{-1} Hg_0\cap L$ for some $H\in \mH^\star$ and $\op{dim} W_{nc}<\op{dim} L_{nc}$.
Then the generic set $\mG(U_0, x_0L)$ is defined by
\begin{equation}\label{gxol} \mG(U_0, x_0L):= (x_0L\cap \RFPM) -\mS(U_0, x_0L).\end{equation}

\subsection*{Definition of $\mathcal L_U$ and $\mathcal Q_U$}
Fix  a non-trivial connected closed subgroup $U<N$.
 We define the collection
 $\mathcal L_U$ of all subgroups of the form $H(\widehat U) C$
 where $U <\widehat U<N$ and $C$ is a closed subgroup of $\op{C}(H(\widehat U))$ satisfying the following:
\begin{equation}\label{deflu}
\mathcal L_U:=\left\{L=H(\widehat U) C : 
\begin{array}{c}
\text{for some $[g]\in \RFPM$, $[g]L $ is closed in $\Gamma\ba G$  }\\  \text{ and $L\cap g^{-1}\Gamma g$ is Zariski
dense in $L$}
\end{array}
\right\}.
\end{equation}

Observe that for $L=H(\widehat U) C\ne G$,
the condition $L\in\cal{L}_U$ with $[g]L$ closed is equivalent to the condition that 
$$ gLg^{-1}\in\mathscr{H}.$$ 

\begin{lem}\label{sin3} Let $L_1$ and $ L_2$ be members of $ \mathcal L_U$ such that
 $xL_1$ and $xL_2$ are closed for some $x\in \RFM$.
If $(L_1)_{nc}=(L_2)_{nc}$, then $L_1=L_2$.
\end{lem}
\begin{proof} If $L_1$ or $L_2$ is equal to $G$, then the claim is trivial. Suppose that both $L_1$ and $L_2$ are proper subgroups of $G$.
If $x=[g]$, then both subgroups $H_1:=gL_1g^{-1}$ and $H_2:=gL_2g^{-1}$ belong to $ \mH$. Since 
 $(H_1)_{nc}=(H_2)_{nc}$, we have $H_1=H_2$ by Corollary \ref{counth}. Hence $L_1=L_2$.
\end{proof} 

We also define 
\begin{equation}\label{qu1}\cal{Q}_U:=\{vLv^{-1} :  L\in\cal{L}_U, v\in \op{N}(U) \}.\end{equation}
Since $\op{N}(U)=AN \op{C}_1(U) \op{C}_2(U)$ by Lemma \ref{nu}, and  the collection $\cal L_U$
is invariant under a conjugation by an element of
$AU\op{C}_1(U) \op{C}_2(U)$, we have
\be \label{qu} \cal{Q}_U=\{vLv^{-1} :  L\in\cal{L}_U, v\in U^\perp \}.\ee

\begin{lem}\label{xhuo}  For $U_0<U<N$, we have
$$X(H(U), U_0) =  \op{N}_G(H(U))\op{N}_G(U_0).$$
\end{lem}

\begin{proof}
 Without loss of generality, we may assume $U=U_m$ and $U_0=U_\ell$
with $1\le \ell \le m\le d-1$. Set $H=H(U_m)$. If $m=d-1$, then $H=G$, and the statement is trivial. Assume $m\le d-2$ below.
We will prove the inclusion $X(H, U_0) \subset\op{N}_G(H)\op{N}_G(U_0)$, as the other one is clear.
Let $g\in X(H, U_0)$ be arbitrary. By  multiplying $g$
by an element of $\op{N}_G(H)$ on the left as well as by an element of $\op{N}_G(U_0)$ on the right, we will reduce $g$ to an element of
$\op{N}_G(U_0)$, which implies the claim. 
In view of the Iwasawa decomposition $G=KAN$,
since $AN<\op{N}_G(U_0)$, we may assume that $g=k\in K$.
As $k\in X(H, U_0)$, we have 
$kU_0k^{-1} \subset H$. 
Since $K\cap H$ is a maximal compact subgroup of $H$, any maximal horospherical subgroups of $H$ are conjugate to each other by an element of $K\cap H$.
Hence there exists $w\in K\cap H$ such that $kU_0k^{-1}= wU_0w^{-1}$.

Since $w^{-1} k U_0 = U_0 w^{-1}k$, we deduce $w^{-1}k (\infty)= U_0(w^{-1} k(\infty))$. 
Since $\infty\in \mathbb S^{d-1}$ is the unique fixed point of $U_0$, $w^{-1} k(\infty)=\infty$. Hence $w^{-1}k\in K\cap (MAN)=M$.
Since $w\in H$, we may now assume that $k\in M$. From $kU_0\subset Hk$, we get
$kU_0(0)\subset Hk(0)=H(0)$ and hence $\langle ke_1,\cdots, ke_\ell \rangle\subset \langle e_1,\cdots, e_m\rangle.$
By considering the action of $H\cap K$ on space of $\ell$-tuples of orthonormal vectors in
the subspace $\langle e_1,\cdots, e_m\rangle$, we may assume $ke_1=e_1$, $\cdots$, $ke_{\ell-1}=e_{\ell-1}$, and $ke_{\ell}=\pm e_{\ell}$.
This implies that $k\in\op{C}_1(U_0)$, or $k \omega \in \op{C}_1(U_0)$
where $\omega \in M$ is an involution which fixes all $e_i$, $i\ne \ell, \ell+1$ and $\omega (e_i)=-e_i$
for $i=\ell, \ell+1$. As $\op{N}_G(U_0)$ contains $\op{C}_1(U_0)$ and $ \omega$, the proof is complete.
\end{proof}

\begin{prop}\label{explain}  Consider a closed orbit $x_0L$ for $L\in \mathcal Q_{U}$ and $x_0\in \RFM$.
 If $x\in \mS(U_0, x_0L )$ for  a connected closed subgroup $U_0 <U$, then there exists a subgroup $Q\in \mathcal Q_{U_0}$ such that
 \begin{itemize}
 \item  $\op{dim}Q_{nc}<\op{dim} L_{nc}$;
 \item $xQ$ is closed;
\item $ \overline{xU_0}\subset x Q$.
\end{itemize}
\end{prop}
\begin{proof} 
If $x=[g]\in \mS(U_0, x_0L )$, then $g\in X(H, U_0)$ for some $H\in \mH$ such that
 $\op{dim}H_{nc}<\op{dim} L_{nc}$.
Then $\cl{xU_0}\subset x (g^{-1} H g)$.
By Proposition \ref{sin1}, $H=q H(\widehat U) C q^{-1}$ for some $U_0<\widehat U<L\cap N$ and some $[q]\in \RFM$.
 Note that $q^{-1} g\in X(H(\widehat U), U_0)$. By Lemma \ref{xhuo}, we have
 $$q^{-1} g\in \op{N}_G(H(\widehat U)) \op{N}_G(U_0).$$ 
 Hence $g^{-1} H g = v H(\widehat U) Cv^{-1}$ for some $v\in \op{N}_G(U_0)$, and
 $\cl{x U_0} \subset  xv  H(\widehat U) C v^{-1}.$
It suffices to set $Q:=v  H(\widehat U) C v^{-1}$. \end{proof}

\begin{lem}\label{dcc}
Let $L=H(\widehat U) C$ for a connected closed subgroup $\widehat U<N$ and
closed subgroup $C<\op{C}(H(\widehat U))$. Let
$W=g^{-1} H(\til U)C_0g $ be a subgroup of $L$ where $g\in L$, $\til U$ is a proper connected closed subgroup of $\widehat U$ and $C_0$
is a closed subgroup of $H(\til U)$.
Then for any non-trivial closed connected subgroup $U<\widehat U$, $(L\cap X(W, U)) H(U)$ is a nowhere dense subset of $L$.
\end{lem}
\begin{proof} Write $g=hc\in H(\widehat U)C$. 
Note that \begin{align*}L\cap X(W, U)&= L\cap X( g^{-1} H(\til U) g , U)\\
&=L\cap X( h^{-1} H(\til U) h , U) \\
&=h (L\cap X(H(\til U), U))\\
&= h (H(\widehat U) \cap X(H(\til U), U))  C.
\end{align*}
Hence it suffices to show that $ (H(\widehat U) \cap X(H(\til U), U) ) H(U)$ is a nowhere dense subset of $H(\widehat U)$.
Without loss of generality, we may now assume $H(\widehat U)=G$. We observe that  using  Lemma \ref{xhuo},
 \begin{align*} &X (H(\til U), U) H(U) =\op{N}_G(H(\til U)) \op{N}_G(U) H(U)
 \\& = H(\til U) \op{C}_1(\tilde U)
 AN \op{C}_1(U)\op{C}_2(U) H(U)\\
 &=
(K\cap H(\til U)) U^\perp H'(U). \end{align*}

Let $\op{dim} \til U=m$ and $\op{dim}U=k$. 
If $k\geq m$, then $X(W,U)=\emptyset$. Hence we may assume that $1\le k\le m<d-1=\op{dim} N$.
Now, if we view the subset $(K\cap H(\til U)) U^\perp H'(U)/H'(U) $ in the space $\mathcal C^k=G/H'(U)$,
this set is contained in the set of all spheres $C\in \mathcal C^k$ which are tangent to
the $m$-sphere given by $S_0:=(K\cap H(\tilde U))(\infty)$.
Since $m<d-1$, it follows that $X (H(\til U), U) H(U)/H'(U)$ is a nowhere dense subset of $\mathcal C^k$,
and hence $X (H(\til U), U) H(U)$ is a nowhere dense subset of $G$.
\end{proof}
Recall from \eqref{fhu} that $F=\RFPM\cdot H(U)$.
\begin{lemma}\label{lem.SHU}
Let  $x_0\widehat L$ be a closed orbit of $\widehat L\in\cal L_U$ with $x_0\in\RFM$. 
If $U$ is a proper subgroup of $\widehat L\cap N$,
then $\mathscr{S}(U,x_0\widehat L)\cdot H(U)\cap F_{H(U)}$ is a proper subset of $x_0 \widehat L\cap F_{H(U)}$.
\end{lemma}
\begin{proof} Choose $g_0\in G$ so that $x_0=[g_0]$.
Let $p:G\to \Gamma\ba G$ be the canonical projection map.
Then $p^{-1}( \mathscr{S}(U,x_0\widehat L)\cdot H(U))$ is a countable union
 $\gamma g_0(\widehat L\cap X(W, U)) H(U)$ where $\gamma \in \Gamma$ and $W \in \mH_{x_0\widehat L}^\star$ by \eqref{hxo}.
Hence by Lemma \ref{dcc},
$ \mathscr{S}(U,x_0\widehat L)\cdot H(U)$ is a countable union
 of nowhere dense subsets of $x_0L$.
Since $F^*_{H(U) }\cap x_0\widehat L$ is an open subset of $x_0\widehat L$,
it follows from the Baire category theorem that 
$F^*_{H(U)}\cap x_0\widehat L\not\subset \mathscr{S}(U,x_0\widehat L)\cdot H(U)$.
This proves the claim. \end{proof}

The following geometric property of a convex cocompact hyperbolic manifold with Fuchsian ends is one of its key features which is needed in the proof of our main theorems stated in the introduction.

\begin{prop}\label{SAVE}  Let $\M$ be a  convex cocompact hyperbolic manifold with Fuchsian ends.
Let  $x_0\widehat L$ be a closed orbit of $\widehat L\in\cal L_U$ with $x_0\in\RFM$ and with $\op{dim}(\widehat L\cap N)\ge 2$.  Either $x_0\widehat L$ is compact or $ \mathscr{S}(U,x_0\widehat L)$ contains
a compact orbit $zL_0$ with $L_0\in \mathcal L_U$. \end{prop}
\begin{proof}
Write $\widehat L=H(\widehat U) C$ for a connected closed subgroup $U<\widehat U<N$. 
Since $x_0\widehat L$ is closed, $\pi(x_0\widehat L)=\pi(x_0H'(\widehat U))$ is a properly immersed convex cocompact geodesic plane
of dimension at least $3$ with Fuchsian ends
by Proposition \ref{PPP}. 
Suppose that $x_0L$ is not compact. Then $\pi(x_0L)$ has non-empty Fuchian ends.
This means that there exist a co-dimension one subgroup $U_0$ of $\widehat U$ and $z\in \widehat L$
 such that $zH'(U_0)$ is compact and $\pi(zH'(U_0))$ is a component of the core of $\pi(x_0\widehat L)$.
By Proposition \ref{count}, there
 exists a closed subgroup $C_0<\op{C}(H(U_0))\cap \widehat L$ such that $H(U_0)C_0\in \mathcal L_{U_0}$ and
$zH(U_0)C_0$ is compact. Let $m\in M\cap \widehat L$ be an element
such that $U\subset m^{-1}U_0m$. Then $zm (m^{-1} H(U_0)C_0 m)$
is a  compact orbit contained in $  \mathscr{S}(U,x_0\widehat L)$ and $m^{-1} H(U_0)C_0m\in \mathcal L_U$, finishing the proof.
\end{proof}

\section{Inductive search lemma}\label{s:ig}
In this section, we prove a combinatorial lemma \ref{lem.search}, which we call an {\it inductive search lemma}, and
use it to prove Proposition \ref{prop.2kthick} on the thickness of  a certain subset of $\br$, constructed
by the intersection of a global thick subset $\mathsf T$ and finite families of triples of subsets of $\br$ with controlled regularity, degree and the multiplicity
with respect to $\mathsf T$. This proposition will be used in the proof of the avoidance theorem \ref{avoid1} in the next section.

\begin{definition}
Let $J^*\subset  I$ be a pair of open subsets of $\mathbb{R}$.
\begin{itemize}
\item The degree of $(I, J^*)$ is defined to be the minimal   $\delta\in \mathbb N\cup\{\infty\}$ such that
for each connected component $I^\circ$ of $I$, the number of connected components of $J^*$ contained in $I^\circ$ is bounded by $\delta$.

\item  For $\beta>0$, the pair $(I,J^*)$ is said to be {\it $\beta$-regular} if for any connected component $I^\circ$ of $I$,
and any component $J^\circ$ of $J^* \cap I^\circ$,
$$J^\circ\pm \beta\cdot   |J^\circ | \subset I^\circ $$
where $|J^\circ|$ denotes  the length of $J^\circ$.
\end{itemize}
\end{definition}

\begin{definition}\label{defreg}
Let $\cal{X}$ be a family of countably many triples  $(I,J^*,J')$ of open subsets of $\br$ such that $I\supset J^*\supset J'$.
\begin{itemize}

\item
Given $\beta>0$ and $\delta\in \mathbb N$, we say that $\cal X$ is \textit{$\beta$-regular of degree $\delta$} if for every triple $(I,J^*,J')\in \cal X$, the pair $(I,J^*)$ is $\beta$-regular with degree at most $\delta$.
\item
Given a subset $\mathsf T\subset \br$, we say that $\cal X$ is of $\mT\textit{-multiplicity free}$ if for any two distinct triples $(I_1,J_1^*,J_1')$ and $(I_2,J_2^*,J_2')$ of $\cal X$, we have 
$$I_1\cap J_2'\cap \mT=\emptyset.$$
\end{itemize}
\end{definition}

For  a family $\mathcal X=\{(I_\la,J_\la^*, J_\la') : \la \in \La\}$, we will use the notation 
$$
I(\cal{X}):=\bigcup\limits_{\la\in\La}I_\la, \quad J^*(\cal{X}):=\bigcup\limits_{\la\in\La}J_\la^* \quad \text{and}\quad J'(\cal{X}):=\bigcup\limits_{\la\in\La}J_\la'.
$$

The goal of this section is to prove:
\begin{prop}[Thickness of $\mT-J'(\cal{X})$]\label{prop.2kthick} 
Given  $ n, k, \delta \in \mathbb N$, there exists a positive number
$\beta_0=\beta_0(n,k,\delta)$ for which the following holds:
let $\mT\subset\mathbb{R}$ be a globally $k$-thick set, and let $\cal{X}_1,\cdots,\cal{X}_\ell$, $\ell \le n$, be
$\beta_0$-regular families  of degree $\delta$ and of $\mT$-multiplicity free.
Let $\cal X= \bigcup_{i=1}^\ell \cal X_i$.
If $0\in \mT -I(\cal{X})$,  then 
$$\mT-J'(\cal{X})$$
is a $2k$-thick set.
\end{prop}

We prove Proposition \ref{prop.2kthick} using the inductive search lemma \ref{lem.search}.  The case of $n=1$ and $\delta=1$ is easy.
As the formulation of the lemma  is rather complicated in a general case,
we first explain  a simpler case of $n=2$ and $\delta=1$ in order to motivate the statement.

For simplicity, let us show that $\mT-(J'(\cal{X}_1)\cup J'(\cal{X}_2))$ is $4k$-thick instead of $2k$-thick, given that $\cal{X}_1$ and $\cal{X}_2$ are $8k^2$-regular families of degree $1$, and of $\mT$-multiplicity free. For any $r>0$, we need to find a point
\begin{equation*}
t\in\pm(r,4kr)\cap \Big( \mT- J'(\cal{X})\Big) 
\end{equation*}
where $\cal{X}=\cal X_1\cup \cal X_2$.

\begin{figure}[h]
\centering
        \includegraphics[totalheight=4cm]{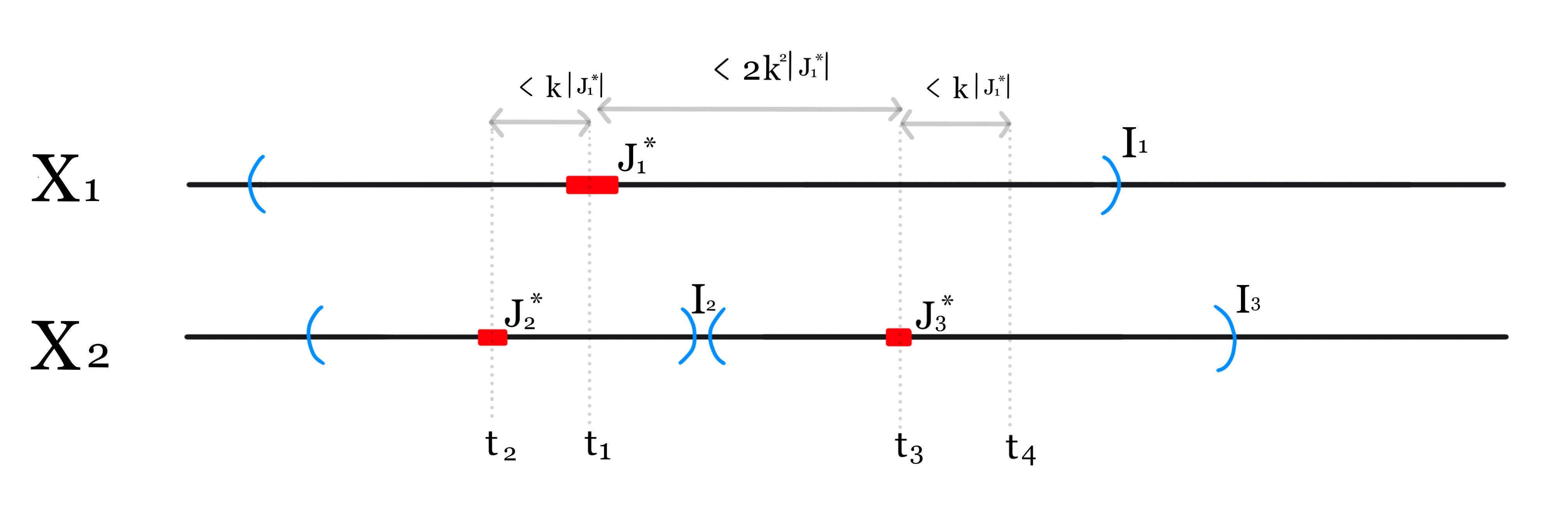}
\label{fig.t4} \label{fig.t3}
\label{fig.t2}
\end{figure}

First, we know that there exists $t_1\in\pm(2r,2kr)\cap \mT$, as $\mT$ is locally $k$-thick at $0$.
If $t_1\not\in J'(\cal{X}_1)\cup J'(\cal{X}_2)$, then we are done.
So we assume that $t_1\in J'(\cal{X}_1)$. Our strategy is then to search for a sequence in $\mT$ of length at most $4$, starting with $t_1$, say $(t_1,t_2,t_3,t_4)$ such that $$\tfrac{|t_{i-1}|}{\sqrt[3]{2}} \leq|t_i|\leq\sqrt[3]{2}|t_{i-1}|\quad\text{ for each $i=2,3,4$,}$$
 and the last element $t_4$ does not belong to $J'(\cal{X})$.
This will imply $|t_1|/2\leq|t_4|\leq 2|t_1|$ and hence 
$$t:=t_4\in\pm(r,4kr)\cap \Big(\mT-J'(\cal{X})\Big)$$ as desired, because $2r\leq|t_1|\leq 2kr$.

We next sketch how we find $t_2$ from $t_1$ and so on.
Let $t_1\in J_1'$ where $(I_1,J_1^*,J_1')\in \cal{X}_1$.
Since $\mT$ is locally $k$-thick at $t_1$, there exists
\begin{equation}\label{eq.t1}
t_2\in ( t_1 \pm(|J_1^*|,k|J_1^*|) ) \cap \mT.
\end{equation}
We will refer to $t_1$ as a pivot for searching $t_2$ in \eqref{eq.t1}, as $t_2$ was found in a symmetric interval around $t_1$.
Note that $t_2\in I_1 -J_1^*$ as $(I_1,J_1^*)$ is $k$-regular.
This implies that $t_2\not\in J'(\cal{X}_1)$ as the family $\cal X_1$ is of $\mT$-multiplicity free.
Now we will assume $t_2\in J_2'$ for some triple $(I_2,J_2^*, J_2')\in \cal{X}_2$, since otherwise, $t_2\not\in J'(\cal{X})$ and we are done.

To search for the next point $t_3\in \mT$, we choose our pivot
 between two candidates $t_1$ and $t_2$ as follows: we will choose $t_1$ if $|J_1^*|\geq |J_2^*|$, and $t_2$ otherwise.
Without loss of generality, we will assume $|J_1^*|\geq |J_2^*|$.
Since  $\mT$ is locally $k$-thick at $t_1$, we can find
\begin{equation*}
t_3\in (t_1 \pm 2k(|J_1^*|,k|J_1^*|) ) \cap \mT.
\end{equation*}

Note that $t_3\in I_1-J_1^*$ as the pair $(I_1,J_1^*)$ is $2k^2$-regular.
This implies $t_3\not\in J'(\cal{X}_1)$ as $\cal{X}_1$ is of $\mT$-multiplicity free.
Now we can assume that $t_3\in J_3'$ for some $(I_3, J_3^*, J_3')\in \cal{X}_2$, otherwise we are done.
One can check that $J_3^*$ cannot coincide with $J_2^*$.
We claim that $|J_1^*|\geq |J_3^*|$. Suppose not, i.e. $|J_3^*|>|J_1^*|$. Then we would have $|t_2-t_1|<k|J_3^*|$ and $|t_1-t_3|<2k^2|J_3^*|$, which implies that $t_2\in I_3$,  as 
the pair $(I_3,J_3^*)$ is $(2k^2+k)$-regular.
This is a contradiction as $\cal X_2$ is $\mT$-multiplicity free and hence $J_2'\cap I_3\cap \mT=\emptyset$.

Finally, we will choose $t_3$ as a pivot and search for $t_4$.
By the local $k$-thickness of $\mT$ at  $t_3$, we can find
\begin{equation*}
t_4\in ( t_3 \pm  (|J_3^*|,k|J_3^*|) )\cap \mT .
\end{equation*}
Since the pair $(I_3,J_3^*)$ is $k$-regular, we have $t_4\in I_3-J_3^*$.
From the fact that the pair $(I_1,J_1^*)$ is $(2k^2+k)$-regular, one can check that $t_4 \in I_1-J_1^*$.
As a result, $t_4 \in (I_1- J_1^*)\cup (I_3-J_3^*)$ and hence $t_4\not\in J'(\cal{X})$.

It remains to check that $|t_{i-1}|/\sqrt[3]{2}\leq|t_i|\leq\sqrt[3]{2}|t_{i-1}|$ for each $i=2,3,4$.
This does not necessarily hold for the current sequence, but will hold after passing to a subsequence where $t_{i-1}$ becomes a pivot for searching $t_i$ for all $i$.
In the previous case, $(t_1,t_3,t_4)$ will be such a subsequence, as $t_2$ was not a pivot for searching $t_3$.

It follows from the $\beta:=8k^2$-regularity of $(I_{i-1},J_{i-1}^*)$ that $|t_{i-1}|- 8k^2|J_{i-1}^*|>0$,  as $t_{i-1}\in J_{i-1}^*$ and $0\not\in I_{i-1}$.
On the other hand, observe that 
\begin{equation*}
t_i\in t_{i-1} \pm C_i(|J_{i-1}^*|,k|J_{i-1}^*|)\cap \mT
\end{equation*}
for some $C_i\leq 2k^2$.
This gives us the desired upper bound for $|t_i/t_{i-1}|$, as
\begin{equation*}
|t_i|<|t_{i-1}|+C_i|J_{i-1}^*|\leq (1+C_i(8k^2)^{-1} )|t_{i-1}|
\end{equation*}
and $1+C_i (8k^2)^{-1} \leq\sqrt[3]{2}$.
The lower bound is obtained similarly, completing the proof for $n=2$ and $\delta=1$.

The general case reduces to the case of $\delta=1$, by replacing $n$ by $n\delta$. Roughly speaking,  the following lemma gives 
an inductive argument for the search of a sequence of $t_i$'s
which is almost geometric in a sense that the ratio $|t_i|/|t_{i-1}|$ is coarsely a constant and which  lands on $\mT-J'(\cal X)$ in a time controlled by $n$.

\begin{lemma}[Inductive search lemma]\label{lem.search}
Let $k>1$, $n\in \mathbb N$ and  $0<\e<1$ be fixed.
There exists $\beta=\beta(n, k, \epsilon)>0$ for which the following holds:
Let $\mT\subset\mathbb{R}$ be a globally $k$-thick set, and
let  $\cal{X}_1,\cdots,\cal{X}_n$ be $\beta$-regular
 families of countably many triples $(I_\la, J^*_\la, J'_\la)$  with degree $1$, and of $\mT$-multiplicity free. 
Set  $\cal X=\cal X_1\cup\cdots\cup\cal X_n$, and assume $0\not\in I(\cal{X})$.
For any $t\in \mT\cap J'(\cal{X})$ and any $1\leq r\leq n$, we can find 
{distinct} triples
 $(I_1,J_1^*,J_1'),\cdots,(I_{m-1},J_{m-1}^*,J_{m-1}')\in \cal X$ with $2\leq m\leq 2^r$,
and a sequence of pivots
$$t=t_1\in \mT\cap J_1',  \;  t_2 \in \mT\cap J_2' ,\cdots, t_{m-1}\in \mT\cap J_{m-1}', \; t_m\in \mT$$
which satisfy  the following conditions:
\begin{enumerate}
\item either $t_m\not\in J'(\cal{X})$, or 
$t_m\in J_m'$ for some $(I_m,J_m^*,J_m')\in \cal X$, which is
{distinct from $(I_i, J_i^*, J_i')$ for all $1 \le i\leq m-1$,} and 
the collection $\{(I_i,J_i^*,J_i'):1\leq i\leq m\}$ intersects at least $(r+1)$ number of $\cal{X}_i$'s;
\item for all $1\leq i\leq j\leq m$,
$$|t_i-t_j|\leq 2((4k)^r-1)k\max\limits_{1\leq p\leq j-1}|J_p^*|;$$
\item
for each $1\leq i\leq m$, 
$$(1-\epsilon)^{i-1}|t_1|\leq|t_i|\leq(1+{\epsilon})^{i-1}|t_1|.$$
\end{enumerate}
In particular, for any $t\in \mT\cap J'(\cal{X})$, there exists $t'\in \mT-J'(\cal{X})$ such that
$$(1-\epsilon)^{2^n-1}|t|\leq|t'|\leq(1+{\epsilon})^{2^n-1}|t|.$$
\end{lemma}
\begin{proof}  We set 
\begin{equation}\label{bbbb}
\beta=\beta(n,k,\epsilon)=(4k)^{n+1}\epsilon^{-1}.
\end{equation}
Consider the increasing sequence {$Q(r):=(4k)^r-1$}
for $r\in\mathbb{N}$.
Note that
\begin{equation*}
Q(1)\geq 2\;\; \text{ and } \;\; Q(r+1)\geq 4Q(r)k+1.
\end{equation*}
Moreover we check that
\begin{equation*}
\beta>\max((Q(n)+4Q(n-1))k, Q(n)k\epsilon^{-1}).
\end{equation*}
We proceed by induction on $r$. First consider the case when $r=1$.
 There exists $(I_1,J_1^*,J_1')\in\cal X$ such that $t_1:=t\in J_1'\cap \mT$.
As $\mT$ is globally $k$-thick, we can choose
\begin{equation}\label{eq.search1}
t_2\in \big(t_1\pm Q(1)(|J_1^*|,k|J_1^*|)\big)\cap \mT.
\end{equation}
We claim that $t_1, t_2$ is our desired sequence with $m=2$. In the case when $t_2\in J'(\cal{X})$, 
there exists $(I_2,J_2^*,J_2')\in\cal X$ such that $t_2\in J_2'$. We check:

(1): If $t_2\in J'(\cal{X})$, then $t_2\in J_2'-J_1^*$ implies that $J_1^*$ and $J_2^*$ are distinct.
Hence $(I_1, J_1^*, J_1')$ and $(I_2, J_2^*, J_2')$ are distinct as well.
Since $\beta> Q(1)k$, by the $\beta$-regularity of $(I_1,J_1^*)$, we have $t_2\in I_1$.
By the $\mathsf T$-multiplicity free condition, $(I_1,J_1^*,J_1')$ and $ (I_2,J_2^*,J_2') $ don't belong to the same family, that is,
$\{(I_1,J_1^*,J_1'), (I_2,J_2^*,J_2')\}$ intersects two of $\cal X_i$'s.

(2): By \eqref{eq.search1}, $|t_1-t_2|<Q(1)k|J_1^*|={(4k-1)}k|J_1^*| $.

(3): Note that $0\not\in I_1$, since $0\not\in I(\cal{X})$.
By the $\beta$-regularity of $(I_1,J_1^*)$, we have
$t_1\pm \beta |J_1^*|\subset I_1$.
Since $0\not\in I_1$ and $\beta >\epsilon^{-1} Q(1)k$, we have
\begin{equation*}
|t_1|-\epsilon^{-1} Q(1)k|J_1^*|>0.
\end{equation*}
On the other hand, by \eqref{eq.search1},
\begin{equation*}
|t_2-t_1|\leq Q(1)k|J_1^*|\leq \epsilon |t_1|.
\end{equation*}
In particular,
\begin{align*}
|t_2|&\leq|t_1|+|t_2-t_1|<|t_1|+Q(1)k|J_1^*|<\left(1+{\epsilon}\right)|t_1|\text{ and }\\
|t_2|&\geq|t_1|-|t_2-t_1|>|t_1|-Q(1)k|J_1^*|>\left(1-{\epsilon}\right)|t_1|.
\end{align*}
This proves the base case of $r=1$.

Next, assume the induction hypothesis for $r$.
Hence we have a sequence
\begin{equation*}
t_1(=t)\in J_1',\: t_2\in J_2',\: \cdots,\: t_{m-1}\in J_{m-1}',\text{ and } t_m
\end{equation*}
 in $\mT$ with $m\leq 2^r$
 together with $\{(I_i,J_i^*,J_i'):1\leq i\leq m-1\}$
  satisfying the three conditions listed in the lemma.
If $t_m\not\in J'(\cal X)$, the same sequence would satisfy the hypothesis for $r+1$ and we are done.
Now we assume that 
$t_m\in J_m'$ for some $(I_m,J_m^*,J_m')\in \cal X$,
 and that $\{(I_i,J_i^*,J_i'):1\leq i\leq m\}$ intersect at least $(r+1)$ numbers of $\cal{X}_i$'s.
We may assume that they intersect exactly $(r+1)$-number of $\cal{X}_i$'s, which we may label as $\cal{X}_1,\cdots, \cal{X}_{r+1}$, since if they intersect more than $(r+1)$ of them, we are already done.
Choose a largest interval $J_{\ell}^*$  among $J_1^*,\cdots, J_m^*$.
Again using the global $k$-thickness of $\mT$, we can choose
\begin{equation}\label{choices}
s_1\in \big(t_{\ell}\pm Q(r+1)(|J_{\ell}^*|,k|J_{\ell}^*|)\big)\cap \mT.
\end{equation}

First, consider the case when $s_1\not\in J'(\cal{X})$. We will show that the points $t_1,\cdots, t_m,s_1$ give the desired sequence.
Indeed, the condition (1)  is immediate.
For (2), observe that by the induction hypothesis for $r$, we have
\begin{equation*}
|s_1-t_i|\leq |s_1-t_\ell|+|t_\ell-t_i|\leq (Q(r+1)k+2Q(r)k) |J_{\ell}^*|
\end{equation*}
for all $1\leq i\leq m$.
The conclusion follows as $Q(r+1)>2Q(r)$. To show (3), since $\beta >\epsilon^{-1} Q(r+1)k$ and $0\not\in I_{\ell}$, by applying the $\beta$-regularity to the pair $(I_{\ell},J_{\ell}^*)$, we have
\begin{equation*}
|t_{\ell}|-\epsilon^{-1} Q(r+1)k|J_{\ell}^*|>0.
\end{equation*}
It follows that
\begin{align*}
|s_1|\leq |t_{\ell}|+|s_1-t_{\ell}|<|t_{\ell}|+Q(r+1)k|J_{\ell}^*|<\left(1+{\epsilon}\right)|t_{\ell}|\leq \left(1+{\epsilon}\right)^{m}|t_1|;\\
|s_1|\geq |t_{\ell}|-|s_1-t_{\ell}|>|t_{\ell}|-Q(r+1)k|J_{\ell}^*|>\left(1-{\epsilon}\right)|t_{\ell}|\geq \left(1-{\epsilon}\right)^{m}|t_1|.
\end{align*}
This proves (3).

For the rest of the proof, we now assume that $s_1\in J'(\cal{X})$.
Apply the induction hypothesis for $r$ to  $s_1\in \mT\cap J'(\cal{X})$ to obtain a sequence
 $\{(\til I_j,\til J_j^*,\til J_j')\in \cal X:1\leq j\leq m'-1\}$ with $m'\leq 2^r$ and
 \begin{equation*}
s_1\in \tilde{J}_1'\cap \mT,\: s_2\in \tilde{J}_2'\cap \mT,\: \cdots,\: s_{m'-1}\in \tilde{J}_{m'-1}' \cap \mT,\text{ and } s_{m'}\in \mT.
\end{equation*}
Set $q_0$ to be the smallest $1\leq q\leq m'-1$ satisfying
\begin{equation}\label{eq.zz}
\{(\til I_j,\til J_j^*,\til J_j'):1\leq j\leq q\} \not\subset \cal{X}_1\cup\cdots \cup  \cal{X}_{r+1}
\end{equation}
if it exists, and $q_0:=m'$ otherwise.
We claim that the sequence
\begin{equation}\label{st1}
t_1,\cdots,t_m,s_1,\cdots, s_{q_0}
\end{equation}
of length $m+q_0\leq 2^{r+1}$ satisfies the conditions of the lemma for $r+1$.

\textbf{Claim:} We have 
\begin{equation}\label{eq.search4}
|J_{\ell}^*|=\max\limits_{1\leq i \leq m, 1\leq j\leq q_0-1}(|J_i^*|,|\tilde{J}_j^*|).
\end{equation}
Recall that $|J_{\ell}^*|$ was chosen to be maximal among $|J_1^*|,\cdots, |J_m^*|$.
Hence, if the claim does not hold, then we can take $j$ to be the least number such that $|\tilde{J}_j^*|>|J_{\ell}^*|$.
Then by the induction hypothesis for (2),
\begin{align*}
|t_\ell-s_j|&\leq |t_\ell-s_1|+|s_1-s_j|\\
&\leq Q(r+1)k|J_{\ell}^*|+2Q(r)k\max_{1\leq i\leq j-1}|\til J_i^*|\\
&\leq (Q(r+1)+2Q(r))k| J_\ell^*|.
\end{align*}
Now the collection $\{( I_i, J_i^*,J_i'):1\leq i\leq m\}$ intersects $(r+1)$ families $\cal{X}_1,\cdots, \cal{X}_{r+1}$ and $(\til I_j,\til J_j^*,\til{J}_j')$ belongs to one of these families, as $j\leq q_0-1$.
Hence there exists a triple $(I_i,J_i^*,J_i')$ that belongs to the same family as $(\til I_j,\til J_j^*,\til{J}_j')$.
Recall that the induction hypothesis for $t_1,\cdots,t_m$ gives us
\begin{equation*}
|t_{\ell}-t_i|\leq 2Q(r)k|J_{\ell}^*|.
\end{equation*}
Since $\beta >(Q(r+1)+{4}Q(r))k$, we have
\begin{align*}
|t_i-s_j|&\leq|t_i-t_{\ell}|+|t_{\ell}-s_j|\\
&\leq (Q(r+1)+{4}Q(r))k|J_{\ell}^*|\\
&< \beta |\tilde{J}_j^*|.
\end{align*}
Applying the $\beta$-regularity to the pair $(\tilde{I}_j, \tilde{J}_j^*)$, we conclude that 
$$t_i\in \tilde{I}_j\cap J_i'\cap \mT.$$
Since $(\til I_j,\til J_j^*,\til J_j')$ and $(I_i,J_i^*,J_i')$ belong to the same family which is $\mathsf{T}$-multiplicity free, they are equal to each other. This is a contradiction since $|\til J_j^*|>|J_\ell^*|\geq|J_i^*|$, proving the claim \eqref{eq.search4}.

We next prove that $(I_i,J_i^*,J_i')$ and $(\tilde I_j,\til J_j^*,\tilde{J}_j')$ are distinct for all $1\leq i\leq m$ and $1\leq j\leq q_0-1$.
It suffices to check that $J_i^*$ and $\tilde{J}_j^*$ are distinct.
Note that we have
\begin{align*}
\max\limits_{1\leq i,j\leq m}|t_i-t_j|&< 2Q(r)k{|J_{\ell}^*|}\;\;\text{ and }\\
\max\limits_{1\leq i,j\leq q_0}|s_i-s_j|&< 2Q(r)k |{J}_{\ell}^*|
\end{align*} 
by the induction hypothesis together with claim \eqref{eq.search4}.
Now for $t_i\in J_i^*(1\leq i\leq m)$ and $s_j\in\tilde{J}_j^*(1\leq j< q_0)$,
we estimate:
\begin{align}\label{eq.distinctness}
|s_j-t_i|&\geq |s_1-t_{\ell}|-|t_i-t_{\ell}|-|s_1-s_j|\\
&> Q(r+1)|J_{\ell}^*|-2Q(r)k|J_{\ell}^*|-2Q(r)k|\tilde{J}_{\ell}^*|\nonumber\\
&=(Q(r+1)-4Q(r)k)|J_{\ell}^*|\nonumber\\
&\geq |J_{\ell}^*|.\nonumber
\end{align}
This in particular means that $s_j\not\in J_i^*$ and $t_i\not\in\tilde{J}_j^*$.
Hence $J_i^*\neq\tilde{J}_j^*.$

We now begin checking the conditions $(1)$, $(2)$ and $(3)$.

{\bf (1)}:  If $s_{q_0}\not\in J'(\cal{X})$, there is nothing to check.

Now assume that
$s_{q_0}\in \til J_{q_0}'$ for some $(\til I_{q_0},\til J_{q_0}^*,\til J_{q_0}')\in \cal X$.
If $q_0<m'$, then again there is nothing to prove, as the union
\begin{equation}\label{eq.search0}
\{( I_i, J_i^*, J_i'):1\leq i\leq m\}\cup\{(\til I_j,\til J_j^*,\til J_j'):1\leq j\leq q_0\}
\end{equation}
intersects a family other than $\cal X_1,\cdots,\cal X_{r+1}$.
Hence we will assume $q_0=m'$.
By the induction hypothesis for $r$ on the sequence $(s_1,\cdots,s_{m'})$, the family
$\{(\til I_j,\til J_j^*,\til J_j'):1\leq j\leq m'\}$ consists of  pairwise distinct triples intersecting at least $(r+1)$ numbers of $\cal{X}_i$'s.
Observe that in the estimate (\ref{eq.distinctness}), there is no harm in allowing $j= q_0$ in addition to $j<q_0$.
This shows that $\tilde{J}_{m'}^*$ is also distinct from all $J_i^*$'s.
{Hence the the triples in \eqref{eq.search0} are all distinct.}
 
Now, unless the following inclusion
\begin{equation}\label{eq.search3}
\{(\til I_j,\til J_j^*,\til J_j'):1\leq j\leq m'\} \subset \cal{X}_1\cup\cdots \cup  \cal{X}_{r+1},
\end{equation}
holds, we are done.
Suppose that \eqref{eq.search3} holds. We will deduce a contradiction.
Without loss of generality, we assume that $$(I_{\ell},J_{\ell}^*,J_{\ell}')\in \cal{X}_{r+1}.$$

We now claim that the following inclusion holds:
\begin{equation}\label{eq.search5}
\{(\til I_j,\til J_j^*,\til J_j'):1\leq j\leq m'\} \subset \cal{X}_1\cup\cdots \cup  \cal{X}_r.
\end{equation}
Note that this gives the desired contradiction, since $\{(\til I_j,\til J_j^*,\til J_j'):1\leq j\leq m'\}$ must intersect at least $(r+1)$ number of $\cal{X}_i$ by
the induction hypothesis.
In order to prove the inclusion \eqref{eq.search5}, suppose on the contrary that $(\til I_j,\til J_j^*,\tilde{J}_j')\in\cal{X}_{r+1}$
 for some $1\le j\le m'$.
Using $\beta >(Q(r+1)+2Q(r))k$ and \eqref{eq.search4}, we deduce
\begin{align*}
|t_{\ell}-s_j|&\leq |t_{\ell}-s_1|+|s_1-s_j|\\
&\leq Q(r+1)k|J_{\ell}^*|+2Q(r)k|J_{\ell}^*|\\
&< \beta|J_{\ell}^*|
\end{align*}
where we used the induction hypothesis for the sequence $(s_1,\cdots, s_{m'})$ in the second line, to estimate the term $|s_1-s_j|$.

Next, applying the $\beta$-regularity to the pair $(I_{\ell},J_{\ell}^*)$, we conclude that $s_j\in I_{\ell}$.
Since $s_j\in \tilde{J}_j'$, it follows that $I_{\ell}\cap\til{J}_j'\cap \mT\neq\emptyset$.
This contradicts the condition that $\cal{X}_{r+1}$ is of $\mT$-multiplicity free, as both $(\til I_j,\til J_j^*,\til{J}_j')$ and $(I_{\ell},J_{\ell}^*,J_{\ell}')$ belong to the same family $\cal{X}_{r+1}$.
This completes the proof of (1).

\textbf{(2)}: 
For $1\leq i\leq m$ and $1\leq j\leq q_0$, observe that
\begin{align*}
|t_i-s_j|&\leq|t_i-t_\ell|+|t_\ell-s_1|+|s_1-s_j|\\
&\leq 2Q(r)k|J_{\ell}^*|+Q(r+1)k|J_{\ell}^*|+2Q(r)k  |J_\ell^*|\\
&< 2Q(r+1)k|J_\ell^*|
\end{align*}
as $Q(r+1)>4Q(r)$.
Hence we get the desired result by \eqref{eq.search4}.

\textbf{(3)}: We already have observed that the inequality $\beta >\epsilon^{-1} Q(r+1)k$ implies that
\begin{equation*}
(1-{\epsilon})^m|t_1|\leq|s_1|\leq(1+{\epsilon})^m|t_1|.
\end{equation*}
Combining this with the induction hypothesis, we deduce that
\begin{equation*}
(1-{\epsilon})^{m+i-1}|t_1|\leq|s_i|\leq(1+ {\epsilon})^{m+i-1}|t_1|
\end{equation*}
for all $1\leq i\leq q_0$.

Finally, the last statement of the lemma is obtained from the case $r=n$, since there are only $n$-number of
$\cal X_i$'s; hence the second possibility of (1) cannot arise for $r=n$.
\end{proof}

\subsection*{Proof of Proposition \ref{prop.2kthick}}
We may assume that $\cal X_i$'s are all of degree 1, by replacing each $\cal X_i$'s with $\delta$-number of families associated to it.

We set  $$\beta_0(n,k,1 )=(4k)^{n+1}\epsilon^{-1}$$
where $\epsilon$ satisfies
$\left(\tfrac{1+{\epsilon}}{1-\epsilon}\right)^{2^{n }-1}\le 2.$ Note that $\beta_0(n,k, 1)$ is equal to the number given in \eqref{bbbb}.
We may assume $x=0$ without loss of generality. Let $\la>0$. We need to find a point
\begin{equation}\label{tsss}
t'\in\Big([-2k\la,-\la]\cup [\la,2k\la]\Big)\cap \Big(\mT-\bigcup_{i\in \La}J'(\cal X_i)\Big).
\end{equation}

Choose $s>0$ such that
\begin{equation}\label{eq.b}
(1-\epsilon)^{-(2^{n}-1)}\la\leq s \leq 2(1+\epsilon)^{-(2^{n}-1)}\la.
\end{equation}
Since $\mT$ is globally $k$-thick, there exists 
$$t\in( [-ks,-s]\cup [s,ks])\cap \mT.$$

If $t\not\in\bigcup_{i=1}^{n} J'(\cal X_i)$, then by choosing $t'=t$, we are done.
Now suppose $t\in \bigcup_{i=1}^{n} J'(\cal X_i)$. Since $0\not\in\bigcup_{i=1}^{n} I(\cal X_i)$,
by applying Lemma \ref{lem.search} to $t\in \mT\cap (\bigcup_{i=1}^{n}J'(\cal X_i))$, we obtain $t'\in \mT- \bigcup_{i=1}^{n} J'(\cal X_i)$ such that
$$(1-\epsilon)^{2^{n}-1}|t|\leq|t'|\leq(1+{\epsilon})^{2^{n}-1}|t|.$$
Note that
$$|t'|\leq(1+{\epsilon})^{2^{n}-1}|t|\leq (1+{\epsilon})^{2^{n}-1}ks\leq 2k\la.$$
Similarly,  we  have
$$|t'|\geq(1-{\epsilon})^{2^{n}-1}|t|\geq (1-{\epsilon})^{2^{n}-1}s\geq \la.$$

This completes the proof since
$t'$ satisfies \eqref{tsss}.

\section{Avoidance of the singular set}\label{s:lt}
Let $\Gamma <G$ be a convex cocompact non-elementary subgroup and let $$U=\{u_t\}<N$$ be a one-parameter subgroup.
Let $\mS (U)$, $\mG(U)$,  $X(H, U)$, and $\mH^{\star}$ be as defined in section \ref{s:st}.
In particular, $\mS(U)$ is a countable union:
\begin{equation*}
\mathscr{S}(U )=\bigcup\limits_{H\in\mathscr{H}^{\star}}\Ga\ba\Ga X(H,U).
\end{equation*}

The main goal of this section is to prove the avoidance Theorem \ref{avoid1} for any convex cocompact hyperbolic manifold with Fuchsian ends.
For this, we extend the linearization method developed by Dani and Margulis \cite{DM} to our setting.
Via a careful analysis of the graded self-intersections of  the union 
  $\bigcup_i \Gamma\ba \Gamma H_i D_i \cap \RFM$ for finitely many groups  $H_i\in \mH^{\star}$ and compact subsets
$D_i\subset X({H_i},U)$, we construct families of triples of subsets of $\br$ satisfying 
the conditions of Proposition \ref{prop.2kthick}
relative to the global $k$-thick subset of the return time to $\RFM$ under $U$ given in Proposition \ref{defk}.

\subsection*{Linearization}\label{sec.rep}
Let $H\in \mH^{\star}$. Then $H$ is reductive, algebraic, and is equal to $\op{N}_G(H)$ by Proposition \ref{sin1} and \eqref{exam}.
There exists an $\br$-regular representation $\rho_H : G\to \mathrm{GL}({V}_H)$ with a point $p_H\in {V}_H $, such that $H=\op{Stab}_G(p_H)$ and the orbit $p_HG$ is Zariski closed \cite[Theorem 3.5]{BH}. 
Since $\Gamma\ba \Gamma H$ is closed, it follows that
$$p_H\Gamma$$
is  a closed (and hence discrete) subset of $V_H$.

Let $\eta_H:G\to V_H$ denote  the orbit map defined by $$\eta_H(g)=p_H g\quad\text{for all $g\in G$.}$$
As $H$ and $U$ are algebraic subgroups, the set $X(H,U)=\{g\in G: gUg^{-1}\subset H\}$ is Zariski closed in $G$.
Since $p_HG$ is Zariski closed in $V_H$, it follows that 
$A_H:=p_HX(H,U)$ is Zariski closed in $V_H$
and   $X(H,U)=\eta_H^{-1}(A_H).$

Following \cite{KM}, 
for given $C>0$ and $\alpha>0$, a function $f : \mathbb{R}\to \mathbb{R}$ is called $(C,\alpha)$-good if for any interval $I\subset\mathbb{R}$ and $\epsilon>0$, we have
\begin{equation*}
\ell \{t\in I : |f(t)|\leq\epsilon\}\leq C\cdot\left(\frac{\epsilon}{\sup_{t\in I}|f(t)|}\right)^\alpha\cdot \ell (I)
\end{equation*}
where $\ell$ is a Lebesgue measure on $\br$.

\begin{lemma}\label{ca} For given $C>1$ and $\alpha>0$,
consider continuous functions $p_1, p_2, \cdots, p_k:\mathbb{R}\to\mathbb{R}$  satisfying the $({C},\alpha)$-good property.
For $0<\delta<1$, set $$I=\{t\in\mathbb{R} : \max_i |p_i(t)|< 1\}\quad\text{ and } \quad  J(\delta) =\{t\in\mathbb{R} : \max_i |p_i(t)|<\delta\}.$$
For any $\beta>1$, there exists $\delta=\delta(C, \alpha, \beta)>0$ such that the pair $(I,J(\delta))$ is $\beta$-regular
(see Def. \ref{defreg}).
\end{lemma}
\begin{proof}
We prove that the conclusion holds for 
$\delta:= \left({(1+\beta)C}\right)^{-{1}/{\alpha}}.$ First, note that the function $q(t):=\max_i |p_i(t)|$ also has the $(C, \alpha)$-good property.
Let $J'=(a,b)$ be a component of $J(\delta)$, and $I'$ be the component of $I$ containing $J'$.
Note that $I'$ is an open interval and $(a,\infty)\cap I'=(a,c)$ for some $b\leq c\leq\infty$.
We claim
\begin{equation}\label{jb}
J'+\beta|J'|\subset (a,\infty)\cap I'\subset I'.
\end{equation}
We may assume that $c<\infty$; otherwise the inclusion is trivial.
We claim that $q(c)=1$. Since $\{t\in\bb{R} : q(t)<1\}$ is open and $c$ is the boundary point
of $I'$,  we have $q(c)\ge 1$. If $q(c)$ were strictly bigger than $1$,
 since $\{t\in\bb{R} : q(t)>1\}$ is open, $I'$ would be disjoint from an open interval around $c$, which is impossible. Hence $q(c)=1$.
Now that $\sup\{q(t) : t\in (a,\infty)\cap I'\}=q(c)=1$, by applying the $(C,\alpha)$-good property of $q$ on the interval $(a,\infty)\cap I'$, we get
\begin{align*}
\ell(J')&\leq \ell \{t\in (a,\infty)\cap I' : |q(t)|\leq\delta\}\\
&\leq C\delta^\alpha\cdot \ell ((a,\infty)\cap I').
\end{align*}
Now as $J'=(a,b)$ and $ (a,\infty)\cap I'$ are nested intervals with one common endpoint, it follows from
the equality $C\delta^\alpha =1/(1+\beta)$ that
\begin{equation*}
J'+\beta|J'|\subset (a,\infty)\cap I'\subset I',
\end{equation*} proving \eqref{jb}.
Similarly, applying the $(C,\alpha)$-good property of $q$ on $(-\infty,b)\cap I'$, we deduce that 
\begin{equation*}
J'-\beta|J'|\subset I'.
\end{equation*}
This proves that $(I,J(\delta))$ is $\beta$-regular.
\end{proof}

\begin{Prop}\label{lem.relsize}
Let $V$ be a finite dimensional real vector space, $\theta\in\bb{R}[V]$ be a polynomial and $A=\{v\in V :\theta(v)=0\}$.
Then for any compact subset $D\subset A$ and any $\beta>0$,
 there exists a compact neighborhood $D'\subset A$ of $D$ which has a $\beta$-regular size with respect to $D$
 in the following sense:
for any neighborhood $\Phi$ of $D'$, there exists a neighborhood $\Psi\subset \Phi$ of $D$ such that for any $q\in V-\Phi$ and 
for any one-parameter unipotent subgroup $\{u_t\}\subset \op{GL}(V)$, the pair $(I(q),J(q))$ is $\beta$-regular where
\begin{align*}
I(q)=\{t\in\mathbb{R} : qu_t\in\Phi\}\text{ and }
J(q)=\{t\in\mathbb{R} : qu_t\in\Psi\}.
\end{align*}
Furthermore, the degree of $(I(q), J(q))$  is at most $(\op{deg}\theta+2)\cdot\op{dim} V$.
\end{Prop}
\begin{proof}
Choose a norm on $V$ so that $\norm{\cdot}^2$ is a polynomial function on $V$.
Since $D$ is compact, we can find $R>0$ such that
\begin{equation*}
D\subset\{v\in V : \norm{v}<R\}.
\end{equation*}
Then we set
\begin{equation*}
D'=\{v\in V : \theta(v)=0,\,\norm{v}<R/\sqrt{\delta}\},
\end{equation*}
where $0<\delta<1$ is to be specified later.
Note that if $\Phi$ is a neighborhood of $D'$, there exists $0<\eta<1$ such that
\begin{equation*}
\{v\in V : \theta(v)<\eta,\,\norm{v}<(R+\eta)/\sqrt{\delta}\}\subset\Phi.
\end{equation*}
We will take $\Psi$ to be
\begin{equation*}
\Psi=\{v\in V : \theta(v)<\eta\delta,\,\norm{v}<(R+\eta)\}.
\end{equation*}
Set \begin{equation*}
\tilde{I}(q)=\{t\in\mathbb{R} : \theta(qu_t)< \eta,\,\norm{qu_t}<(R+ \eta)/\sqrt{\delta}\}.
\end{equation*}
Since $\tilde I(q)\subset I(q)$ for $0<\delta <1$,  it suffices to find $\delta$ (and hence $D'$ and $\Psi$) so that the pair
 $(\tilde{I}(q),J(q))$ is $\beta$-regular. If we set
\begin{equation*}
\psi_1(t):=\frac{\theta(qu_t)}{\eta}\quad \text{ and }\quad \psi_2(t):=\left(\frac{\norm{qu_t} \sqrt{\delta}}{R+ \eta}\right)^2,
\end{equation*}
then 
\begin{align*}
\tilde{I}(q)&=\{t\in\mathbb{R} : \max(\psi_1(t),\psi_2(t))<1\} ;\\
J(q)&=\{t\in\mathbb{R} : \max (\psi_1(t),\psi_2(t))<\delta\}.
\end{align*}

As $\psi_1$ and $\psi_2$ are polynomials, they have  the $(C, \alpha)$-property
for an appropriate choice of $C$ and $\alpha$. Therefore by applying Lemma
\ref{ca}, by choosing $\delta$ small enough, we can make the pair $(\tilde{I}(q),J(q))$ $\beta$-regular for any $\beta>0$.
Note that the degrees of $\psi_1$ and $\psi_2$ are bounded by $\deg\theta\cdot\op{dim} V$ and $2\op{dim}V$ respectively.
Therefore $J(q)$ cannot have more than $(\deg\theta+2)\cdot \op{dim} V$ number of
components.  Hence the proof is complete.
\end{proof}

\subsection*{Collection $\cal{E}_U$}
Recall the collection $\mathscr H^{\star}$ and the singular set:
$$\mathscr{S}(U )=\bigcup\limits_{H\in\mathscr{H}^\star}\Ga\ba\Ga X(H,U).$$
\begin{Def} \label{defeu}
We define $\cal E=\cal{E}_U$ to be the collection of all compact subsets of $\mS (U)\cap \RFM$ which can be written as
\begin{equation}\label{efff}
E=
\bigcup_{i\in \Lambda}  \Gamma\ba \Gamma H_i D_i \cap \RFM
\end{equation}
where  $\{H_i\in \mH^{\star}: i\in \Lambda\}$ is a finite collection and
 $D_i\subset X(H_i, U)$ is a compact subset. In this expression, we always use the minimal index set $\Lambda$ for $E$.
 When $E$ is of the form \eqref{efff}, we will say that $E$ is associated to the family $\{H_i: i\in \La\}$.
 \end{Def}
 
\begin{Rmk}  We note that $E$ can also be expressed as $\bigcup_{i\in \La}\Gamma\ba \Gamma H_i D_i \cap \RFM$
 where $H_i\in \mH$ is a finite collection, and $D_i\subset X(H_i, U)$ is a compact subset which is left
 $\op{C}(H_i)$-invariant.\end{Rmk}
 
 \begin{lem} \label{minimal} In the expression \eqref{efff} for $E\in \cal E$,
 the collection $\{H_i: i\in \La\}$ is not redundant,
in the sense that
\begin{itemize}
\item no $\gamma H_j\gamma^{-1} $ is equal to $H_i$ for all triples
$(i,j,\gamma)\in \Lambda\times \Lambda \times
\Gamma $ except for the trivial cases of $i=j$ and $\gamma\in H_i$.
\end{itemize}\end{lem}
\begin{proof}
Observe that if $\gamma H_j\gamma^{-1}=H_i$ for some $\gamma\in \Gamma$, then
$ \Gamma H_j D_j= \Gamma  H_i \gamma D_j$, and hence by replacing $D_i$ by $D_i\cup \gamma D_j\subset X(H_i, U)$, we may remove $j$ from the index subset $\La$. This contradicts
the minimality of $\La$.
\end{proof}

Observe that for any subgroups $H_1, H_2$ of $G$, and $g\in G$,
\begin{align*} X(H_1\cap g H_2 g^{-1}, U)&=
X(H_1 ,U)\cap X(gH_2 g^{-1}, U)\\ &=X(H_1,U)\cap g X(H_2, U).\end{align*}

Note that  for $D_i\subset X({H_i},U)$, and $\gamma\in \Gamma$,
the intersection
$H_1D_1\cap \gamma H_2D_2$ only depends on the $(\Ga\cap {H_1}, \Ga\cap {H_2})$-double coset of $\ga$.

\begin{prop}\label{lem.multiplicity}\label{mul}
Let $H_1, H_2\in \mH^{\star}$.
Then for any compact subset $D_i\subset X({H_i},U)$ for $i=1,2$ and a compact subset $K\subset \Ga\ba G$,
there exists a finite set $\Delta\subset (H_1\cap \Gamma)\ba \Gamma / (H_2\cap \Gamma)$ such that
\begin{equation*}
\big\{K\cap \Ga\ba\Ga( H_1D_1 \cap \gamma H_2D_2 )\big\}_{\gamma\in\Ga}
=\big\{K\cap \Ga\ba\Ga( H_1D_1\cap \gamma  H_2D_2)\big\}_{\gamma\in \Delta}
\end{equation*}
where the latter set consists of distinct elements.

Moreover for each $\gamma\in \Delta$, there exists a compact subset $C_0 \subset
H_1D_1 \cap {\gamma} H_2D_2
\subset
X(H_1\cap \gamma H_2\gamma^{-1}, U)$
such that
$$K\cap \Ga\ba\Ga( H_1D_1 \cap \gamma H_2D_2 ) =\Gamma\ba \Gamma C_0.$$
\end{prop}
\begin{proof} 
For simplicity, write $\eta_{H_i}=\eta_i$ and $p_i=p_{H_i}$.
Let $K_0\subset G$ be a compact set such that $K=\Ga\ba\Ga K_0$.
We fix $\ga\in\Ga$, and define for any $\ga'\in \Ga$,
\begin{equation*}
K_{\ga'}=\{g\in K_0 : \ga'g\in H_1D_1\cap\ga H_1D_2\}.
\end{equation*}
We check that
\begin{equation*}
K\cap \Ga\ba\Ga( H_1D_1 \cap\ga H_2D_2 )=\Ga\ba\Ga \left(\cup_{\ga'\in\Ga}K_{\ga'}\right).
\end{equation*}
If this set is non-empty, then $K_{\ga'}\neq\emptyset$ for some $\gamma'\in \Gamma$ and
\begin{equation*}
p_1\ga'g\in p_1 D_1,\quad p_2 \ga^{-1}\ga'g \in p_2 D_2
\end{equation*}
for some $g\in K_0$.
In particular,
\begin{equation}\label{eq.discrete}
p_1 \ga' \in p_1 D K_0^{-1},\quad p_2\ga^{-1} \in p_2 DK_0^{-1} \ga'^{-1}.
\end{equation}
As $p_1\Ga$ is discrete, and
$p_1D_1K_0^{-1}$ is compact, the first condition of (\ref{eq.discrete}) implies that there exists a finite set $\Delta_0\subset G$ such that $\ga'\in(H_1\cap\Ga) \Delta_0$.
Writing $\ga'=h\delta_0$ where $h\in H_1\cap\Ga$, and $\delta_0\in\Delta_0$, the second condition of (\ref{eq.discrete}) implies
\begin{equation*}
p_2 \ga^{-1}h \in p_2D_2 K_0^{-1} \delta_0^{-1}.
\end{equation*}
As $p_2 D_2 K_0^{-1} \Delta_0^{-1}$ is compact and $p_2\Ga$ is discrete, there exists a finite set $\Delta\subset G$ such that $\ga^{-1}h\in(H_2\cap\Ga)\Delta$.
Hence, if $
K\cap \Ga\ba\Ga( H_1D_1 \cap \gamma H_2D_2)\ne \emptyset$, then
  $\gamma\in(H_1\cap\Ga)\Delta(H_2\cap\Ga)$.
This completes the proof of the first claim.

For the second claim, it suffices to set $C_0:=\bigcup_{\gamma' \in \Delta} K_{\gamma'}$.
\end{proof}

\begin{prop}\label{sha}
Let $H_1, H_2 \in \mH^\star$ be such that $H_1\cap H_2$ contains a unipotent element.
Then there exists a unique smallest connected closed subgroup, say $H_0$,
 of $H_1\cap H_2$ containing all unipotent elements of $H_1\cap H_2$
such that $\Gamma\ba \Gamma H_0$ is closed.
Moreover, $H_0\in \mH$.
\end{prop}
\begin{proof} The orbit $\Gamma \ba \Gamma (H_1\cap H_2)$ is closed \cite[Lemma 2.2]{Sh1}. Hence such $H_0$ exists.
We need to show that $\Gamma\cap H_0$ is Zariski dense in $H_0$.
Let $L$ be the subgroup of $H_0$ generated by all unipotent elements in $H_0$. Note that $L$ is a normal
subgroup of $H_0$ and hence $(H_0\cap \Gamma) L$ is a subgroup of $H_0$.
If $F$ is the identity component of the closure of $(H_0\cap \Gamma)L$, then $\Gamma\ba \Gamma F$ is closed. By the minimality assumption on $H_0$,
we have $F=H_0$. Hence  $\overline{(H_0\cap \Gamma) L}=H_0$; so $\overline{[e]L} =[e]H_0$.
We can then apply \cite[Corollary 2.12]{Sh1} and deduce the Zariski density of $H_0\cap \Gamma$ in $H_0$.
\end{proof}

\begin{cor}\label{ddd2} Let $H_1, H_2\in \mH^\star$ and $\gamma\in \Gamma$ be satisfying that
 $X(H_1\cap \gamma H_2\gamma^{-1},U)\ne \emptyset$. Then there exists a subgroup $H\in \mH^{\star}$  contained in $H_1\cap \gamma H_2\gamma^{-1}$
such that
 for any compact subsets $D_i\subset X({H_i},U)$, $i=1,2$, there exists a compact subset $D_0\subset X({H},U)$ such that
$$K\cap \Gamma\ba \Gamma (H_1D_1\cap\gamma H_2D_2)= K\cap
\Gamma\ba \Gamma HD_0.$$
\end{cor}
\begin{proof} Let $F\in \mH$ be given by Proposition \ref{sha} for the subgroup $H_1\cap \gamma H_2\gamma^{-1}$. 
Set $H:=\op{N}_G(F_{nc})\in \mH^\star$.
Note that $X(H_1\cap \gamma H_2 \gamma^{-1}, U)= X(H,U)$. 
Hence, by the second claim of Proposition
\ref{mul}, there exists a compact subset $D_0\subset H_1D_1\cap \gamma H_2D_2$ such that
\be\label{ddd} K\cap \Gamma\ba \Gamma (H_1D_1\cap\gamma H_2D_2)=
\Gamma\ba \Gamma D_0.\ee
We claim that
$$\Gamma\ba \Gamma D_0 =K\cap \Gamma\ba \Gamma H D_0.$$
The inclusion $\subset$ is clear. Let $g:=hd \in HD_0$ with $h\in H$ and $d\in D_0$, and $[g]\in K$.
Then by the condition on $D_0$, we have $g \in H_1 D_1$ and $\gamma^{-1}g \in H_2D_2$.
Therefore $g\in H_1D_1\cap\gamma H_2D_2$. By \eqref{ddd},
this proves the inclusion $\supset$.
\end{proof}

\begin{Def} [Self-intersection operator on $\mathcal E_U$]
We define an operator $$\s: \mathcal E_U \cup \{\emptyset\} \to \mathcal E_U\cup\{\emptyset\}$$ as follows: we set $\s(\emptyset)=\emptyset$.
For any \be\label{eee} E=\bigcup_{i\in \Lambda} \Gamma\ba \Gamma H_iD_i \cap \RFM
\in  \mathcal E_U,\ee
we define
$$\s(E):=\bigcup_{i, j\in \Lambda}\bigcup_{ \gamma_{ij} \in \Gamma }\Gamma \ba \Gamma (
H_iD_i\cap \gamma_{ij} H_jD_j) \cap \RFM$$
where  $\gamma_{ij} \in \Gamma$ ranges over all elements of $\Gamma$ satisfying
$$\op{dim}(H_i\cap \gamma_{ij} H_j\gamma_{ij}^{-1})_{nc} <\min\{\op{dim} (H_i)_{nc}, \op{dim} (H_j)_{nc}\}.$$
\end{Def}

By Proposition \ref{mul} and Corollary \ref{ddd2},  we have:
\begin{cor}\label{comb}
\begin{enumerate}
\item For $E\in \cal E_U$, we have $\s(E)\in \cal E_U$.
\item For $E_1, E_2\in \cal E_U$, we have  $E_1\cap E_2\in \cal E_U$.
\end{enumerate}
\end{cor}

Hence for $E\in \cal E_U$ as in \eqref{eee},   $\s (E)$
is of the form
$$\s(E)= \bigcup_{i\in \Lambda'} \Gamma\ba \Gamma H_iD_i \cap \RFM$$
where $\Lambda'$ is a (minimal) finite index set, $H_i\in \mH$ with $X(H_i, U)\ne\emptyset$
 and $$\max\{ \op{dim} (H_i)_{nc} : i\in \Lambda'\} < \max\{ \op{dim}(H_i)_{nc} : i\in \Lambda\} .$$
Hence, $\s$ maps $\cal E_U$ to $\cal E_U\cup \{\emptyset\}$ and for any $E\in \cal E_U$,  $$\s^{\op{dim} G}(E)=\emptyset.$$

\begin{Def} For a compact subset $K\subset \Gamma\ba G$ and $E\in \mathcal E_U$, we say that $K$ does not have any self-intersection point of $E$,
or simply say that $K$ is {\it $E$-self intersection-free},
if  $$K\cap \s(E)=\emptyset .$$
\end{Def}

\begin{prop}\label{dmc}
Let $E=\bigcup_{i\in \Lambda} \Gamma\ba \Gamma H_iD_i \cap \RFM\in \mathcal E$ where $D_i\subset X({H_i},U)$ is a compact subset  and $\Lambda$ is a finite subset. Let $K\subset \RFM$ be a compact subset  which is $E$-self intersection-free.
Then there exists a collection of open neighborhoods $\Omega_i$ of $D_i$, $i\in \La$, such that
for $\cal O:=\bigcup_{i\in \Lambda} \Gamma\ba \Gamma H_i\Omega_i$, the compact subset
$K$ is $\cal O$-self intersection free, in the sense that, if $\op{dim} H_i=\op{dim}H_j$ and
$$K\cap \Gamma \ba \Gamma (H_i \Omega_i\cap \gamma H_j \Omega_j)\ne \emptyset$$
for some $(i,j,\gamma)\in \La\times \La\times \Gamma$,
then $i=j$ and $\gamma\in H_i\cap \Gamma.$ \end{prop}
\begin{proof}
 For each $k\in \bb N$ and $i\in \La$,
 let $\Omega_i(k)$ be the $1/k$-neighborhood of the compact subset $D_i$. 
Since $\Lambda$ is finite, if the proposition  does not hold,
by passing to a subsequence, there exist $i, j\in \Lambda$ with $\op{dim}H_i=\op{dim}H_j$ and a sequence $\gamma_k\in \Gamma$
 such that
$$K\cap \Gamma \ba \Gamma (H_i \Omega_i(k)\cap \gamma_k 
H_j \Omega_j(k))\ne \emptyset $$
and 
\be\label{iik} (i,j,\gamma_k)\notin \{(i,i, \gamma):i\in \La, \gamma \in H_i\cap \Gamma\}.\ee
Hence there exist $g_k=h_k w_k \in H_i \Omega_i(k)$ and $g_k'=h_k' w_k' \in H_j \Omega_j(k)$
such that $g_k=
\gamma_k g_k'$  where $[g_k]\in K$.
Now as $k\to \infty$, we have $w_k\to w\in D_i$ and $w_k'\to w'\in D_j$. There exists $\delta_k\in \Gamma$
such that $\delta_k g_k \in \tilde K$ where $\tilde K$ is a compact subset of $G$ such that 
$K=\Gamma\ba \Gamma \tilde K$, so the sequence $\delta_k g_k $ converges to $g_0$ as $k\to \infty$.
Since $ \Gamma H_i$ and $\Gamma H_j$ are closed, we have
  $\delta_kh_k \to \delta_0 h_i$ and $\delta_k \gamma_k h_k' \to \delta_0' h_j$
where $\delta_0, \delta_0'\in \Gamma$, $h_i\in H_i$ and $h_j\in H_j$.
As $\Gamma [H_i]$ and $\Gamma [H_j]$ are  discrete in the spaces $G/H_i$ and $G/H_j$ respectively,
 we have
\be\label{haha} \delta_0^{-1}\delta_k\in H_i\quad\text{and}\quad (\delta_0')^{-1}\delta_k \gamma_k\in H_j\ee
for all sufficiently large $k$.
Therefore $g_0=\delta_0 h_iw =\delta_0' h_j w' \in \delta_0(H_iD_i \cap \delta_0^{-1}\delta_0' 
H_jD_j)$ and
$[g_0]\in K$.
Hence $$K\cap \Gamma\ba \Gamma (H_iD_i \cap \delta_0^{-1}\delta_0' H_jD_j) \ne \emptyset.$$
Set $\delta:=\delta_0^{-1}\delta_0'\in \Gamma$.

Since $K\cap \s(E)= \emptyset$,
this implies that $\RFM\cap \Gamma \ba \Gamma (H_iD_i \cap \delta H_jD_j) \not\subset \s(E)$. By the definition
of $\s(E)$,
$$\op{dim}(H_i\cap \delta H_j\delta^{-1})_{nc} =\min\{\op{dim} (H_i)_{nc}, \op{dim} (H_j)_{nc}\}.$$
Since $H_i=\op{N}_G(H_i)=\op{N}_G((H_i)_{nc})$, and similarly for $H_j$,
we have  $H_i\cap \delta H_j\delta^{-1}$ is either $H_i$ or $\delta H_j \delta^{-1}$. Since $\dim H_i=\dim H_j$,
$\delta H_j\delta^{-1}= H_i$ or $H_i= \delta H_j\delta^{-1}$.

By Lemma \ref{minimal}, this implies that $i=j$ and $\delta\in \op{N}_G(H_i)\cap \Gamma$.
It follows from \eqref{haha} that $$\gamma_k\in \op{N}_G(H_i)\cap \Gamma=H_i\cap \Ga$$
for all large $k$.
This is a contradiction to \eqref{iik}, completing the proof.
\end{proof}

In the rest of this section, we assume that $\M=\Gamma\ba \bH^d$ is a convex cocompact hyperbolic manifold with Fuchsian ends, and let $k$ be
as given by Proposition \ref{defk}.
\begin{theorem}[Avoidance theorem I] \label{prop.2kthickapplication}\label{avoid1}
 Let $U=\{u_t\}<N$ be a one-parameter subgroup.
For any $E\in \mathcal E_U$, there exists  $E'\in \mathcal E_U$ such that the following holds:
If $F\subset\RFM$ is a compact set disjoint from $E'$, then
there exists a neighborhood $\cal O^\diamond$ of $E$ such that for all $x\in F$, 
the following set $$\{ t\in \bb{R} : xu_t\in\RFM-\cal O^\diamond\}$$ is $2k$-thick.
Moreover, if $E$ is associated to $\{H_i:i\in \La\}$, then $E'$ is also associated to the same family $\{H_i:i\in \La\}$ in the sense
of Definition \ref{defeu}.
\end{theorem}

\begin{proof}

\noindent{\bf $\spadesuit 1$. The constant $\beta_0$:} We write $\mH^\star=\{H_i\}$.
For simplicity, set $V_i=V_{H_i}$ and $p_i=p_{H_i}$.
Since $A_{H_i}$ is real algebraic, we can find a single polynomial $\theta_i$ whose zero locus coincides with $A_{H_i}$; namely, if the finitely generated ideal of polynomials vanishing on $A_{H_i}$ is given by $\langle f_1,\cdots,f_s\rangle$, then we can set $\theta_i=f_1^2+\cdots+f_s^2$.

 Set
 $$m:= \op{dim} (G)^2;\text{ and} $$
 $$\delta:=\max_{H_i\in \mH^\star} (\op{deg}\theta_i+2)\op{dim} V_i.$$

Note that if $H_i$ is conjugate to $H_j$, then $\theta_i$ and $\theta_j$ have same degree and $\op{dim} V_i=\op{dim}V_j$.
Since there are only finitely many conjugacy classes in $\mH^\star$ by Proposition \ref{sin1},  the constant
$\delta$ is finite.
Now let 
$$\beta_0:=\beta_0( m\delta ,k,1)=(4k)^{m\delta+1}\epsilon^{-1}$$
be given as in Proposition \ref{prop.2kthick}
where $\epsilon=\e_{m\delta }$ satisfies
$\left(\tfrac{1+{\epsilon}}{1-\epsilon}\right)^{2^{m\delta }-1}\le 2.$

\noindent{\bf $\spadesuit 2.$ Definition of $E_n$ and $E_n'$:}
We write $$E=\bigcup_{i\in \La_0}\Ga\ba\Ga H_iD_i \cap \RFM$$ for some finite minimal set $\La_0$. 
Set  
$$\ell:=\max_{i\in \Lambda_0} \op{dim} (H_i)_{nc}.$$
We  define $E_n, E_n' \in \cal E_U$ for all $1\le n\le \ell$  inductively as follows:
set $$E_\ell :=E \quad \text{and}\quad \Lambda_{\ell}:=\Lambda_0.$$
For each $i\in \La_{\ell}$, let $D_i'$ be a compact subset of $X({H_i},U)$ containing $D_i$ such that $p_{i}D_i'$ has a $\beta_0$-regular size with respect to $p_{i}D_i$ as in Proposition \ref{lem.relsize}.
Set $$E_{\ell}':=\bigcup_{i\in \La_{\ell}}\Ga\ba\Ga H_iD_i' \cap \RFM.$$
Suppose that $E_{n+1}, E_{n+1}' \in \cal E_U$ are given for $ n\ge 1$.
Then,  define
$$E_{n}:=E \cap \s (E_{n+1}').$$
Then by Corollary \ref{comb},
 $E_{n}$ belongs to $ \mathcal E_U$ and hence can be written as 
 $$E_{n} =\bigcup_{i\in \Lambda_{n}} \Gamma\ba \Gamma H_iD_i\cap \RFM$$
where $D_i$ is a compact subset of $X(H_i, U)$, so that $\La_{n}$ is a minimal index set.
 For each $i\in \La_{n}$, let $D_i'$ be a compact subset of $X({H_i},U)$ containing $D_i$ such that $p_iD_i'$ has a $\beta_0$-regular size with respect to $p_iD_i$ as in Proposition \ref{lem.relsize}.
Set $$E_{n}':=\bigcup_{i\in \La_{n}}\Ga\ba\Ga  H_i D_i' \cap \RFM.$$
Hence we get a sequence of compact (possibly empty) subsets of $E$:
$$E_1,  E_2,   \cdots , E_{\ell-1}, E_{\ell}=E ,$$
and a sequence of compact sets 
$$E_1',  E_2',   \cdots , E_{\ell-1}', E_{\ell}'=E' .$$

Note that $\s(E_1)=\s(E_1')=\emptyset$ by the dimension reason.\footnote{In fact $E_{\ell -i}=\emptyset$ for all $i\ge d-1$, but we won't  use this information}

\noindent{\bf $\spadesuit 3.$ Outline of the plan:}
Let $F\subset\RFM$ be a compact set disjoint from $E'$.
For $x\in F$, we set $$\mathsf T(x):=\{t\in\bb{R} :xu_t\in\RFM\}$$ which is a globally $k$-thick set
by Proposition \ref{defk}.
We will construct
\begin{itemize}
 \item
 a neighborhood $\cal O'$ of $E'$ disjoint from $F$, and
 \item
a neighborhood $\cal O^\diamond$ of $E$ 
\end{itemize}
such that for any $x\in \RFM -\cal O'$, 
we have
\begin{equation*}
\{ t\in \br : xu_t\in \RFM -\cal O^\diamond\} \supset \mathsf T(x)-J'(\cal X)
\end{equation*}
where $\cal X=\cal X(x)$ is the union of at most $m$-number  of $\beta_0$-regular families $\cal X_i$ of triples  $(I(q),J^*(q),J'(q))$ of subsets of $\bb R$ with degree $\delta$ and of $\mathsf T (x)$-multiplicity free.
Once we do that, the theorem is a consequence of Proposition \ref{prop.2kthick}.
Construction of such $\cal O'$ and $\cal O^\diamond$ requires an inductive process on $E_n$'s.

\noindent{\bf $\spadesuit 4$. Inductive construction of $K_n$, $\cal O_{n+1}'$, $\cal O_{n+1}$,  and $\cal O_{n+1}^\star$:}
Let $$K_0:=\RFM.$$
For each $i\in\La_1$, there exists  a neighborhood $\Omega_i'$ of $D_i'$  such that
 for $$\cal O_1':=\bigcup_{i\in \La_1}\Gamma\ba \Gamma H_i\Omega_i',$$
the compact subset $K_0$ is $\cal O_1'$-self intersection free  by Lemma \ref{dmc}, since $\s(E'_1)=\emptyset$.
By Proposition \ref{lem.relsize}, there exists a neighborhood $\Omega_i$ of $D_i$ such that the pair $(I(q),J(q))$ is $\beta_0$-regular for all $q\in V_i- p_i\Omega_i'$ where
\begin{align}\label{eq.l4}
I(q)=\{t\in\mathbb{R} : qu_t\in p_i \Omega_i'\}\;\; \text{ and }\;\;
J(q)=\{t\in\mathbb{R} : qu_t\in p_i \Omega_i\}.
\end{align}
Set $$\cal O_1:=\bigcup_{i\in \La_{1}}\Ga\ba\Ga H_i\Omega_i.$$
Since $E_1=\bigcup_{i\in \Lambda_1} \Gamma\ba \Gamma H_iD_i \cap \RFM$,
 $\cal O_1$ is a neighborhood
of $E_1=\s(E_{2}' )\cap E$. Now the compact subset $\s(E_{2}')- \cal O_1 $ is contained in $\s(E_{2}') -E$, which is relatively open
in $\s(E_{2}')$. Therefore we can take a neighborhood $\cal O_{1}^\star$
of $ \s(E'_{2}) -\cal O_1 $ so that
\begin{equation*}\label{contra} 
\overline{\cal O_1^{\star}} \cap E =\emptyset.\end{equation*}

We will now define the following quadruple  $K_n, \cal O_{n+1}', \cal O_{n+1}$ and
$\cal O_{n+1}^\star$  for each $1\leq n\leq \ell -1$ inductively:
\begin{itemize}
 \item a compact subset $K_n=K_{n-1} - (\cal O_n \cup \cal O_n^\star) \subset \RFM$,
\item  a neighborhood $\cal O_{n+1}'$ of $E_{n+1}'$, 
\item a neighborhood $\cal O_{n+1}$ of $E_{n+1}$  and
\item  a neighborhood $\cal O_{n+1}^\star$ of $\s(E_{n+2}')-\cal O_{n+1}$  such that
 $$\overline{ \cal O_{n+1}^\star} \cap E=\emptyset.$$
\end{itemize} 
Assume that the sets $K_{n-1}$, $\cal O_{n}'$, $\cal O_{n}$ and $\cal O_n^\star$ are defined.
We  define
$$K_n:=K_{n-1}-\cal ( \cal O_n \cup \cal O_{n}^\star)=\RFM- \bigcup_{i=1}^n \cal ( \cal O_i \cup \cal O_{i}^\star).$$
For each $i\in \La_{n+1}$, let $\Omega_i'$ be a neighborhood of $D_i'$ in $G$ such that 
for $\cal O_{n+1}':=\bigcup_{i\in \La_{n+1}} \Gamma\ba \Gamma H_i \Omega_i'$,
$K_n$ is $\cal O_{n+1}'$-self intersection free.  Since $\cal O_n \cup \cal O_n^\star$ is a neighborhood of
$\s(E_{n+1}')$, which is the set of all self-intersection points of $E_{n+1}'$,  such collection of
$\Omega_i'$, $i\in \Lambda_{n+1}$ exists by Lemma \ref{dmc}.

Since $F\subset\RFM$ is compact and disjoint from $E'$, we can also assume $\Ga\ba\Ga H_i\Omega_i'$ is disjoint from $F$, by shrinking $\Omega_i'$ if necessary.
More precisely, writing $F=\Ga\ba\Ga\til F$ for some compact subset $\til F\subset G$, this can be achieved by choosing a neighborhood $\Omega_i'$ of $D_i'$ so that $p_i\Omega_i'$
is disjoint from $p_i\Ga\til F$; and this is possible since $p_i\Ga\til F$ is a closed set disjoint from  a compact subset $p_iD_i'$.
After choosing $\Omega_i'$ for each $i\in\La_{n+1}$, define 
the following neighborhood of $E'_{n+1}$:
$$\cal O_{n+1}':=\bigcup_{i\in \La_{n+1}}\Ga\ba\Ga H_i\Omega_i'.$$

We will next define $\cal O_{n+1}$.
By Lemma \ref{lem.relsize}, there exists a neighborhood $\Omega_i$ of $D_i$ such that the pair $(I(q),J(q))$ is $\beta_0$-regular for all $q\in V_i-p_i\Omega_i'$ where
\begin{align*}
I(q)=\{t\in\mathbb{R} : qu_t\in p_i\Omega_i'\}\text{ and }
J(q)=\{t\in\mathbb{R} : qu_t\in p_i\Omega_i\}.
\end{align*}
We then define the following neighborhood of $E_{n+1}=\s(E_{n+2}' )\cap E$:
$$\cal O_{n+1}:=\bigcup_{i\in \La_{n+1}}\Ga\ba\Ga H_i\Omega_i .$$
Since the compact subset $\s(E_{n+2}')- \cal O_{n+1} $ is contained in the set
$\s(E_{n+2}') -E$, which is relatively open
inside $\s(E_{n+2}')$, we can take a neighborhood $\cal O_{n+1}^\star$
of $ \s(E'_{n+2}) -\cal O_{n+1} $ so that
\begin{equation*}
\overline{\cal O_{n+1}^\star} \cap E =\emptyset.\end{equation*}
This finishes the inductive construction.

\noindent{\bf $\spadesuit 5$. Definition of $\cal O'$ and $\cal O^\diamond$:}
We  define:
 $$\cal O':=\bigcup_{n=1}^\ell \cal O_n', \quad \cal O:=\bigcup_{n=1}^\ell \cal O_n \quad
 \cal O^\star:= \bigcup_{n=1}^\ell \overline{\cal O_n^\star} .$$
Note that $\cal O'$ and $\cal O$ are neighborhoods of $E'$ and $E$ respectively.
Since $E\cap \cal O^{\star}=\emptyset$, the following defines a neighborhood
of $E$:
\be\label{ooo} \cal O^{\diamond}:=\cal O - \cal O^\star .\ee

\noindent{\bf $\spadesuit 6$. Construction of $\beta_0$-regular families of $\mathsf T (x)$-multiplicity free:}

Fix $x\in F\subset\RFM-\cal O'$.
Choose a representative $g\in G$ of $x$.
We write each $\Lambda_n$ as the disjoint union 
$$\La_n=\bigcup_{j\in \theta_n} \Lambda_{n, j}$$
where
$\La_{n, j}=\{i\in \La_n:\op{dim}H_i=j\}$ and $\theta_n=\{j: \La_{n, j}\ne \emptyset\}$.
Note that $\# \theta_n <\dim G$.

Fix $1\le n\le \ell$, $j\in \theta_n$ and $i\in \La_{n,j}$.
For each $q\in p_i \Ga g$, we define the following subsets of $\br$:
\begin{itemize}
\item $I(q):=\{t: qu_t\in p_i\Omega_i'\}$ \text{ and}
\item $J(q):=\{t: qu_t\in p_i \Omega_i\}$.
\end{itemize}
In general, $I(q)$'s have high multiplicity among $q$'s in $ \bigcup_{i\in \Lambda_{n,j}} p_i\Gamma g$,
 but the following subset $I'(q)$'s will be multiplicity-free, and this is is why we defined $K_{n-1}$ as carefully as above:
\begin{itemize}
\item
$I'(q):=\{t :\text{for some } a\geq 0, \text{ }[t,t+a]\subset I(q)\text{ and } xu_{t+a}\in K_{n-1}\}$;
\item
$J^*(q):=I'(q)\cap J(q)$;
\item
$
J'(q):=\{t\in J(q) : x u_t\in K_{n-1}\}$.
\end{itemize}

Observe that $I'(q)$ and $J^*(q)$ are unions of finitely many intervals, $J'(q)\subset\mathsf T (x)$ and 
that $$J'(q)\subset J^*(q)\subset I'(q).$$

Now, for each $1\le n \le \ell$ and $j\in \theta_n$,
  define the family \be\label{nj} \cal X_{n,j}=\big\{(I(q),J^*(q),J'(q))) : q\in \bigcup_{i\in \Lambda_{n,j}} p_i\Ga g\big\}.\ee

We claim that each $\cal X_{n,j}$ is a $\beta_0$-regular family with degree at most $\delta$ and $\mathsf T (x)$-multiplicity free.

Note for each $q\in p_i\Ga g$, the number of connected components of $J^*(q)$ is less than or equal to that of $J(q)$. 
Now that $J^*(q)\subset J(q)$ and all the pairs $(I(q),J(q))$ are $\beta_0$-regular pairs of degree at most $\delta$, it follows
that  $\cal X_{n,j}$'s are $\beta_0$-regular families with degree at most $\delta$.

We now claim that $\cal X_{n,j}$ has $\mathsf T (x)$-multiplicity free, that is, 
 for any distinct indices $q_1, q_2\in \bigcup_{i\in \Lambda_{n,j}} p_i\Ga g$ of $\cal X_{n,j}$,
$$ I(q_1)\cap J'(q_2)=\emptyset.$$
We first show that
$$I'(q_1)\cap I'(q_2)=\emptyset.$$
 Suppose not. Then there exists
 $t\in I'(q_1)\cap I'(q_2)$ for some $q_1=p_i\ga_1g$ and $q_2=p_k\ga_2 g$, where $i, k\in \La_{n, j}$.
Then for some $a\geq 0$, we have $[t,t+a]\subset I(q_1)\cap I(q_2)$ and $xu_{t+a}\in K_{n-1}$.
In particular, $$x u_{t+a}\in \Gamma\ba \Gamma (\ga_1^{-1} H_i\Omega_i' \cap\ga_2^{-1}H_k\Omega_k')
\cap K_{n-1}.$$
Since $K_{n-1}$ is $\mathcal O_{n}'$-self intersection free, and $\op{dim}H_i=\op{dim}H_k=j$,
we deduce from Proposition \ref{dmc} that
this may happen only when $i=k$, and $\ga_1\ga_2^{-1}\in H_i\cap \Gamma$.
Hence we have  $$q_1=q_2.$$
This shows that $I'(q)$'s are pairwise disjoint.
Now  suppose that there exists an element $t\in I(q_1)\cap J'(q_2)$.
Then by the disjointness of $I'(q_1)$ and $I'(q_2)$, it follows that
$$t\in (I(q_1)-I'(q_1))\cap J'(q_2).$$
By the definition of $I'(q_1)$, we have $xu_t\not\in K_{n-1}$.
This contradicts the assumption that $t\in J'(q_2)$.

\noindent{\bf $\spadesuit 7$. Completing the proof:}
Let $\cal X:=\bigcup_{1\le i\le \ell, j\in \theta_n}\cal X_{n,j}$.
In view of Proposition \ref{prop.2kthick}, it remains to
check that the condition
$t\in \mathsf T(x)- J'(\cal X)$ implies that $xu_t\notin \mathcal O^\diamond$
where $\cal O^\diamond$ is given in \eqref{ooo}.

Suppose that there exists  $t\in  \mathsf T(x)- J'(\cal X)$ such that  $xu_t\in \cal O^\diamond$.
Write the neighborhood
$\mathcal O^\diamond$  as the disjoint union 
$$\cal O^\diamond=\bigcup_{n=1}^{\ell} \left(\cal O_n- (\cup_{i\le n-1} \cal O_{i} \cup \cal O^\star\right)).$$
Let  $n\le \ell $ be such that 
  $$xu_t\in \cal O_{n}- (\bigcup_{i=1}^{ n-1} \cal O_{i}\cup \cal O^\star).$$ 
  Since $t\in \mathsf T(x)-J'(\cal X)$, we have
 $xu_t\in \RFM -K_{n-1}$. Since $K_{n-1}=\RFM- \bigcup_{i=1}^{n-1} \cal ( \cal O_i \cup \cal O_{i}^\star)$, 
$$xu_t\in  \bigcup_{i=1}^{ n-1}  \cal O_{i}\cup \cal O_{i}^\star .$$
This is a contradiction, since $\bigcup_{i=1}^{\ell} \cal O_i^\star\subset \cal O^\star$.
\end{proof}

As  $\mH^{\star}$ is countable and
$X(H_i, U)$ is $\sigma$-compact, 
the intersection $\mS(U)\cap \RFM$ can be exhausted by the union of the increasing sequence of $E_j\in \mathcal E_U$'s.
Therefore, we deduce:

\begin{corollary}\label{cor.unifreclin}\label{lin} There exists an increasing sequence of compact subsets $E_1\subset E_2\subset \cdots $ in $\cal E_U$
with $\mathscr{S}(U)\cap \RFM=\bigcup\limits_{j=1}^\infty E_j$
 which satisfies the following:
Let $x_i\in \RFM$ be a sequence converging to $x\in \mathscr{G}(U)\cap \RFM$.
Then for each $j\in \bb N$, there exist a neighborhood $\cal O_j$ of $E_j$ and $i_j\ge 1$ such that
\begin{equation*}
\{t\in\bb{R} : x_iu_t\in\op{RF}\M-\cal O_j\}
\end{equation*}
is $2k$-thick for all $i\ge i_j$.
\end{corollary}
\begin{proof} For each $j\ge 1$,
 we may assume $E_{j+1}\supset E_j'$ where $E_j'$ is given by Theorem \ref{avoid1}.
For each $j\ge 1$, there exists $i_j\in\bb{N}$ such that $x_i\not\in E_{j+1}$ for all $i\geq i_j$.
Applying Proposition \ref{prop.2kthickapplication} to a compact subset $F=\{x_i : i\geq i_j\}$ of $\RFM$, we obtain 
a neighborhood $\cal O_j$ of $E_j$ such that
\begin{equation*}
\{t\in\bb{R} : x_iu_t\in\RFM-\cal O_j\}
\end{equation*}
is $2k$-thick for all $i\geq i_j$.
\end{proof}
Indeed we will apply Corollary \ref{lin} for the sequence $\{x_i\}$ contained in a closed orbit $x_0L$
of a proper connected closed subgroup $L<G$, which can be proved in the same way:

\begin{thm}[Avoidance Theorem II]\label{lin2}
Consider a closed orbit $x_0 L$  for some $x_0\in\RFM$ and $L\in \mathcal Q_U$. There exists an increasing sequence of compact subsets $E_1\subset E_2\subset \cdots $ 
in $\cal E_U$ with
$
\mathscr{S}(U,x_0L)\cap \RFM=\bigcup\limits_{j=1}^\infty E_j,
$
which satisfies the following:  if $x_i\to x$ in $\op{RF}\M\cap x_0L$ with $x\in \mathscr{G}(U, x_0L)$, then
 for each $j\in \bb N$,
there exist $i_j\ge 1 $ and  an open neighborhood $\cal O_j\subset x_0L$ of $ E_j$ such that  
\begin{equation*}
\{t\in \bb{R} : x_iu_t\in\op{RF}\M-\cal O_j\}
\end{equation*}
is a $2k$-thick set for all $i\ge i_j$.
\end{thm}

\section{Limits of $\RFM$-points in $F^*$ and generic points}\label{s:ge}
 Let  $\M=\Ga\ba \bH^d$ be a convex cocompact hyperbolic manifold with Fuchsian ends.
Recall that $\La\subset \bb{S}^{d-1}$ denotes the  limit set of $\Gamma$.
In this section, we collect some geometric lemmas which are needed in modifying a sequence limiting on an $\RFM$ point (resp. limiting
on a point in $\RFM\cap \mG(U)$)
 to a sequence of $\RFM$-points (resp. whose limit still remains inside $\mG(U)$).
Recall from Defintion \ref{d:rigid} that $\Omega=\bb S^{d-1}-\La$.
\begin{lemma}\label{lem.liminfnotsingleton}
Let $C_n\to C$ be a sequence of convergent circles in $ \bb{S}^{d-1}$.
If $C\not\subset \cl{B}$ for any component $B$ of $\Omega$,
then 
$$\#\limsup_{n\to \infty} C_n\cap \Lambda \ge 2.$$
\end{lemma}
\begin{proof} Without loss of generality, we may assume that $\infty\notin \Lambda$ and hence  consider
$\Lambda$ as a subset of the Euclidean space $\br^{d-1}$. Note that there is one component, say, $B_1$ of $\Omega$ which
contains $\infty$ and all other components of $\Omega$ are contained in the complement of $B_1$, which is a (bounded) round ball
in $\br^{d-1}$.
Since $B_1^c$ has a finite Lebesgue measure,
there are only finitely many components of $\Omega$ whose diameters are bounded from below by a fixed positive number.

Let $\delta=0.5\op{diam}(C)$ so that we may assume $\op{diam}(C_n)>\delta$ for all sufficiently large $n\gg 1$.
It suffices to show  that there exists $\e_0>0$ such that
$C_n\cap\La$ contains  $\xi_n, \xi_n'$ with
 $d(\xi_n,\xi_n')\geq\epsilon_0$ for all sufficiently large $n$.
Suppose not. Then for any $\epsilon>0$, there exists an interval $I_n\subset C_n$ such that $\op{diam}(I_n)\leq\epsilon$ and $C_n-I_n\subset\Omega$ for some infinite sequence of $n$'s.
Since $C_n-I_n$ is connected, there exists a component $B_n$ of $\Omega$ such that $C_n\subset\cal{N}_\epsilon(B_n)$, where $\cal{N}_\epsilon(B_n)$ denotes the $\epsilon$-neighborhood of $B_n$.
In particular, we have $\op{diam}(B_n)+\epsilon>\delta$.
Taking $\epsilon$ smaller than $0.5 \delta$, this means that $\op{diam}(B_n)>0.5\delta$.
On the other hand, there are only finitely many components of $\Omega$ whose diameters are greater than $0.5\delta$, say $B_1,\cdots,B_{\ell}$.
Let $\epsilon_0>0$ be such that $\cal{N}_{\epsilon_0}(B_1),\cdots,\cal{N}_{\epsilon_0}(B_{\ell})$ are all disjoint.
Then  by passing to a subsequence, there exists $B_i$ such that
 $C_n\subset\cal{N}_\epsilon(B_i)$ for all small $0<\epsilon <\epsilon_0$ and all $n\ge 1$; hence $C\subset\cal{N}_\epsilon(B_i)$.
Since this holds  for all sufficiently small $\epsilon>0$, we get that $C\subset \overline{B_i}$, yielding a contradiction to the hypothesis on $C$.
\end{proof}

In the next two lemmas, we set $U^-=U$ and $U^+=U^t$.
\begin{lemma}\label{lem.propersphere}
Let $U<N$ be a connected closed subgroup.  Let $[g]L$ be a closed orbit for some $L\in\cal L_U$ and $[g]\in \RFM$.
Let $S_0$ and $S^*$ denote the boundaries of $\pi(gH(U))$ and $\pi(gL)$ respectively.
If $S$ is a sphere such that $S_0\subset S\subsetneq S^*$ and $\Ga S$ is closed, then $[g]\in\mathscr{S}(U^\pm,[g] L)$.
\end{lemma}
\begin{proof}
Write $L=H(\widehat U)C\in \mathcal L_U$.
Since $S_0\subset S\subsetneq S^*$, there exists a connected proper subgroup $\tilde U $ of $\widehat U$, containing $U$
such that $S$ is the boundary of $\pi(gH(\tilde U))$.
Since $\Ga S$ is closed, $[g]H'(\tilde U)$ is closed by Proposition \ref{proper}. Now the claim follows from Proposition \ref{prop.countability}
and the  definition of $\mathscr{S}(U^\pm,[g]L)$.
\end{proof}

\begin{lemma}\label{lem.frameshift}
Let $U<N$ be a connected closed subgroup with dimension $m\ge 1$, and let $U_{\pm}^{(1)}, \cdots, U_{\pm}^{(m)}$ be one-parameter subgroups generating $U^{\pm}$.
Consider a closed orbit $yL$ where $L\in \mathcal L_U$ and $$y\in F_{H(U)}^*\cap\RFM\cap  \bigcap_{i=1}^m \mG(U_{\pm}^{(i)}, yL).$$ 
If  $x_n\to y$ in $yL$, then, by passing to a subsequence,
there exists a sequence  $h_n\to h$ in $H(U)$ so that  
$$x_nh_n \in \RFM\cap yL\quad\text{ and }yh\in \RFM\cap  \bigcap_{i=1}^m \mG(U^{(i)}_{\pm}, yL) .$$
\end{lemma}
\begin{proof} Recall from Section \ref{sec.notation} that $\mathcal C^k$ denotes the space of all oriented $k$-spheres in $\bb S^{d-1}$.
Let $g_0\in G$ be such that $y=[g_0]$ and $S^*$ denote the boundary of $\pi(g_0L)$.
Let $\mathcal Q$ be the collection of all spheres $S\subsetneq S^*$ such
 that $S\cap \La \ne \emptyset$ and $\Gamma S$ is closed
in $\mathcal C^{\op{dim}S}$. By Corollary \ref{rcount} and Remark \ref{Qcount}, $\cal Q$ is countable.
Choose a sequence $g_n \to g_0$ in $G$ as $n\to\infty$, so that $x_n=[g_n]$. 
Let $S_n$ and $S_0$ denote the boundaries of $\pi(g_nH(U))$ and $\pi(g_0H(U))$ respectively
so that $S_n\to S_0$ in $\mathcal C^m$ as $n\to\infty$.

We will choose a circle $C_0\subset S_0$ and a sequence of circles $C_n\subset S_n$ so that
$C_n \to C_0$ and $\limsup ( C_n\cap \La)$ contains two distinct points outside of $\cup_{S\in \cal Q} S$.
If $m=1$, we set $C_0=S_0$. When $m\geq 2$, we choose a circle $C_0\subset S_0$ as follows. 
Note that $S_0$ is not contained in any sphere in $\cal Q$ by the assumption on $y$ and Lemma \ref{lem.propersphere}.
Hence for any $S\in \cal Q$,  $S_0\cap S$ is a proper sub-sphere of $S_0$.
 Since $y\in F^*_{H(U)}$,
for any component $B_i$ of $\Omega$, $S_0\not\subset \cl{B_i}$ and hence
$S_0\cap \partial B_i$ is a proper sub-sphere of $S_0$. 
Choose a circle $C_0\subset S_0$  such that $\{g_0^+, g_0^- \}\subset C_0\cap \La$, 
 $C_0\not\subset S$ for any $S\in \cal Q$, and
$C_0\not\subset \partial B_i\cap S_0$ for all $i$.
This is possible, since $\cal Q$ is countable.
  Since $S_n\to S_0$, we can find a sequence of circles $C_n\subset S_n$  such that $C_n \to C_0$.
We claim that  $\limsup_n (C_n\cap\La) $ is uncountable.
Since $\# C_0\cap \La \ge 2$ and $C_0\not\subset \partial B_i$,
 $C_0\not\subset \overline{B_i}$ for all $i$.
Therefore, by Lemma \ref{lem.liminfnotsingleton}, for any infinite subsequence $C_{n_k}$ of $C_n$,
\begin{equation*}
\# \limsup_k (C_{n_k}\cap\La) \ge 2.
\end{equation*}
By passing to a subsequence, we can find two distinct points $\xi_n,\xi_n'\in C_n\cap\La$ which converge to two distinct points $\xi,\xi'$ of $ C_0\cap \Lambda$ respectively as $n\to\infty$.
Choose a sequence $p_n\to p \in G$ such that $p_n^+=\xi_n$, $p_n^-=\xi_n'$, $p^+=\xi$ and $p^-=\xi'$.
Let $\langle u_t\rangle<N$ be a one-parameter subgroup such that $p_nu_t^-=C_n-\{\xi_n\}$.
By Proposition \ref{lem.thickreturntime}, $T_n=\{t: [p_n]u_t\in \RFM\}$ is a global $k$-thick subset, and hence
$\mathcal T:=\limsup_n T_n$ is a global $k$-thick subset contained in the set $\{t: [p]u_t\in \RFM\}$. Then $C_n \cap \La$  converges, in the Hausdorff topology, to a compact subset $L\subset C_0\cap \La$ homeomorphic
to the one-point compactification of $\mathcal T$.
Therefore  $L$ is uncountable, so is $\limsup_n (C_n\cap\La)$, proving the claim.

Let $\Psi:=\cup_{S\in \cal Q} C_0\cap S$, i.e., the union of all possible intersection points of $C_0$ and spheres in $\cal Q$.
Since $C_0\not\subset S$ for any $S\in \cal Q$, $\#C_0\cap S \le 2$.
Hence $\Psi$ is countable, and hence  $\limsup_n (C_n\cap\La) -\Psi$ is uncountable.
Note that this works for any infinite subsequence of $C_n$'s.
Therefore we can choose sequences
 $\xi_n^-,\xi_n^+\in C_n\cap\La$ converging to
  distinct points $\xi^-,\xi^+$ of $(C_0\cap \La)-\Psi$ respectively, by passing to a subsequence.
 As $\xi^-,\xi^+\in C_0$ and $C_0\subset S_0$, there exists a frame $g_0h=(v_0, \cdots, v_{d-1})\in g_0H(U)$ whose first vector $v_0$ is tangent to the geodesic $[\xi^-, \xi^+]$. Setting $g:=g_0h$, we claim that 
$$[g] \in\bigcap_i \mG(U^{(i)}_\pm,yL).$$
Suppose that  $[g] \in\mS(U^{(i)}_\pm,yL)$ for some $i$.
We will assume $[g]\in\mS(U^{(i)}_-,yL)$, as the case when $[g]\in\mS(U^{(i)}_+,yL)$ can be dealt similarly, by changing the role of $g^-$ and $g^+$ below. 
For simplicity, set $U^{(i)}:=U^{(i)}_- $. Now by Proposition \ref{explain},
there exist $L_0\in\cal L_{U^{(i)}}$  and $\alpha\in N\cap L$ such that $(L_0)_{nc}\lneq L_{nc}$ and $[g] \alpha L_0$ is closed. 
Let $S$ denote the boundary of $\pi(g\alpha L_0)$. Since $\alpha  \in N\cap L$, we have
$(g\alpha)^+=g^+=\xi^+  \in S\cap \La\cap C_0$. Since $S\subsetneq S^*$, $S\cap \La \ne \emptyset$ and
 $\Ga S$ is closed, we have $S\in \cal Q$.
It follows that $\xi^+\in\Psi$, contradicting the choice of $\xi^+$. This proves the claim.

Now choose a vector $v_0^{(n)}$ which is tangent to the geodesic $[\xi_n^-,\xi_n^+]$.
We then extend $v_0^{(n)}$ to a frame $g_nh_n\in g_nH(U)$ so that $g_nh_n$ converges to $g=g_0h$ as $n\to\infty$.
Since $\{\xi_n^{\pm}\}\subset \La$, we have $[g_nh_n]\in \RFM$.
This completes the proof.
\end{proof}

We will need the following lemma later.
\begin{lemma}\label{lem.bddgeodimage} Let $k\ge 1$.
Let $\chi$ be a $k$-horosphere in $\bb{H}^{k+1}$ resting at $p\in\partial\bb{H}^{k+1}$, and $\cal P$ be the geodesic $k$-plane in $\bb{H}^{k+1}$.
Let $\xi\in\partial\cal P$, $\delta$ be the geodesic joining $\xi$ and $p$, and $q=\delta\cap \chi$.
There exists $R_0>1 $ such that for any $R>R_0$, if $d(\chi,\cal P)<R-1$, then $d(q,\cal P)<R$.
\end{lemma}
\begin{proof}
For $k=1$, this is shown in \cite[Lemma 4.2]{MMO2}. Now let $k\ge 2$.
Consider a geodesic plane $\bb{H}^2\subset\bb{H}^{k+1}$ which passes through $q$ and orthogonal to $\cal P$.
Then $\chi\cap\bb{H}^2$  and $\cal P\cap\bb{H}^2$ are  a horocycle and a geodesic
in $\bb{H}^2$ respectively.
As  $d_{\bb{H}^{k+1}}(\chi,\cal P)=d_{\bb{H}^2}(\chi\cap\bb{H}^2,\cal P\cap\bb{H}^2)$ and  $d_{\bb{H}^{k+1}}(q,\cal P)=d_{\bb{H}^2}(q,\cal P\cap\bb{H}^2)$,
the conclusion follows from the case $k=1$.
\end{proof}
Recall the definition of $\check H=H(U_{d-2})$ from Section \ref{s:ri}.
\begin{lemma}\label{lem.zc3}
Let $U<\check H\cap N$  be a non-trivial connected closed subgroup.
If the boundary of $\pi(gH(U))$ is contained in $\partial B$ for some component $B$ of $\Omega$, then $[g]\in\BFM\cdot\op{C}(H(U))$.
\end{lemma}
\begin{proof}
As  $U$ is equal to $mU_km^{-1}$ for some $m\in\check H\cap M$ and $1\le k\le d-2$, the general case is easily reduced to the case when  $U=U_k$.
Since $g=(v_0, \cdots, v_{d})$ has its first $(k+1)$-vectors tangent to the geodesic $(k+1)$-plane $\pi(gH(U_k))$ and
$\partial (\pi (gH(U_k)) \subset\partial B$, we can 
use an element $c\in \op{C}(H(U_k))=\SO(d-k-2)$ to modify the next  $(d-k-2)$-vectors so that $gc$
has its first $(d-1)$-vectors tangent to $\op{hull}(\partial B)$. Then $[gc]\in\BFM$, proving the claim.
\end{proof}

\begin{lemma}\label{lem.geometric}\label{geo} Let $U<\check H\cap N$ be a non-trivial connected closed subgroup.
If $x_n\in\op{RF}\M\cdot U$ is a sequence converging to some $x\in\op{RF}\M$, then passing to a subsequence, there exists $u_n\in U$ such that $x_nu_n\in\op{RF}\M$ and at least one of the following holds:
\begin{enumerate}
\item
$u_n\to e$ and hence $x_nu_n\to x$, or
\item
$x=zc$ for some $z\in\op{BF}\M$ with $c\in\op{C}(H(U))$, and $x_nu_n$ accumulates on $z\check{H}c$.
\end{enumerate}
\end{lemma}
\begin{proof}
If $x_n$ belongs to $\op{RF}\M$ for infinitely many $n$, we simply take $u_n=e$.
So assume that $x_n\not\in\op{RF}\M$ for all $n$.
Choose a sequence $g_n\to g_0$ in $G$ so that $x_n=[g_n]$ and $x=[ g_0]$.
As $x\in \RFM$, we have $\{g_0(0), g_0(\infty) \}\subset \La$.  As $x_n \in \RFPM -\RFM$, 
we have $g_n(\infty)\in \La$ and $g_n(0)\in \Omega$.
For each $n$, choose an element $u_n\in U$ so that
$0< \alpha_n:=\|u_n\| \le \infty$ is the minimum of $\|u\|$ for all $u\in U$ satisfying  $g_nu(0) \in\La$.
Set
\begin{equation*}
\alpha:=\limsup\limits_{n}\alpha_n.
\end{equation*}
If $\alpha=0$, then we are in case (1). Hence we will assume $0<\alpha\leq\infty$.
Let $C_n$ denote the boundary of $\pi(g_nH(U))$ and $C_0$ the boundary of $\pi(g_0H(U))$. Then $C_n\to  C_0$ in $\cal C^{\op{dim} U}$.
Recall that $B_U(r)$ denotes the ball of radius $r$ centered at $0$ inside $U$.
Set $$\cal B_n:=g_nB_U(\alpha_n)(0) \;\;\text{ and } \;\;\cal B_0:= g_0B_U(\alpha)(0).$$
Then $\cal B_n\subset C_n\cap \Omega $, and $\partial \cal B_n\cap \La\ne \emptyset$ by the choice of $u_n$.
By passing to a subsequence, we have $\alpha_n \to \alpha$ and $\cal B_n\to \cal B_0$ as $n\to\infty$ and hence
the diameter of $\cal B_n$ in $\bb{S}^{d-1}$ is bounded below by some positive number.
Hence, passing to a subsequence, we may assume that
$\cal B_n$ are all contained in the same component, say $B$ of $\Omega$.
Consequently, $\cal B_0\subset\cl B$.

We claim that $\# \cl{\cal B_0}\cap \partial B\ge 2$.
First note that  $g_0(0)\in\La$. 
If $\alpha=\infty$, then $g_n u_n (0)\to g_0(\infty)\in \La\cap \overline{\cal B_0}$.
If $\alpha<\infty$, then $u_n $ converges to some $ u\in U$, passing to a subsequence, and $u\ne e$, as $\alpha>0$.
Now, $g_nu_n(0)\to g_0u(0)\in \La \cap \overline{\cal B_0}$. Since $\La\cap \cl{B}\subset \partial B$, this  proves the claim.

  Therefore $\cal B_0$ is contained in $\partial B$, and hence so is $C_0$. By Lemma \ref{lem.zc3}, this implies that $x=zc$ for some $z\in \BFM$ and $c\in\op{C}(H(U))$.
We proceed to show that $x_nu_n$ accumulates on $z\check Hc$.
Since $c\in\op{C}(H(U))$, we may assume $c=e$ by replacing $x$ with $xc^{-1}$, and $x_n$ with $x_nc^{-1}$.

We claim that the distance between $\pi(g_nu_n)$ and the plane $\pi(g_0\check H)$ tends to $0$ as $n\to\infty$.
Since $x\check H=[g_0]\check H$ is compact, $g_nu_n\in g_n\check H$ and $\pi(g_n\check H)$ is a geodesic plane nearly parallel to $\pi(g_0\check H)$ for all  large $n$,
 this claim implies  that $[g_n]u_n$ accumulates on $z\check H$, completing the proof.

Now, to prove the claim, let $D_n:=C_n\cap \partial B$, and $\cal P_n:=\hu(D_n)$. Let $k=\op{dim}U$.
Since $C_n$ is a $k$-sphere meeting the $(d-2)$-sphere $\partial B\subset \mathbb S^{d-1}$, and $C_n\not\subset\partial B$, it follows that $D_n$ is a $(k-1)$-sphere.
We set $\cal{H}_n:=\op{hull}(C_n)$, $\cal H_0:=\op{hull}(C_0)$ and $\cal H:=\op{hull}(\partial B)=\pi(g_0\check H)$.
Then $\cal H_n\cap \cal H=\cal P_n$.
Let $\epsilon>0$ be arbitrary, and $\cal{N}_\epsilon(\cal H) $ denote  the $\epsilon$-neighborhood of $\cal H$ in $\bb{H}^d$.
Letting $d_{\cal H_n}(\cdot,\cdot)$ denote the hyperbolic distance in $\cal H_n$, we may write
\begin{equation*}
\cal{N}_\epsilon (\cal H) \cap \cal{H}_n=\{p\in \cal {H}_n : d_{\cal H_n}(p, \cal P_n)< R_n\}
\end{equation*}
for some $R_n>0$.
This is because $\cal{N}_\epsilon (\cal H) \cap \cal H_n$ is convex and invariant under family of isometries, whose axes of translation and rotation are contained in $\cal P_n$.
As $C_n\to C_0\subset \partial B$ as $n\to\infty$, it follows that $R_n\to\infty$ as $n\to\infty$.
Let $\chi_n\coloneqq\pi(g_nU)$, and $\chi_0\coloneqq\pi(g_0U)$, which are $k$-horospheres contained in $\cal H_n$ and $\cal H_0$ respectively.

We next show that there is a uniform upper bound for $d_{\cal H_n}(\cal P_n,\chi_n)$, $n\in \bb N$.
To see this, we only need to consider those $\cal P_n$'s which are disjoint from $\chi_n$, as $d_{\cal H_n}(\cal P_n,\chi_n)=0$ otherwise.
Since $\chi_n\to\chi_0$ and $C_n\to C_0$ as $n\to\infty$, it suffices to check that the diameters of $D_n$ with respect to the
spherical metric on $\bb{S}^{d-1}$ have a uniform positive lower bound.
Let us write $C_n-D_n=E_n\cup E_n'$, where $E_n$ is a connected component of $C_n-D_n$ meeting $B$, and $E_n'$ is the other component.
Since $C_n\to C_0$ as $n\to\infty$, a uniform lower bound for both $\op{diam}(E_n)$ and $\op{diam}(E_n')$ will give a uniform upper bound for $\op{diam}(D_n)$.
Since $\cal B_n\subset E_n$, $\op{diam}(E_n)>\op{diam}(\cal B_0)/2$ for all sufficiently large $n$.
On the other hand, note that $\chi_n\subset\cal H_n$ is a horosphere resting at a point in $E_n'$.
Since $\chi_n$ converges to $\chi$, the condition that
$\cal P_n\cap\chi_n=\emptyset$ implies that $\op{diam}(E_n')$ is also bounded below by some positive constant.
Since $R_n\to\infty$, we conclude that $d_{\cal H_n}(\cal P_n,\chi_n)<R_n-1$ for all sufficiently large $n$.
Applying Lemma \ref{lem.bddgeodimage} to $\bb{H}^{k+1}=\cal H_n$, $\chi=\chi_n$, $\cal P=\cal P_n$, $\xi=g_n^+$ and $q=\pi(g_nu_n)$, we have 
$$d_{\cal H_n}(\pi(g_nu_n),\cal P_n)<R_n$$ and hence $\pi(g_nu_n)\in \cal{N}_\epsilon (\mathcal H) \cap \cal H_n$, for all sufficiently large $n$.
As $\epsilon>0$ was arbitrary, this proves that $\pi(g_nu_n)$ goes arbitrarily close to $\pi(g_0\check H)$ as $n\to \infty$.
This finishes the proof.
\end{proof}

\begin{lem}\label{xfss} Let $U<N$ be a non-trivial connected closed subgroup.
If $x_n \to x$ in $F^*\cap \RFPM$, and $x\in F^*\cap \RFM$, then
there exists $u_n \to e$ in $U$ such that $x_nu_n\in \RFM$; in particular, $x_nu_n \to x$ in $F^*\cap \RFM$.
\end{lem}
\begin{proof} The general case easily reduces to the case when $U<\check H\cap N$. Then
the claim follows  from Lemma \ref{geo} and Lemma \ref{lem.R1}.
\end{proof}

\subsection*{Obtaining limits in $F^*$} 
For $\epsilon>0$, we set
\begin{equation}\label{coree}
\coree:=\big\{x\in\Ga\ba G : \pi(x)\in\core {\M}\text{ and }d(\pi(x),\partial\core {\M})\geq\epsilon\big\}.
\end{equation}
We note that $\coree$ is a compact of $F^*$ for all sufficiently large $\e>0$.
In the rest of the section, let $$U<N$$ denote a non-trivial connected closed subgroup.

\begin{lemma}\label{lem.epsiloncore}\label{corea}
Let $x\in\RFM$, and $V=\{v_t:t\in \br\}<N$ be a one-parameter subgroup.
   If $\pi(xV)\not\subset\partial\core {\M}$, and
$xv_{t_i}\in \RFM$ for some sequence $t_i\to  +\infty$, then
there exists a sequence $s_i \to  +\infty$ such that
$xv_{s_i}$ converges to a point in $F^*$.
\end{lemma}
\begin{proof}
It suffices to show that there exists a sequence $s_i \to  +\infty$ such that
$xv_{s_i}\in  \op{core}_{\eta/3}(M)$
where $\eta$ is as given in \eqref{min}.  Let $x=[g]$, and set $o=(1, 0,\cdots,0)\in\bb{H}^d=\br^+ \times \br^{d-1}$.
We may assume $g=(e_0,\cdots,e_{d-1})_o\in\op{F}\bb{H}^d$ where $e_i$ are standard basis vectors
in $\op{T}_o\bb{H}^d\simeq\bb{R}^d$.
Note that for $V^+=\{v_t: t>0\}$, $g V^+$ is a translation of the frame $g$ along a horizontal ray emanating from $o$ along the $V^+$-direction.
By the definition of $\eta$,  the $\eta/3$-neighborhoods of $\hull B_i$'s are mutually disjoint. For each $i$, set $s_i:=t_i$ if $xv_{t_i}\in  \op{core}_{\eta/3}(M)$. Otherwise,
there exists a unique $j$ such that $d(\pi( gv_{t_i}), \hull B_j)<\eta/3$.
If
$\pi(gV_{[t_i, \infty)})$ were contained in the $\eta/3$-neighborhood of $\hull B_j$, then the unique geodesic $2$-plane
 which contains $\pi(gV_{[t_i, \infty)})$ must lie in $\partial \hull B_j$, and hence
 $\pi(xV)\subset \partial\core\M$; this contradicts the hypothesis.
Therefore there exists $t_i<s_i<\infty$ such that  $d(\pi( gv_{s_i}), \hull B_j)=\eta/3$.
 The sequence $s_i$ satisfies the claim.
 \end{proof}

\begin{lemma}\label{lem.epsiloncore3}
Let  $x_nL_nv_n$ be a sequence of closed orbits with $x_n\in  \RFPM$, $L_n\in\cal{L}_U$ and $v_n\in (L_n\cap N)^\perp$.
Suppose that either 
\begin{enumerate}
\item  $x_n\in F^*$ for all $n$; or
\item  $x_nL_n v_n\cap \RFPM\cap F^*\ne \emptyset $ for all $n$.
\end{enumerate}  Then $$ F^*\cap  \limsup_n(x_nL_nv_n \cap\op{RF}_+\M) \ne \emptyset.$$
\end{lemma}

\begin{proof}
We claim that if $x_n\in F^*$, then
$x_nL_nv_n \cap \RFPM\cap F^*\ne \emptyset$, that is, the hypothesis (1) implies (2).
 Suppose not. Then, since $A\subset L_n$,
$(x_nA v_n A\cap\op{RF}_+\M) \subset\op{RF}_+\M-F^*$.
Since the set $\op{RF}_+\M-F^*$ is a closed $A$-invariant set and $e\in \overline{Av_nA}$, we would have
$x_n\in \op{RF}_+\M-F^*$, yielding a contradiction.
It follows from the claim that there exists $z_n \in x_nL_n\cap \RFPM$ such that
 $\pi(z_n v_nU)\not\subset \partial\core\M$ for all $n$. In particular, there exists $u_n\in U$ such that
$ z_n v_nu_n\in \core_{\eta/3}(M)$. 
Since $\core_{\eta/3}(M) $ is a compact subset of $F^*$, 
$z_nv_nu_n=z_nu_nv_n$ converges to a point in $F^*$, finishing the proof.
\end{proof}

\begin{lem}\label{lem.epsiloncore2}\label{onev} Let  $x_0L$ be a closed orbit with $x_0\in \RFM$ and $L\in\cal{L}_U$.
Suppose that $E$ is a closed $U$-invariant subset containing $x_0L v_n \cap\RFPM $ for some sequence $v_n \to \infty$ in $(L\cap N)^\perp$.
If $x_0\in F^*$ or $x_0L v_n \cap \RFPM \cap F^*\ne \emptyset$ for all $n$,
then there exist $y\in \RFM \cap F^*$ and a one parameter subgroup $V\subset (L\cap N)^\perp$ such that
$$E\supset y(L\cap N) V.$$ 
\end{lem}
\begin{proof}
Note that
\begin{equation*}
(x_0L v_n \cap\op{RF}_+\M) (v_n^{-1}Av_n)\subset E.
\end{equation*}
By Lemma \ref{lem.epsiloncore3}, there exists
\begin{equation*}
y\in F^*\cap \limsup\limits_{n\to\infty}(x_0L v_n \cap\op{RF}_+\M) .
\end{equation*}
Since $y\in F^*\cap \RFPM\subset \RFM\cdot  U$, we may assume $y\in F^*\cap \RFM$ by modifying $y$ using an element of $U$.
Note that $\liminf \limits_{n\to\infty}(x_0L v_n \cap\op{RF}_+\M) \supset y (L\cap N)$, passing to a subsequence.
Since $\limsup\limits_{n\to\infty}
(v_n^{-1}Av_n)
$ contains a one-parameter subgroup $V\subset (L\cap N)^\perp$ by Lemma \ref{vAv}, we obtain that $y(L\cap N)V\subset E$.
\end{proof}

\begin{lem}\label{aplus}
If $yLv_0\cap \RFM \cap F^* \ne \emptyset$ for some $v_0\in N$ and $L\in \mathcal L_U$, then $yLv\cap F^*\cap \RFM\ne\emptyset$ for all $v\in Av_0A$.
\end{lem}
\begin{proof}
Let $y_0:=y\ell v_0\in yL v_0\cap F^*\cap \RFM$, and $v=a v_0 b\in Av_0A$. 
Then $(y\ell a^{-1}) v= y\ell v_0b \in F^*\cap \RFM$ as $F^*\cap \RFM$ is $A$-invariant. Since $y\ell a^{-1}v\in yLv$, the claim is proved.
\end{proof}

\begin{lem}\label{onev2} Let  $x_0L$ be a closed orbit with $x_0\in \RFM$ and $L\in\cal{L}_U$.
Suppose that $E$ is a closed $AU$-invariant subset containing $x_0L v\cap\op{RF}_+\M $ for some non-trivial element $v\in (L\cap N)^\perp$.
If $x_0\in F^*$ or $x_0L v\cap\op{RF}\M\cap F^*  \ne \emptyset$, then
 there exist $y\in F^*\cap \RFM$ and a one parameter subgroup $V\subset (L\cap N)^\perp$ such that
 $$E\supset y(L\cap N) V A.$$
\end{lem}
\begin{proof}
Since $X$ is $A$-invariant, we get
\begin{equation*}\label{eq.h}
(x_0L\cap\op{RF}_+\M)AvA \subset E.
\end{equation*}
Choose a sequence $v_n:=a_{n}va_n^{-1}\in AvA$ tending to $\infty$.
Note that either $x_0\in F^*$ or for all $n$,
$x_0L v_n\cap\op{RF}\M\cap F^*\ne \emptyset$ by Lemma  \ref{aplus}.
 Therefore the claim follows from Lemma \ref{onev}. \end{proof}

\section{Limits of unipotent blowups}\label{s:un}
Let
 $\M$ be a convex cocompact hyperbolic manifold with Fuchsian ends and  fix $k>1$ as given by Proposition \ref{defk}.
In the whole section, we fix a non-trivial connected subgroup $U< N$. For a given sequence $g_i\to e$, and a sequence
of $k$-thick subsets $\mathsf T_i$ of a one-parameter subgroup $U_0<U$,
we study the following set
$$\limsup \mathsf T_i g_i U$$
under certain conditions on the sequence $g_i$.
The basic tool used here is the so-called {\it quasi-regular map} associated to the sequence 
$g_i$ introduced in the work of Margulis-Tomanov \cite{MT} to study the object 
$\limsup U_0g_iU$ in the finite volume case. For our application, we need a somewhat more precise information on
the shape of the set $\limsup U_0g_iU$ as well as $\limsup \mathsf T_i g_i U$ than discussed in \cite{MT}. 

 Let $U^\perp$ denote the orthogonal complement of $U$ 
in $N\simeq\br^{d-1}$ as defined  in section \ref{s:ba}.
Recall from \eqref{conetwo}
 that $$\op{N}(U)=AN\op{C}_1(U)\op{C}_2(U)$$
  where $ \op{C}_1(U)=\op{C}(H(U))$ and $ \op{C}_2(U)=H(U)\cap M \cap \op{C}(U^\perp) $.
  Since $\op{N}(U)$ is the identity component of $\op{N}_G(U)$, for a sequence $g_i\to e$,
  the condition $g_i\in \op{N}_G(U)$ means $g_i\in \op{N}(U)$ for all sufficiently large $i\gg 1$.
 Note that the product $AU^\perp \op{C}_2(U)$ is a connected subgroup of $G$, since
$\op{C}_2(U)$ commutes with $U^\perp$, and $A$ normalizes $U^\perp \op{C}_2(U)$.

\begin{lemma}\label{lem.QR}
 For a given sequence $g_i\to e$ in $G-\op{N}(U)$,  there exists a one-parameter subgroup $U_0<U$ 
  such that the following holds;
for any given sequence of $k$-thick subsets $\mathsf T_i\subset U_{0}$,  there exist
sequences $t_i\in \mathsf T_i$,
and $u_i\in U$ such that  as $i\to \infty$, $$u_{i}g_iu_{t_i}\to \alpha$$ for some non-trivial element $\alpha\in AU^\perp \op{C}_2(U)-\op{C}_2(U)$.
Moreover, $\alpha$ can be made arbitrarily close to $e$.
\end{lemma}
\begin{proof} 
Set $L:=AU^\perp MN^+$. Note that
$$\op{N}(U)\cap L=AU^\perp \op{C}_1(U)\op{C}_2(U)$$
and that the product map from $ U\times L$ to $G$ is a diffeomorphism onto a Zariski open neighborhood
 of $e$ in  $G$.

Following \cite{MT}, we will construct a quasi-regular map
\begin{equation*}
\psi : U\to  \op{N}(U)\cap L\end{equation*}
associated to the sequence $g_i$. Except for a Zariski closed subset of $U$,
 the product $g_i u$ can be written as an element of $ UL$ in a unique way. We denote by $\psi_i(u)\in L$ its $L$-component
 so that $$g_i u\in U \psi_i (u).$$ By Chevalley's theorem, there exists an $\br$-regular representation $G\to \op{GL}(W)$ with a distinguished point $p\in W$ such that $U=\op{Stab}_G(p)$. Then $pG$ is locally closed, and
\begin{equation*}
\op{N}_G(U)=\{g\in G : p gu =p g\text{ for all }u\in U\}.
\end{equation*} For each $i$, the map $\tilde \phi_i : U\to W$ defined by $$\tilde \phi_i(u)=p g_i u$$ is a polynomial map in $U=\br^{m}$
of degree uniformly bounded, and $\tilde \phi_i (e)$ converges to $ p$ as $i\to \infty$.
As $g_i\not\in\op{N}_G(U)$, $\tilde \phi_i$ is non-constant.
Denote by $B(p, r)$ the ball of radius $r$ centered at $p$, fixing a norm $\norm{\cdot}$ on $W$.
Since $pG$ is open in its closure, we can find $\la_0 >0$ such that
\begin{equation}\label{eq.m1}
B(p,\la_0)\cap \cl{pG}\subset p G.
\end{equation}
Without loss of generality, we may assume that $\la_0=2$ by renormalizing the norm.
Now define
\begin{equation*}
\la_i:=\sup\{\la\ge 0: \tilde \phi_i(B_U(\la))\subset B(p,2)\}.
\end{equation*}
Note that $\la_i<\infty$ as $\phi_i$ is nonconstant, and $\la_i\to\infty$ as $i\to\infty$, as $g_i\to e$.
We define $\phi_i:U\to W$ by
$$\phi_i (u):=\tilde \phi_i(\lambda_i u).$$
This forms an equi-continuous family of polynomials on $U$. Therefore, after passing to a subsequence,
 $\phi_i$ converges to a non-constant polynomial $\phi$ uniformly on every compact subset of $U$.
 Moreover $\sup\{ \|\phi (u) -p \|: u\in B_U(1)\}=1$, $\phi(B_U(1))\subset pL$, and $\phi(0)=p$.
Now the following map $\psi$ defines  a non-constant rational map defined on a Zariski open dense neighborhood
of $\mathcal U$ of $e$ in $U$:
 $$\psi:=\rho_L^{-1} \circ \phi$$
where $\rho_L$ is the restriction to $L$ of the orbit map $g\mapsto p. g$.
 We have  $\psi(e)=e$ and
$$\psi(u)=\lim _i \psi_i(\la _i u)$$
where  the convergence is uniform on compact subsets of $\mathcal U$
and $$\psi(u)\in L\cap \op{N}(U) =AU^\perp \op{C}_1(U)\op{C}_2(U).$$
Since $\psi$ is non-constant,
 there exists a one-parameter subgroup $U_0<U$ such that $\psi|_{U_0}$ is non-constant. Now let $
 \mathsf T_i$ be a sequence of $k$-thick sets in $U_{0}\simeq \br$.
Then $\mathsf T_i/\la_i$ is  also a $k$-thick set, and so is
\begin{equation*}
\mathsf T_\infty\coloneqq\limsup\limits_{i\to\infty}\left(\mathsf T_i/\la_i\right)\subset U_0.
\end{equation*}
Finally, for all $t\in \mathsf T_\infty$, there exists a sequence $t_i\in \mathsf T_i$ such that $t_i/\la_i\to t$ as $i\to\infty$
(by passing to a subsequence). Since $\psi_i\circ \la_i\to \psi$ uniformly on compact subsets,
\begin{equation*}
\psi(t)=\lim_{i\to\infty} (\psi_i\circ \la_i )\left( {t_i}/{\la_i}\right)=\lim_{i\to \infty} \psi_i(t_i).
\end{equation*}
By the definition of $\psi_i$, this means
that there exists $u_i\in U$ such that 
\begin{equation*}
\psi(t)=\lim\limits_{i\to\infty}u_{i}g_iu_{t_i}.
\end{equation*}
Since $\psi|_{U_0}$ is a non-constant continuous map, and 
an uncountable set $\mathsf T_\infty$ accumulates on $0$, the image $\psi(\mathsf T_\infty)$ contains
a non-trivial
 element $\alpha$ of $AU^\perp \op{C}_1(U)\op{C}_2(U)$ which can be taken arbitrarily close to $e$.

We now claim that  if $\alpha$ is sufficiently close to $e$, then it belongs to
 $  AU^\perp \op{C}_2(U)$. 
Consider $H'(U):=H(U) \op{C}_1(U)$, and let $\mathfrak h$ denote its Lie algebra.
Now for all $i$ large enough, using the decomposition $\frak g=\mathfrak h \oplus \mathfrak{h}^\perp$
in \eqref{perp},
we can write $g_i=c_i  d_i r_i$ where $c_i\in \op{C}_1(U)$, $d_i\in H(U)$ and  $r_i\in\exp \frak h^\perp$, 
Since $c_i$ commutes with $U$,
we can write $$u_{i} g_i u_{t_i} =(u_i u_{t_i}) c_i (u_{t_i}^{-1} d_i u_{t_i})(u_{t_i}^{-1} r_i u_{t_i}).$$
On the other hand,  
we have $$\lim_i pu_{i} g_i u_{t_i} =\lim_i pc_i (u_{t_i}^{-1} d_i u_{t_i})(u_{t_i}^{-1} r_i u_{t_i})
= p \alpha.$$
Since $c_i \to e$, $u_{t_i} d_i u_{t_i}^{-1}\in H(U)$, and 
$u_{t_i} r_i u_{t_i}^{-1} \in \exp \frak h^\perp $, it follows that 
both sequences $u_{t_i} d_i u_{t_i}^{-1}$ and $u_{t_i} r_i u_{t_i}^{-1}$ must 
converge, say to  $h\in H(U)$ and to $q
\in \exp \frak h^\perp$, respectively. 
Hence $\alpha= h q$ by replacing $h$ by $uh$ for some $u\in U$.
On the other hand, 
we can write $\alpha = av c_1c_2 \in AU^\perp \op{C}_1(U)\op{C}_2(U) $.
So  $hq= avc_1c_2$. Note that $c:=c_1c_2\in \op{C}(H(U)) H(U)=H'(U)$. We get
\be\label{appear}(a^{-1} h c^{-1}) (c q c^{-1})=v.\ee
Now, when $\alpha$ is sufficiently close to $e$, all elements appearing in \eqref{appear}
are also close to $e$. Recall that the map $H'(U)\times \frak h^\perp\to G$ given by $(h',X)\to 
h' \exp X$ is a local diffeomorphism onto a neighborhood of $e$.
 Since  $(a^{-1} h c^{-1})\in H'(U)$, and
$cqc^{-1}, v \in \exp \frak h^\perp$, we have $a^{-1}h c^{-1}=e$ and
$c q c^{-1}=v$ for $\alpha$ sufficiently small.
In particular,
$$a^{-1}h c_2^{-1} =c_1^{-1} \in H(U)\cap \op{C}(H(U))=\{e\}.$$ Hence $c_1=e$. It follows that
$\alpha\in  AU^\perp\op{C}_2(U) ,$ as desired.

We further claim that we can choose $\alpha$ outside of $\op{C}_2(U)$.
As $\op{C}_2(U)$ is a compact subgroup, we can choose a
$\op{C}_2(U)$-invariant Euclidean norm $\norm{\cdot}$ on $W$.
If $\alpha=\psi(t)\in \op{C}_2(U)$ for some $t\in \mathsf T_\infty\subset U_0$, then
 $t$ is one of finitely many solutions of the polynomial equation $\norm {\phi(t)}^2=\norm p^2$.
 Therefore, except for finitely many $t\in \mathsf T_\infty$, $\alpha=\psi(t)\in AU^\perp \op{C}_2(U)-\op{C}_2(U)$.
This finishes the proof.
\end{proof}

The following lemma is similar to Lemma \ref{lem.QR}, but here we consider the case when $U$ is
the whole horospherical subgroup $N$. In this restrictive case, the limiting element can be taken inside $A$.

\begin{lemma}\label{lem.fullhoro}  Let  $\mathsf T_i\subset N$ be a sequence of $k$-thick subsets in the sense that
for any one-parameter subgroup $U_0<N$, $\mathsf T_i\cap U_0$ is a $k$-thick subset of $U_0\simeq \br$.
 For any sequence $g_i\to e$ in $G-\op{N}_G(N)$,
  there exist $t_i\to\infty$ in $\mathsf T_i$ and $u_{i}\in N$ such that $$u_{i}g_iu_{t_i}\to a$$ for some non-trivial  element
  $a\in A$.
Moreover, $a$ can be chosen to be arbitrarily close to $e$.
\end{lemma}
\begin{proof}
We first consider the case when $g_i$ belongs to the opposite horospherical subgroup $ N^+$.
We will use the notations $u^+$ and $u^-$ defined in Section \ref{sec.notation}.
 Write $g_i=\exp u^+(w_i)$ for some $w_i\in\bb{R}^{d-1}$.
For $x\in\bb{R}^{d-1}$, set $u_x:=\exp u^-(x) \in N$. 
Let $\epsilon>0$ be arbitrary.
Since $\mathsf T_i$ is a $k$-thick subset of $N$, there exists $\alpha_i\in \bb R$ such that  $\alpha_i w_i\in \mathsf T_i$ and
$$ \epsilon<\frac{|\alpha_i|\norm{w_i}^2}{2}<k\epsilon.$$
Setting
$u_{x_i}:=u_{\alpha_i w_i}\in \mathsf T_i$ and
$y_i:=-\alpha_i w_i \left(1+\tfrac{\alpha_i\norm{w_i}^2}{2}\right)^{-1}$,
we compute: \begin{equation*}
u_{y_i}g_iu_{x_i}=
\left(
\begin{array}{ccc}
\left(1+\frac{\alpha_i\norm{w_i}^2}{2}\right)^{-2}&0&0\\
\left(1+\frac{\alpha_i\norm{w_i}^2}{2}\right)^{-1}w_i&\op{I}_{d-1}&0\\
-\frac{\norm{w_i}^2}{2}&-\left(1+\frac{\alpha_i\norm{w_i}^2}{2}\right)w_i^t&\left(1+\frac{\alpha_i\norm{w_i}^2}{2}\right)^2
\end{array}
\right).
\end{equation*}
The condition for the size of $\alpha_i$ guarantees that, by passing to a subsequence, the sequence $ u_{x_i}g_iu_{y_i}$ converges to an element 
$$\text{diag} (\alpha, \op{I}_{d-1},\alpha^{-1}) \in A,\quad
\text{for } \alpha \in [(1-\epsilon)^{-2},(1-k\epsilon)^{-2}]\cup [(1+k\epsilon)^{-2}, (1+\epsilon)^{-2}]$$
as $i\to\infty$.
This proves the claim when $g_i\in N^+$.

Since the product map $A\times M\times N^+\times N\to G$ is a diffeomorphism onto a Zariski-open neighborhood of $e$ in $G$,
 we can write $g_i=a_im_iu_i^+ u_i^-$ for some $a_i\in A$, $m_i\in M$, $u_i^+\in N^+$ and $u_i^-\in N$ all of which converge to $e$ as $i\to\infty$. By the previous case, we can find $u_{t_i}\in \mathsf T_i$ and $u_{i}\in N$ such that 
$ u_i u_i^+ u_{t_i} $ converges to a non-trivial element $a\in A$.
Let $\til u_i:=(a_im_i)u_i(a_im_i)^{-1}\in N$.
Then $\til u_ig_iu_{t_i}=a_im_i u_iu_i^+ u_i^- u_{t_i}= a_im_i( u_iu_i^+u_{t_i})u_i^-\to a$ as $i\to\infty$, proving the claim.
\end{proof}

\begin{lemma}\label{lem.QR2}
Let $L$ be any connected reductive subgroup of $G$ normalized by $A$. Let $U_0$ be a one-parameter subgroup of $L\cap N$.
Let  $\mathsf T_i\subset U_{0}$ be a sequence of $k$-thick subsets.
For a given sequence $r_i\to e$ in $\exp (\mathfrak l^\perp) -\op{N}(U_0)$,
there exists a sequence $t_i\in \mathsf T_i$ such that as $i\to \infty$,
$$u_{t_i}^{-1} r_iu_{t_i} \to v$$ for some non-trivial element $v\in (L\cap N)^\perp$, and  $v$ can be chosen arbitrarily close to $e$.
Moreover, for all $n$ large enough, we can make $v$ so that \begin{equation*}
n\leq\norm{ v}\leq 2k^2 n.
\end{equation*}
\end{lemma}
\begin{proof} Without loss of generality, by Proposition \ref{str},
we may assume that $L_{nc}=H(U)$ for $U=U_k=\br^k$ some $k\ge 1$ and
$U_0:=\br e_1$. 
We write $r_i=\op{exp}(q_i)$ where $q_i\to 0$ in $\mathfrak l^\perp$. Using the notations introduced in section \ref{sec.notation} and setting $\frak u^\perp=\op{Lie} (U^\perp)=\br^{d-1-k}$, 
we can write
$$q_i=u^-(x_i)+u^+(y_i)+m(C_i)$$ where $x_i\in \frak u^\perp$, $y_i\in (\frak u^\perp)^t$, and
$
C_i=\begin{pmatrix}
0_k &B_i\\
-B_i^t&A_i
\end{pmatrix} $
is a skew symmetric matrix, all of which
converge to $0$ as $i\to\infty$.
We consider $U_0=\br e_1$ as $\{u_s=s e_1 \in \br^{d-1}\}$
and define  the map $\psi_i : \br  \to \mathfrak l^\perp$ by
\begin{equation*}
\psi_i(s)=u_s^{-1}q_i u_s \quad\text{ for all $s\in \br$};
\end{equation*}
this is well-defined since $\mathfrak l^\perp$ is $\op{Ad}(L)$-invariant.
Then a direct computation shows
\begin{equation}\label{eq.matrixcomputation}
\psi_i(s)=u^-(x_i+ s B_i^t e_1 + s^2 y_i/2)+u^+(y_i)+m(\til{C}_i)
\end{equation}
where $\til C_i$ is a skew-symmetric matrix of the form
\begin{equation*}
\til{C}_i=\left(
\begin{array}{cc}
0_k&B_i+s e_1 y_i^t\\
-B_i^t-s y_i e_1^t&A_i
\end{array}
\right).
\end{equation*}
Since $r_i\not\in\op{N}(U_0)$, it follows that either $y_i\ne 0$ or 
$y_i=0$ and $B_i^te_1\ne 0$.
Hence $\psi_i$ is a non-constant polynomial of degree at most $2$, and $\psi_i(0)\to 0$.
Let  $\la_i\in \br$ be defined by 
$$\la_i=\sup\{ \la >0: |\psi_i [-\la, \la] | \le 1\}.$$
Then $0<\la_i< \infty$ and $\la_i \to \infty$.
Now the rescaled polynomials 
 $\phi_i=\psi_i\circ \la_i: \br \to \mathfrak l^\perp$
 form  an equicontinuous family of polynomials of degree at most $2$ and
 $\lim_{i\to \infty} \phi_i(0)= 0$. Therefore
$\phi_i$  converges
 to a  polynomial $\phi : \br \to \mathfrak l^\perp$ uniformly on compact subsets.
 Since $\phi(0)=0$ and $\sup \{|\phi(\la)|: \la\in [-1,1]\}=1$, $\phi$ is a non-constant polynomial.
  From (\ref{eq.matrixcomputation}), it can be easily seen  that 
 $\op{Im}(\phi)$ is contained  $\op{Lie}(N)\cap \frak l^\perp$, by considering
  the two cases of $y_i\ne 0$, and $y_i=0$ and $B_i^te_1\ne 0$ separately.
For a given sequence $\mathsf T_i$ of $k$-thick subsets of $U_0$, set
\begin{equation*}
\mathsf T_\infty:=\limsup\limits_{i\to\infty}(\mathsf T_i/\la_i),
\end{equation*}
which is also a $k$-thick subset of $U_0$.

Let $s\in \mathsf T_\infty$.  By passing to a subsequence,
there exists $t_i\in \mathsf T_i$ such that $t_i/\la_i\to s$ as $i\to\infty$.
As $\phi_i\to \phi$ uniformly on compact subsets, it follows that
$$
\phi(s)=\lim_{i\to \infty}\psi_i (\la_i \cdot t_i/\la_i)= 
\lim_{i\to\infty} {u_{t_i}^{-1}}q_i {u_{t_i}}.
$$
Since $\mathsf T_\infty$ accumulates on $0$, so does $\phi(\mathsf T_\infty)$.
Taking the exponential map to each side of the above, the first part of the lemma follows. 

The second part of the lemma holds  by applying Lemma \ref{pol} below for the non-constant polynomial $p(s)=\norm{\phi(s)}^2$ of degree
at most $4$.
\end{proof}

\begin{lemma}\label{pol}
If $p\in \br [s]$ is a polynomial of degree $\delta \ge 1$  and $\mathsf T\subset \br$ is a $k$-thick subset,
then $p(\mathsf T)$ is $2k^{\delta}$-thick at $\infty$.
\end{lemma}

\begin{proof}
Let $C$ be the coefficient of $s^{\delta}$ term of the polynomial $p$. Then there exists $s_0>1$ such that
$\frac{1}{\sqrt 2}\le \frac{|p(s)|}{|Cs^{\delta}|}\le \sqrt 2$ for all $|s|>s_0$.
Let $r>\frac{|C|s_0^{\delta}}{\sqrt 2}$.
Since $\mathsf T$ is $k$-thick, there exists $t\in\mathsf T$ such that $(\sqrt2 r/|C|)^{1/\delta}<|t|<k(\sqrt2 r/|C|)^{1/\delta}$.
We compute that
$r\le |p(t)| \le  2k^{\delta} r$, proving the claim.
\end{proof}

\section{Translates of relative $U$-minimal sets}\label{s:tr}
Assume that
 $\M$ is a convex cocompact hyperbolic manifold with Fuchsian ends and  fix $k>1$ as given by Proposition \ref{defk}.
In this section, we fix a non-trivial connected closed subgroup $U<N$. Unless mentioned otherwise,
we let $R$ be a compact $A$-invariant subset of $\RFM$  such that
for every $x\in R$, and for any one-parameter subgroup $U_0=\{u_t\}$ of $U$, the following set
$$\{t\in\br: xu_t\in R\}$$
is $k$-thick.
In practice, $R$ will be either $\RFM$
or a compact subset of the form $\RFM\cap F_{H(U)}^*\cap X$  for a closed $H(U)$-invariant subset $X$.

The main aim of this section is to prove Propositions \ref{YLY} and \ref{7.6} using the results of section \ref{s:un}.
The results in this section are needed in the step of finding a closed orbit in a given $H(U)$-orbit closure of an $\RFM$-point.

\begin{Def}\begin{itemize}
\item  A $U$-invariant closed subset $Y\subset \Gamma\ba G$ is $U$-minimal 
if  $yU$ is dense in $Y$ for any $y\in Y$.

\item  A $U$-invariant closed subset $Y\subset \Gamma\ba G$ is $U$-minimal with respect to $R$
if $Y\cap R\ne \emptyset$ and for any $y\in Y\cap R$, $yU$ is dense in $Y$.
\end{itemize}
\end{Def}
A $U$-minimal subset may not exist, but a $U$-minimal subset with respect to a compact subset $R$ always exists by Zorn's lemma.
In this section, we study how to find an additional invariance of $Y$ beyond $U$
under certain conditions.

\begin{lem} \label{bbo}  Let $Y\subset \Gamma\ba G $ be a $U$-minimal subset with respect to $R$.
For any $y\in Y\cap R$, there exists a sequence $u_n\to \infty$ in $U$ such that $yu_{n}\to y$.
\end{lem}
\begin{proof} The set $Z:=\{z\in Y: yu_n \to z \text{ for some $u_n\to \infty$ in $U$}\}$
is $U$-invariant and closed. By the assumption on $R$, there exists $u_n\to \infty$ in $U$ such that
$yu_n\in Y\cap R$. Since $Y\cap R$ is compact, $yu_n$ converges to some $z\in Y\cap R$, by passing to a subsequence.
Hence $Z$ intersects $Y\cap R$ non-trivially. Therefore $Z=Y$, by the $U$-minimality of $Y$ with respect to $R$.
 \end{proof}
 
A subset $S$ of a topological space is said to be {\it locally closed} if $S$ is open in its closure $\overline S$.
\begin{lemma}\label{lem.normalizerorbit}
Let $Y$ be a $U$-minimal subset of $\Ga\ba G$ with respect to $R$, and $S$ be a closed  subgroup of $\op{N}(U)$ containing $U$.
For any $y_0\in Y\cap R$, the orbit $y_0S$ is not locally closed.
\end{lemma}
\begin{proof}
Suppose that $y_0S$ is locally closed for some $y_0\in Y\cap R$.
Since $Y$ is $U$-minimal with respect to $R$,  there exists $u_n\to\infty$ in $U$ such that $y_0u_n\to y_0$ by Lemma \ref{bbo}.
We may assume that $y_0=[e]$ without loss of generality.
Since $y_0S$ is locally closed, $y_0S$ is homeomorphic to $(S\cap\Ga)\ba S$ (cf. \cite[Theorem 2.1.14]{Zi}). 
Therefore there exists $\delta_n\in S\cap\Ga$ such that $\delta_nu_n\to e$ as $n\to\infty$.
Since $\op{N}(U)=AN\op{C}_1(U) \op{C}_2(U)$, writing $\delta_n=a_nr_n$ for $a_n\in A$ and $r_n\in N\op{C}_1(U) \op{C}_2(U)$, it follows that  $a_n\to e$.
On the other hand, note that $a_n$ is non-trivial as $\Ga$ does not contain any elliptic or parabolic element.
This is a contradiction, as there exists a positive lower bound for the translation lengths of elements of $\Ga$, which is given by
the minimal length of a closed geodesic in $M$.
\end{proof}

In the rest of this section,
we use the following notation:
 $$H=H(U), \; H'=H'(U), \; \text{and} \; F^*=F^*_{H(U)}.$$

\begin{lemma}\label{cor.gn}   For every $U$-minimal subset $Y\subset \Gamma\ba G$ with respect to $\RFM$ such that $Y \cap F^*\cap \RFM\ne \emptyset$, and for any $y_0\in Y\cap F^*\cap \RFM$, there exists a sequence $g_n\to e$ in $G-\op{N}(U)$ such that $y_0g_n\in Y\cap \RFM$ for all $n$.
\end{lemma}
\begin{proof}
Let $y_0\in Y\cap F^*\cap \RFM$. As $Y=\overline{y_0U}$, $Y\subset \RFPM$.
Using Lemma \ref{lem.R1} and the fact that $F^*$ is open, we get that
there exists an open neighborhood $\cal{O}$ of $e$ such that
\be\label{yyy} y_0 \cal O\cap Y\subset Y\cap F^*\subset Y\cap \RFM \cdot U.\ee
Without loss of generality, we may assume that the map $g \mapsto y_0g\in \Gamma\ba G$ is injective on $\cal O$, by shrinking $\cal O$ if necessary.
We claim that there exists $g_n\to e$  in $G-\op{N}(U)$  such that $y_0g_n  \in Y\cap F^*$.
Suppose not. Then there exists a neighborhood $\cal{O}'\subset \cal O$ of $e$ such that
\begin{equation}\label{sn} 
y_0\cal{O}'\cap Y\subset y_0\op{N}(U).
\end{equation}
 Set $$S:=\{g\in \op{N}(U) : Yg= Y\}$$
 which is a closed subgroup of $\op{N}(U)$ containing $U$.
We will show that $y_0S$ is locally closed;
this contradicts Lemma \ref{lem.normalizerorbit}.
We first claim that \be\label{yss} y_0\cal{O}'\cap Y\subset y_0S.\ee
If $g\in \cal O'$ such that $y_0g\in Y$, then $g\in \op{N}(U)$. Therefore
$\overline{y_0g U}= \overline{y_0U} g=Yg\subset Y$.
Moreover, $Yg\cap \RFM \ne \emptyset$ by \eqref{yyy}.
Hence $Yg=Y$, proving that $g\in S$.
Now, \eqref{yss} implies that  $y_0S$ is open in $Y$.
On the other hand, since $U\subset S$, we get $Y=\cl{y_0S}$.
Therefore, $y_0S$ is locally closed. 

Hence we have $g_n\to e$  in $G-\op{N}(U)$  such that $y_0g_n  \in Y\cap F^*$. Since
$y_0g_n\in F^*\cap \RFPM$ converges to $y_0\in F^*\cap \RFM$, by Lemma \ref{xfss},
there exists a sequence $u_n\to e$ in $U$ such that
$y_0g_nu_n\in \RFM$.  Therefore, by replacing $g_n$ with $g_nu_n$, this finishes the proof.
\end{proof}

\begin{lem}\label{ss}
Let $Y$ be a $U$-minimal subset with respect to $R$, and let $W$ be a connected closed subgroup of $\op{N}(U)$.
Suppose that there exists a sequence $\alpha_i\to e$ in $W$ such
that $Y\alpha_i\subset Y$.
Then there exists a one-parameter subsemigroup $S<W$ such that $YS\subset Y$.

Moreover if $W_0$ is a compact  Lie subgroup of $W$ and $\alpha_i\in W-W_0$ for all $i$, then
$S$ can be taken so that $S\not\subset W_0$.
\end{lem}
\begin{proof} The set $S_0=\{g\in W: Yg\subset Y\}$ is a closed subsemigroup of $W$.
Write $\alpha_i=\exp \xi_i$ for some $\xi_i \in \op{Lie}(W)$.
Then the sequence $v_i:=\|\xi_i\|^{-1} \xi_i$ of unit vectors has a limit, say, $v$.
It suffices to note that $S:=\{\exp (tv): t\ge 0\}$ is contained in the closure of the subsemigroup generated by $\alpha_i$'s.
Now suppose that $\alpha_i\in W - W_0$.
Set $M_0:=\{g\in W_0 : Yg=Y\}$.
This is a closed Lie subgroup of $W_0$.
Write $\op{Lie} W=\frak m_0\oplus\frak m_0^\perp$ where $\frak m_0=\op{Lie} M_0$.
By modifying $\alpha_i$ by elements of $M_0$, we may assume $\alpha_i=\exp \xi_i$ for $\xi_i\to 0$ in $\frak{m}_0^{\perp}$.
Letting $v\in \frak m_0^\perp$ be a limit of $\xi_i/\norm{\xi_i}$, it remains to check 
$v\notin W_0$. 
Suppose not.  Since $W_0$ is compact, we have
$\{\exp{tv}:t\ge 0\}=\exp \br v$. Hence for all $t\ge 0$,  $ Y\exp tv\subset Y$ as well as $Y\exp (-tv)\subset Y$. 
Therefore
$Y \exp t v =Y$. Hence $\exp v\in M_0$. This is a contradiction, since $v\in \frak m_0^\perp$.
\end{proof}

\begin{proposition}[Translate of $Y$ inside of $Y$]\label{prop.YLY}\label{YLY}
Let $Y$ be a $U$-minimal set of $\Ga\ba G$  with respect to $\RFM$ such that $Y\cap F^*\cap \RFM\ne \emptyset$.

Then there exists an unbounded one-parameter subsemigroup $S$ inside 
the subgroup $AU^\perp \op{C}_2(U)$ such that 
\begin{equation*}
YS\subset Y.
\end{equation*}

\end{proposition}
\begin{proof}
Choose $y_0\in Y\cap \RFM \cap F^*$.
By Lemma \ref{cor.gn}, there exists $g_i\to e$ in $G-\op{N}(U)$  such that $y_0g_i\in Y\cap \RFM$.
Let $U_0=\{u_t\}$ be a one-parameter subgroup of $U$ as given by
 Lemma \ref{lem.QR}, with respect to the sequence $g_i$.
 
Let
\begin{equation*}
\mathsf T_i:=\{u_t\in U_0 : y_0 g_i  u_t\in Y\cap \RFM\}
\end{equation*}
which is a $k$-thick subset of $U_0$.
By Lemma \ref{lem.QR},
  there exist sequences
$u_{t_i}\to\infty$ in $\mathsf T_i$,
and $u_{i}\in U$ such that $$ u_ig_iu_{t_i} \to  \alpha$$ for some element $\alpha\in AU^\perp \op{C}_2(U)-\op{C}_2(U)$.  Note that $y_0g_i u_{t_i}\in Y\cap \RFM$ converges to some $y_1\in Y\cap \RFM$ by passing to a subsequence.
Hence as $i\to \infty$,
$$y_0 u_i^{-1} = y_0 g_i u_{t_i}  (u_{i}g_i  u_{t_i} )^{-1}\to y_1 \alpha^{-1}.$$

So $y_1 \alpha^{-1}\in Y$, and hence $Y \alpha^{-1}\subset Y$, since $y_1\in Y\cap \RFM$.
Since $\alpha$
 can be made arbitrarily close to $e$ in Lemma \ref{lem.QR}, the claim follows from Lemma \ref{ss}.
\end{proof}
\begin{proposition}[Translate of $Y$ inside of $X$] \label{prop.R3}\label{YvX}
Let $X$ be a closed $H'$-invariant set such that $X\cap R \neq\emptyset$.
Let $Y\subset X$ be a $U$-minimal subset with respect to $R$, and
assume that there exists $y\in Y\cap R$ and a sequence $g_n\to e$ in $G-H'$ such that $yg_n\in X$ for all $n$.
Then there exists some non-trivial $v\in U^\perp$ such that
\begin{equation*}
Yv\subset X.
\end{equation*}
\end{proposition}
\begin{proof}
Let $\frak h$ denote the Lie algebra of $H'$.
We may write $g_n=r_nh_n$ where $h_n\in H'$ and $r_n\in \exp \frak h^\perp$.
By replacing $g_n$ with $g_nh_h^{-1}$, we may assume $g_n =r_n$.
If $r_n \in U^\perp$ for some $n$,
 then the claim follows since $y_0r_n \in X$ and hence $Yr_n\subset X$.
 Hence we assume that $r_n\notin U^\perp$ for all $n$. 
 We have from \eqref{perp}
 $$\frak h^\perp\cap \op{Lie}(\op{N}(U))=\op{Lie}U^\perp.$$ Hence
 $r_n \notin \op{N}(U)$ for all $n$.
Therefore
 there exists a one-parameter subgroup $U_0=\{u_t\} <U$ such that $r_n\notin \op{N}(U_0)$.
Let
\begin{equation*}
\mathsf T=\{t\in \br  : yu_t\in R\}.
\end{equation*}
Since $y\in R$, it follows that $\mathsf T$ is a $k$-thick subset of $\br$ by the assumption on $R$.
Hence, by Lemma \ref{lem.QR2}, there exists $t_n\in \mathsf T$ such that $u_{t_n}^{-1}r_nu_{t_n}\to v$ for some non-trivial $v\in U^\perp$.
Observe
\begin{equation*}
(yu_{t_n})(u_{t_n}^{-1}r_nu_{t_n})=yr_nu_{t_n}\in X.
\end{equation*}
Passing to a subsequence, $yu_{t_n}\to y_0$ for some $y_0\in Y\cap R$, and hence $y_0v\in X$.
It follows $Yv\subset X$.
\end{proof}
For a one-parameter subgroup $V=\{v_t:t\in \br\}$ and a subset $I\subset \br$,
the notation $V_I$ means the subset $\{v_t: t\in I\}$.
\begin{lemma}\label{lem.V_I}
Let $X$ be a closed $AU$-invariant set of $\Gamma\ba G$, and $V$ be a one-parameter subgroup of $U^\perp$.
Assume that $R:=X\cap\op{RF}\M\cap F^*$ is non-empty and compact.
If $x_0V_I\subset X$ for some $x_0\in R$ and a closed interval $I$ containing $0$, then $X$ contains a $V$-orbit of a point in $R$.
\end{lemma}
\begin{proof}
Choose a sequence $a_n\in A$ such that $\liminf_{n\to \infty} a_nV_Ia_n^{-1}$ contains a subsemigroup $V^+$  of $V$ as $n\to\infty$.
Then 
\begin{equation*}
(x_0a_n^{-1})(a_nV_Ia_n^{-1})=x_0V_Ia_n^{-1}\subset X.
\end{equation*}
 By passing to a subsequence, we have
 $x_0a_n^{-1}$ converges to some $ x_1\in \RFM$; so $x_1V^+\subset X$.
Since $R$ is compact, so is $\cl{x_0A}\cap F^*$, which implies that $x_1\in \cl{x_0A}\cap F^*$.
Since $x_1$ belongs to the open set $ F^*$, it follows $x_1v_s\in F^*$ for all sufficiently small $s\in \br$.
By Lemma \ref{lem.R1}, this implies that  $x_1v_sU\cap\op{RF}\M\neq\emptyset$ for some $s>0$ with $v_s\in V^+$.
Note that
\begin{equation*}
(x_1v_sU)(v_s^{-1}V^+)=x_1UV^+\subset X.
\end{equation*}
Choose $x_2\in x_1v_sU\cap\op{RF}\M\subset X\cap\op{RF}\M\cap F^*$. Then  $x_2(v_s^{-1}V^+)\subset X$.
Similarly as before, let $a_n\in A$ be a sequence such that $\liminf_{n\to \infty} a_n(v_s^{-1}V^+)a_n^{-1}= V$ and such that $x_2 a_n^{-1}$ converges to some $x_3\in R$. 
From
\begin{equation*}
(x_2a_n^{-1})(a_nv_s^{-1}V^+a_n^{-1})=x_2v_s^{-1}V^+a_n^{-1}\subset X,
\end{equation*}
we conclude  that
$x_3V\subset X.$ This finishes the proof.
\end{proof}

\begin{proposition}\label{prop.YVIX}\label{7.6}
Let $X$ be a closed $H'$-invariant set.
Assume  that $R:=X\cap F^*\cap\op{RF}\M$ is a non-empty compact set, and let $Y\subset X$ be a $U$-minimal subset with respect to $R$.
Suppose that there exists $y\in Y\cap R$ such
  that $X-yH'$ is not closed.
Then there exist an element $z\in R$ and a  non-trivial connected closed subgroup $V < U^\perp$ such that
$$z UV \subset X.$$ \end{proposition}
\begin{proof}
Since $X-yH'$ is not closed, there exists
 a sequence $g_n\to e$ in $G-H'$ such that $yg_n\in X$ for all $n\ge 1$.
By Lemma \ref{lem.V_I}, it suffices to find $x_0\in R$ and
 a one-parameter subgroup $V<U^\perp$ such that $x_0V_I\subset X$  for some interval $I<\br$ containing $0$.
It follows from Propositions  \ref{prop.YLY} and \ref{prop.R3} that 
$$Yv_0\subset X\quad\text{ and} \quad YS\subset Y$$
where $v_0\in U^\perp-\{e\}$ and $S$ is an unbounded one-parameter subsemigroup of
$AU^\perp \op{C}_2(U)$.
By Lemma \ref{rmk.1psg}, $S$ is either of the form
\begin{enumerate}
\item
$S=\{\exp(t\xi_V)\exp(t\xi_C) : t\ge 0\}$, or
\item
$S=\{\left(v\exp(t\xi_A)v^{-1}\right)\exp(t\xi_C) : t\ge 0\}$
\end{enumerate}
for some $\xi_A\in\op{Lie}(A)-\{0\}, \xi_C\in\op{Lie}(\op{C}_2(U)), \xi_V\in\op{Lie}(V)-\{0\},$ and $v\in U^\perp$.

{\noindent {\bf  Case $(1)$:}}
Since $X$ is $H'(U)$-invariant and $C_2(U)\subset H'(U)$, we may assume $YS\subset X$ with $\xi_C=0$; so the claim follows.

{\noindent {\bf  Case $(2)$:}} Set $$Y_0:=Y \op{C_2}(U).$$
It is easy to check that $Y_0$ is a $U\op{C}_2(U)$-minimal subset of $X$ with respect to $R$.
First suppose that  $v=e$. Let $A^+:=\{\exp(t\xi_A): t\ge 0\}$.
Since $YS\subset Y$ and $\xi_C\in\op{Lie}(\op{C}_2(U))$, it follows that
$Y_0 A^+\subset Y_0$. Choose $y\in Y\cap R$, and let $a_n\to \infty$ be a sequence in $A^+$.
Since $R$ is compact and $A$-invariant, $ya_n$ converges to some 
$
z_0\in R$
by passing to a subsequence.
Since $Y_0A^+\subset Y_0$, we have $z_0\in Y_0\cap R$.
Since $\liminf a_{-n}A^+=A$, 
we get $z_0A\subset Y_0$. Since $z_0A U \op{C}_2(U)=z_0 U \op{C}_2(U) A$, and $Y_0$ is $U \op{C}_2(U)$-minimal with respect to $R$,
we obtain
 $Y_0A\subset Y_0 .$
Since $v_0$ commutes with $\op{C}_2(U)$,
we also get $Y_0 v_0\subset X$.
Therefore
$Y_0Av_0\subset Y_0v_0\subset X$.
By the $A$-invariance of $X$, it follows $Y_0(Av_0A)\subset X$.
Since $Av_0A$ contains some $V^+$, the claim follows.

Next suppose $v\neq e$. Since $\op{C}_2(U)$ commutes with $v$,
it follows that $$Y_0 v A^+ v^{-1}\subset Y_0 .$$
Since  $X$ is $A$-invariant, we get
\begin{equation*}
Y_0(vA^+v^{-1})A\subset Y_0A\subset X.
\end{equation*}
Set $V:=\exp \br (\log v)$.
Since  $vA^+v^{-1} A$ contains $V_I$ for some interval $I$ containing $0$ for any subsemigroup $A^+$ of $A$.
we get $Y_0V_I\subset X$, finishing the proof.
\end{proof}

\section{Closures of orbits inside $\partial F$ and non-homogeneity}\label{s:cl}
Let $\M=\Ga\ba \bH^d$ be a convex cocompact hyperbolic manifold with non-empty Fuchsian ends.
Let $U$ be a connected closed subgroup of $\check H\cap N$,
$H:=H(U)$ and $\BFM$ be as in \eqref{eq.BFMDEF}.
Then $$\partial F= \BFM\cdot \check V^+ \cdot H'(U)\quad\text{and}\quad
\partial F\cap \RFM=\BFM \cdot \op{C}(H(U)).$$
In this section, we classify closures
of $xH(U)$ and $xAU$ for $x\in \partial F -\RFM$ (Theorem \ref{b-ratner}); they are never homogeneous.

\begin{theorem}\label{thm.ratner}
If $x=zc \in\BFM\cdot\op{C}(H(U))$ with $z\in \BFM$ and $c\in \op{C}(H(U))$.
Then \begin{enumerate}
\item $\cl{xU}=xL$ for some $L\in\cal Q_U$ contained in $c^{-1} \check H c$;
\item $\cl{xH(U)}=xL$ for some $L\in\cal L_U$ contained in $c^{-1} \check H c$, and  for any $y\in \mG(U, xL)$,
$\cl{yU}=xL$;
\item $\cl{xAU}=\cl{xH(U)}$. \end{enumerate}
\end{theorem}
\begin{proof} Since $x$ is contained in the compact homogeneous space $xc^{-1} \check H c$, 
the claims (1) and (2) are  special cases of Ratner's theorem  \cite{Ra2}, which were also proved by Shah independently \cite{Sh0}.
So we only need to discuss the proof of $(3)$. We show that $\cl{xAU}=xL$ where $L$ is given by (2).
If $U=L\cap N$, then the claim follows from Theorem \ref{tm}. Suppose that $U$ is a proper subgroup of $L\cap N$.
Since $\cl{xAU} (K\cap H(U))= \cl{xH(U)}=xL$ and $\mathscr{S}(U,xL)\cdot (K\cap H(U))$ is a proper subset of $x L$ (cf. Lemma \ref{lem.SHU}), there exists $y\in \cl{xAU}\cap \mG(U, xL)$. 
Hence (3) follows from (2). \end{proof}

\begin{lemma}\label{lem.HVPH} Let $V^+\subset N$ be a one-parameter subsemigroup which is not contained in $\check H$.
Then  $V^+H(U)$ is a closed subset of $G$.
\end{lemma}
\begin{proof}
Since the product map $A\times N\to AN$ is a diffeomorphism and $AN$ is closed, 
the product subset $AW$ is closed in $G$ for any closed subset $W$ of $N$.
Hence $AUV^+$ is a closed subset of $AN$.
We use Iwasawa decompositions  $H(U)=UA(K\cap H(U))$, and the fact that $AV^+=V^+A$
in order to write 
$V^+ H(U)=AUV^+(K\cap H(U))$.
Hence the conclusion follows from compactness of  $K\cap H(U)$.
\end{proof}

\begin{lemma}\label{lem.HVPH2} Let $V^+\subset N$ be as in Lemma \ref{lem.HVPH}.
If $g_i\in \check H$ is a sequence such that
$g_iv_ih_i$ converges for some $v_i\in V^+$ and $h_i\in H(U)$ as $i\to\infty$,
 then, after passing to a subsequence, there exists $p_i\in AU$ such that 
 $g_ip_i$ converges to an element of $\check H$ as $i\to\infty$.
\end{lemma}
\begin{proof}
We write $g_i=\til k_i\til a_i \til n_i \in  (K\cap\check H) A(N\cap\check H)$ and $h_i=u_ia_ik_i\in UA(K\cap H(U))$.
Since $K\cap \check H$ and $K\cap H(U)$ are compact, we may assume without loss of generality that
$\til k_i=k_i=e$ for all $i$.
Observe that
\begin{align*}
g_iv_i h_i&=\til a_i \til n_iv_iu_ia_i\\
&=\til a_ia_i(a_i^{-1} \til n_iu_ia_i)(a_i^{-1}v_ia_i)
\end{align*}
where $\til a_ia_i\in A$, $a_i^{-1} \til n_iu_ia_i\in N\cap\check H$, and $a_i^{-1}v_ia_i\in V^+$.
Since $g_iv_ih_i$ converges as $i\to\infty$
and
the product map $A\times (N\cap \check H)\times V^+\to G$ is an injective proper map, it follows that
all three sequences $\til a_ia_i$, $a_i^{-1} \til n_iu_ia_i$  and $a_i^{-1}v_ia_i$ are convergent as $i\to\infty$.
Noting that $$g_iu_ia_i=\til a_i \til n_iu_ia_i=\til a_ia_i(a_i^{-1} \til n_iu_ia_i),$$ 
it remains to set $p_i:=u_ia_i\in AU$ to finish the proof.
\end{proof}

For $z\in \BFM$, $\pi(z\check H\check V^+ \check H)=\pi(z\check H\check V^+ )$ is the closure of a Fuchsian end, of the form
 $S_0\times [0,\infty)$
where $S_0=\pi(z\check H)$. 
\begin{lemma} \label{lem-ra} Let $z\in \BFM$.
Let $zL$ be a closed orbit contained in $\BFM$ for some $L\in \mathcal L_U$ contained in $\check H$, 
and $V^+\subset N$ be a one-parameter subsemigroup such that $\check HV^+ = \check H\check V^+$.
Then both $zL V^+ H(U)$  and $zLV^+$ are closed.
\end{lemma}
\begin{proof} 
Without loss of generality, we assume $z=[e]$. 
Let $B$ denote the component of  $\Omega$ such that $\hull(\partial B)=\pi(\check H)$ for the projection map $\pi:G\to \bH^d$. 
Since  $\check HV^+ = \check H\check V^+$, we have  $\pi (\check HV^+\check H)= \hull \cl {B}$.
Note that  if $\gamma (\hull (B))\cap \hull (B)\ne \emptyset$ for $\gamma\in \Gamma$,
then $\gamma\in \check H\cap \Gamma=\op{Stab}_\Gamma (B)$.

Suppose that $\gamma_i \ell_i v_i h_i$ converges to some element $g\in G$ where $\gamma_i\in \Gamma$, $\ell_i\in L$, $v_i \in V^+$ and $h_i\in H(U)$.
Since $\pi(\gamma_i \ell_i v_i h_i )\in \Gamma \hull \cl B$, and $\Gamma \hull \cl {B}$ is a closed subset of $\bH^d$,
we have  $\pi(g)\in \Gamma \hull \cl {B}$. Without loss of generality, we may assume $\pi(g)\in \hull \cl {B}$ by replacing $\gamma_i$
 by $\gamma \gamma_i$ for some $\gamma\in \Gamma$ if necessary.
 
We claim that by passing to a subsequence, $$\gamma_i\in \check H\cap \Gamma.$$ 
Let $\cal O$ be a neighborhood of $\pi(g)$ such that
 $$\cal O\cap \Gamma \hull \cl{B}\subset \hull \cl{B};$$ such $\cal O$ exists, since $d(\hull (\gamma B), \hull (B))\ge \eta
 $  for all $\gamma\in \Gamma -(\check H\cap \Gamma)$ where $\eta>0$ is given in \eqref{min}.
 By passing to a subsequence, we may assume that $\pi(\gamma_i \ell_i v_i h_i )\in \cal O$.
 Since $\pi(\ell_i v_i h_i)\in \hull  \cl{B}$ for all $i$, it follows that 
$\pi(\gamma_i \ell_i v_i h_i )\in \hull \cl{B}$ for all $n$.
Therefore $\gamma_i\in \check H\cap \Gamma$.
Applying Lemma \ref{lem.HVPH2} to the sequence $(\ga_i\ell_i)v_ih_i\to g$, there exists $p_i\in AU$ such that $\ga_i\ell_ip_i\to h$ in $\check H$ as $i\to\infty$.
Since  $\Ga L$ is closed, we have $h\in \Ga L$.

Since  $p_i^{-1}v_ih_i\in AUV^+ H(U)= V^+H(U)$ and \be\label{pvhh} \lim_{i\to \infty} p_i^{-1}v_ih_i = h^{-1}g,\ee  
 we have
  $h^{-1}g\in V^+H(U)$ by Lemma \ref{lem.HVPH}.
Therefore, $g=h(h^{-1}g)\in\Ga LV^+H(U)$. This proves that $\Ga LV^+H(U)$ is closed.
Note that in the above argument,
if $h_i=e$ for all $i$, then $h^{-1} g = \lim p_i^{-1}v_i \in AUV^+$. Hence $g=h (h^{-1} g)\in \Gamma L AUV^+=\Gamma L V^+$.
This proves that $\Ga L V^+$ is closed.
\end{proof}

Note that  $x\in \RFPM -\RFM\cdot H(U) $ if and only if $x\in (\RFPM\cap \partial F_{H(U)})- \BFM \cdot \op{C}(H(U))$.
\begin{thm}\label{b-ratner} Let $x\in \RFPM -\RFM\cdot H(U) $.
Then there exist a compact orbit $zL\subset \BFM$ with $L\in \mathcal L_U$,
an element $c\in \op{C}(H(U))$ and a one-parameter subsemigroup $V^+\subset N$ with $\check HV^+ = \check H\check V^+$ such that
\begin{enumerate}\item $\overline{xH(U)}= zL V^+ H(U) c$;
\item $\overline{xAU}=zLV^+c$.
\end{enumerate}
Moreover the closure of the geodesic plane $\pi(xH(U))$ is diffeomorphic to a properly immersed submanifold
$S\times [0, \infty)$
where $S=\pi(zL)$ is a compact geodesic plane inside $\BFM$.
\end{thm}

\begin{proof} 
We claim that we can write $x=z _0vc$ for some non-trivial $v\in \check V^+$, $z_0\in \BFM$ and $c\in \op{C}(H(U))$.

Let $k=\op{dim}U$ and $C\subset\bb S^{d-1}$ denote the $k$-dimensional sphere stabilized by $H(U)$, and $g\in G$ be such that $x=[g]$. 

In view of the identification $\Ga\ba G\simeq\op{F}\cal M$ induced from \eqref{map},
the hypothesis $x\in\RFPM-\RFM\cdot H(U)$ implies that there exists a component $B_i$ of $\Omega$ such that $gC\subset \cl{B_i}$ and $gC$ is tangent to $\partial B_i$. 
Let $S\subset\cl{B_i}$ be the unique $d-2$ dimensional sphere tangent to $\partial B_i$ containing $C$.
Considering $g=(v_1,\cdots,v_d)$ as an orthonormal frame in $\bb H^d$, we can obtain a frame tangent to $\op{hull}(S)$ by rotating the last $d-(\op{dim}(U)+1)$ vectors of $g$.
Since any frame tangent to $\op{hull}(S)$ can be written as $zv$ for some $z\in\BFM$ and a nontrivial $v\in\check V^+$, and the process of rotating the last $d-(\op{dim}(U)+1)$ coordinates corresponds to a right multiplication in $\Ga\ba G$ by an element of $c\in\op{C}(H(U))$, this proves the claim.

 Without loss of generality, we may assume $c=e$. By Theorem \ref{thm.ratner},  $\overline{z_0U}= z_0v_0^{-1} Lv_0$ where $L\in \mathcal L_U$ is contained in $\check H$ and
$v_0\in \check H\cap N$. Hence $\overline{xH(U)}$ contains $z L(v_0v)H(U) $ for $z:=z_0v_0^{-1}\in \BFM$.
Set $V^+:=\{\exp t(\log(vv_0)):t\ge 0 \}$.

Note that $V^+$ is contained in $A(v_0v) A\cup \{e\}$, and hence
$$zL\cup zL v_0 v H(U) =zLV^+ H(U)$$
and $\check H V^+=\check H \check V^+$ since $v\ne e$.

Since $\overline{xH(U)} \subset  zL \cup z L(v_0v)H(U)$, and $zL$ lies in the closure of $z L(v_0v)H(U) $,
the claim (1) follows since $zLV^+H(U)$ is closed by Lemma \ref{lem-ra}.
For the claim (2), note that $\cl{xAU}\supset \cl{z_0U}vA=zLV^+$.
By Lemma \ref{lem-ra}, $zLV^+$ is $AU$-invariant and closed.
Since $x\in zLV^+$, we conclude $\cl{xAU}=zLV^+$. 

To see the last claim, observe that
 $\pi(z LV^+H(U))=\pi(zLV^+AU)=\pi(z L V^+)$ since $V^+AU=AUV^+$, and $AU<L$.
Since $\check HV^+ = \check H\check V^+$,  and
$\pi(zL)$ is a compact geodesic plane (without boundary) in $\pi(z\check H)$, we get
$\pi(z\check H V^+) \simeq \pi(z\check H) \times [0, \infty)$ and
$\pi(zLV^+)\simeq \pi(zL)\times [0, \infty)$.
\end{proof}

\begin{Rmk}\label{NCL}
An immediate consequence of Theorem \ref{b-ratner} is that
 if $P\subset \M$ is a geodesic plane
such that $P\cap \core {\M}=\emptyset$ but $\overline P\cap \core {\M}\ne \emptyset$, then
 $P$ is not properly immersed in $\M$ and $\overline{P}$ is a properly immersed submanifold with non-empty boundary.
\end{Rmk}

\section{Density of almost all $U$-orbits}\label{s:err}
Let $\Gamma <G=\SO^\circ (d,1)$ be a Zariski dense convex cocompact subgroup. 
The action of $N$ on $\RFPM$ is minimal, and hence any $N$-orbit is dense in $\RFPM$ 
 \cite{Win}. Given a non-trivial connected closed subgroup $U$ of $N$,
 there  exists a dense $U$-orbit in $\RFPM$ \cite{MS}. In this section, we deduce from
 \cite{MO} and \cite{MS} that 
 almost every $U$-orbit is dense in $\RFPM$ with respect to the Burger-Roblin measure  in the case
 of a convex cocompact hyperbolic manifold with Fuchsian ends (Corollary \ref{gen}).

The critical exponent $\delta=\delta_\Gamma$ of $\Gamma$ is defined to be the infimum $s\ge 0$ such
that the Poincare series $\sum_{\gamma\in \Gamma} e^{-s d(o, \gamma(o))}$ converges for any $o\in \bH^d$.
It is known that $\delta$ is equal to the Hausdorff dimension of the limit set $\Lambda$ and $\delta=d-1$ if and only if
$\Gamma$ is a lattice in $G$ \cite{Su1}.

Denote by $\mathsf m^{\BR}$ the $N$-invariant Burger-Roblin measure supported on $\RFPM$;
it is characterized as a unique locally finite Borel measure supported on $\RFPM$ (up to a scaling) by (\cite{Bu}, \cite{Ro}, \cite{Win}).
We won't give an explicit formula of this measure as we will  only use the fact that its support is equal to $\RFPM$, together 
with the following theorem:
recall that a locally finite $U$-invariant measure $\mu$ is ergodic if every $U$-invariant measurable
subset has either zero measure or zero co-measure, and is conservative if for any 
measurable subset $S$ with positive measure, $\int_U 1_{S}(xu)du =\infty$ for $\mu$-almost all $x$,
where $du$ denotes the Haar measure on $U$.
\begin{thm} [\cite{MO}, \cite{MS}]\label{thm.MOMS}
 Let $U< N$ be a connected closed subgroup, and let $\Gamma$ be a convex cocompact Zariski dense subgroup of $G$.
Then $\mathsf m^{\BR}$  is $U$-ergodic and conservative if $\delta> \op{co-dim}_N(U).$
\end{thm}

\begin{lem}\label{dop}
Suppose that $\Gamma_1<\Gamma_2$ are convex cocompact subgroups of $G$ with $[\Gamma_1:\Gamma_2]=\infty$.
Then $\delta_{\Gamma_1}<\delta_{\Gamma_2}$.
\end{lem}

\begin{proof} Note that a convex cocompact subgroup is of divergent type (\cite{Su1}, \cite{Ro}). Hence
the claim follows from  \cite[Proposition 9]{DOP} if we check that
 $\Lambda_{\Gamma_1}\ne \Lambda_{\Gamma_2}$.
 
 If $\Lambda:=\Lambda_{\Gamma_1}= \Lambda_{\Gamma_2}$, then their convex hulls are the same,
 and hence the convex core of the manifold $\Gamma_i\ba \bH^d$ is equal to $\Gamma_i\ba \text{hull}(\Lambda)$,
 which is compact.
 Since we have a covering map $\Gamma_1\ba \text{hull}(\Lambda) \to \Gamma_2 \ba \text{hull}(\Lambda)$,
 it follows that
$[\Gamma_1:\Gamma_2]<\infty$.
\end{proof}

\begin{lem}\label{lem.n-2}
If $\Gamma\ba \bH^d$ is a convex cocompact hyperbolic manifold with Fuchsian ends, then  $\delta >d-2 $. \end{lem}
\begin{proof}
If $\Gamma$ is a lattice, then $\La=\bb S^{d-1}$ and $\delta=d-1$.
If $\Gamma\ba \bH^d$ is a convex cocompact hyperbolic manifold with non-empty Fuchsian ends, then
$\Gamma$ contains a cocompact lattice $\Gamma_0$ in a conjugate of $\SO(d-1,1)$ whose limit
set is equal to $\partial B_i$ for some $i$. Now $[\Gamma: \Gamma_0]=\infty$; otherwise, $\Lambda=\partial B_i$.
 Hence $\delta >\delta_{\Gamma_0}=d-2$ by Lemma \ref{dop}.
\end{proof}

\begin{cor}\label{gen}
Let $\M=\Gamma\ba \bH^d$ be a convex cocompact hyperbolic manifold with Fuchsian ends. Let $U<N$ be any non-trivial
 connected closed subgroup. Then for  $\mathsf{m}^{\BR}$-almost every $x\in\op{RF}_+\M$, 
$$\cl{xU}=\op{RF}_+\M.$$
\end{cor}
\begin{proof} Without loss of generality, we may assume that $U=\{u_t\}$ is a one-parameter subgroup.
By Lemma \ref{lem.n-2} and Theorem \ref{thm.MOMS}, $\mathsf{m}^{\BR}$ is $U$-ergodic and conservative. Since  $\op{RF}_+\cal M$ is second countable and the $U$-action on it is continuous, the claim follows. \end{proof}

Since $F^*_{H(U)}\cap \RFPM$ is a non-empty open subset, it follows that almost all $U$-orbits in 
$F^*_{H(U)}\cap \RFPM$ are dense in $\RFPM$.

\section{Horospherical action in the presence of a compact factor}\label{s:um}
Let $\M=\Gamma\ba \bH^d$ be  a convex cocompact hyperbolic manifold with Fuchsian ends and
 fix a non-trivial connected closed subgroup $U$ of $N$.
 Consider a closed orbit $xL$ for $x\in \RFM$ where $L \in \mathcal Q_U$ and $U=L\cap N$. The subgroup $U$ is a horospherical subgroup of $L$, which is known to act
  minimally on $xL\cap \RFPM$ provided $L=L_{nc}$.
 In this section, we extend the $U$-minimality on $xL$ in the case when $L$ has a compact factor.

\begin{theorem} \label{thm.new5}\label{tm} \label{lem.denseAorbit}
Let $X:=xL$ be a closed orbit where
 $x\in\RFPM$, and  $L \in \mathcal Q_U$. Let $U:=L\cap N$. Then the following holds:
\begin{enumerate}
\item
 $X\cap\op{RF}_+\M$ is $U$-minimal.
\item $X$ is $L_{nc}$-minimal.
\item If $L\in \mathcal L_U$ and $x\in \RFM$, then $X\cap \RFM$ contains a dense $A$-orbit.
 \item
 For any non-trivial connected closed subgroup $U_0<U$,  for $\mathsf m_X^{\BR}$-almost all $x\in X$,
$$\cl{xU_0}=X\cap\op{RF}_+\M.$$
\end{enumerate}
\end{theorem}

The subgroup $L\in \cal Q_U$ is of the form $v^{-1} H(U)C v$ where $H(U)C\in \mathcal L_U$ and $v\in N$.
A general case can be easily reduced to the case where $L\in \cal L_U$.  In the following,  we assume $L=H(U)C\in \cal L_U$.
 As before, we set 
 $$H=H(U),  \; H'=H'(U),  \;    \text{and} \;\; F^*=F^*_{H(U)}$$
 and let $\pi_1:H'\to H$ and $\pi_2:H'\to \op{C}(H)$ be the canonical projections.
 In order to define $\mathsf m_X^{\BR}$, choose $g\in G$ so that $[g]=x$.
If we identify $H\simeq\op{SO}^\circ(k,1)$, then by Proposition \ref{count}, $S:=\pi_1(g^{-1}\Gamma g\cap HC)\ba\bb{H}^k$ is a convex cocompact hyperbolic manifold with Fuchsian ends.
Now $\pi_1(g^{-1}\Gamma g\cap HC)\ba H$ is the  frame bundle of $S$, on which there exists
the Burger-Roblin measure as discussed in section \ref{s:err}.
In the above statement, the notation $\m_X^{\BR}$ means  the $C$-invariant lift of this measure to $X=xHC$. 

We first prove the following, which is a more concrete version of Proposition \ref{YLY} in the case at hand:
\begin{proposition}\label{prop.new6}
Let $X$ be as in Theorem \ref{lem.denseAorbit}.
Any $U$-minimal set $Y$ of $X$ with respect to $\RFM$ such that $Y\cap F^*\cap \RFM\ne \emptyset$ is $A$-invariant.
\end{proposition}
\begin{proof}
Let $Y$ be a $U$-minimal set of $X$ with respect to $\RFM$.
Let $y_0\in Y\cap F^*\cap \RFM$.
By Lemma \ref{cor.gn}, there exists a sequence $g_i\to e$ in $HC-\op{N}(U)$ such that 
$y_0g_i\in Y\cap \RFM$ for all $i\ge 1$.

Since $U$ is a horospherical subgroup of $H$ and $C$ commutes with $H$, we can apply Lemma \ref{lem.fullhoro} to the sequence $g_i^{-1}$ and
the sequence of $k$-thick sets $\mathsf T_i:=\{u\in U : y_0g_i u \in Y\cap \RFM \}$ of $U$.
This gives us  sequences $u_{t_i}\to\infty$ in $\mathsf T_i$ and $u_{i}\in U$ such that 
as $i\to \infty$,
$$u_{t_i}^{-1} g_iu_{i}\to a$$ for some non-trivial element $a\in A$.
Since $y_0u_{t_i}$ converges to some $y_1\in Y\cap \RFM$ by passing to a subsequence, we have 
$$y_1 a=\lim (y_0u_{t_i})(u_{t_i}^{-1}g_i u_i) \in Y.$$ Since $\overline{y_1}U=Y$, we get  $Ya\subset Y$.
Since $a$ can be made arbitrarily close to $e$ by Lemma \ref{lem.fullhoro}, 
there exists a  subsemigroup $A_+$ of $A$ such that $YA_+\subset Y$ by Lemma \ref{ss}.
Moreover, for any $a\in A_+$, $Ya\cap \RFM\ne \emptyset$ as $\RFM$ is $A$-invariant. Therefore, $Ya=Y$.
It follows that $Ya^{-1}=Y$ as well. Hence $Y$ is $A$-invariant. \end{proof}

We now present:
\subsection*{Proof of Theorem \ref{thm.new5}}
First suppose that $xL\cap F^*\ne \emptyset$. We may then assume $x\in F^*\cap \RFM$.
Let $Y$ be a $U$-minimal set of $X$ with respect to $\RFM$.
If $Y$ were contained in $\partial F$, then $Y\subset \partial F\cap \RFM$. Since  $\op{Stab}_L(x)$ is Zariski dense in $L$ by the definition of $\cal L_U$,
it follows from \cite[Lemma 4.13]{BQ} that $X\cap\op{RF}_+\M$ is $AU$-minimal. 
Therefore we have $$\cl{YA}=X\cap \RFPM$$ and hence $X$ has to be contained in 
the closed $A$-invariant subset $\partial F\cap \RFM$ as well, yielding a contradiction.
Therefore, $$Y\cap F^*\cap \RFM\ne \emptyset.$$

Hence, by Proposition \ref{prop.new6}, $Y$ is $A$-invariant.
Therefore the claim (1) follows from the $AU$-minimality of $X\cap\op{RF}_+\M$ if $x\in F^*$. 
Now suppose $xL\subset \partial F$.
 In this case, it suffices to consider the case when $U$ is a proper subgroup of $N$; otherwise $L=G$ and has no compact factor.
Hence we may assume without loss of generality that $U\subset \check H\cap N$.
As $xL$ is closed, Theorem \ref{b-ratner} implies that $xL\subset \BFM \cdot \op{C}(H(U))$. Hence by modifying $x$ by an element of $\op{C}(H(U))$,
we may assume that $X$ is contained in a compact homogeneous space of
$\check H=\SO^\circ(d-1,1)$, which is the frame bundle of a convex cocompact hyperbolic manifold with empty Fuchsian ends.
Therefore the claim (1) follows from the previous case of $x\in F^*$, since $F^*=\RFM$ in the finite volume case.

Claim (2) follows from (1) since $\RFPM \cdot H$ is closed, and $X\subset \RFPM\cdot  H$.

For the claim (3), it suffices to show that the $A$ action on $X\cap\op{RF}\M$ is topologically transitive  (cf. \cite{DK}).
Let $x$, $y\in X\cap\op{RF}\M$ be arbitrary, and $\cal{O}$, $\cal{O}'$ be open neighborhoods of $e$ in $H$.
The set $U U^tA(M\cap H)$ is a Zariski open neighborhood of $e$ in $H$ where $U^t$ is the expanding horospherical subgroup
of $H$ for the action of $A$. 
Choose an open neighborhood $Q_0$ of $e$ in $U$, and an open neighborhood $P_0$ of $e$ in $U^tA(M\cap H)$ such that $Q_0P_0\subset\cal{O}$.

We claim that $xQ_0A\cap y\cal{O}'\neq\emptyset$, which implies $x\cal{O}A\cap y\cal{O}'\neq\emptyset$.
Suppose that this is not true. Then
\begin{equation*}
xQ_0A\subset\Ga\ba G-y\cal{O}'
\end{equation*}
where the latter is a closed set.
Now, choose a sequence $a_n\in A$ such that $a_n Q_0a_n^{-1}\to U$ as $n\to\infty$, and observe
\begin{equation*}
xa_n^{-1}(a_n Q_0a_n^{-1})=xQ_0a_n^{-1}\subset \Ga\ba G-y\cal{O}'.
\end{equation*}
Passing to a subsequence, $xa_n^{-1}\to x_0$ for some $x_0\in\op{RF}\M$, and we obtain 
that $x_0U$ is contained in the closed subset $\Ga\ba G-y\cal{O}'$.
This contradicts the $U$-minimality of $X\cap \op{RF}_+\M$, which is claim (1).
This proves (3).

For the claim (4), note that by Corollary  \ref{gen}, almost all $U_0$-orbits in $\pi_1(g^{-1}\Ga g \cap HC)\ba H$ are dense
in the corresponding $\RFPM$-set.
It follows that for almost all $x$, the closure $\overline{xU_0}$ contains a $U$-orbit of $X$.
Hence (4) follows from the claim (1).

\section{Orbit closure theorems: beginning of the induction}\label{sec.W-orbit}\label{s:or}
In the rest of the paper, let $\M=\Gamma\ba \bH^d$ be a convex cocompact hyperbolic $d$-manifold with Fuchsian ends, and $G=\SO^\circ(d,1)$.
Let $U< N$ be a non-trivial connected proper closed subgroup, and $H(U)$ be its associated simple Lie subgroup of $G$.

Let $\mathcal L_U$ and $\mathcal Q_U$ be as defined in \eqref{deflu} and \eqref{qu1}.
The remainder of the paper is devoted to the proof of the next theorem from which Theorem \ref{mtp} follows:
\begin{theorem}\label{thm.H'UMIN}\label{mainth}
\begin{enumerate}
\item
For any $x\in \RFM$,
\begin{equation*}
\cl{xH(U)}=xL\cap F_{H(U)}
\end{equation*}
where  $xL$ is a closed orbit of some $L\in \mathcal L_{U}$.
\item Let $x_0\widehat L$ be a  closed orbit for some $\widehat L\in\cal{L}_U$ and $x_0\in \RFM$.
\begin{enumerate} 
\item  For any $x\in x_0\widehat L\cap \RFPM$, 
 \begin{equation*}
\cl{xU}=x  L\cap\op{RF}_+\M
\end{equation*}
where  $x L$ is a closed orbit of  some $L\in \cal{Q}_U$.
\item  For any $x\in x_0\widehat L\cap \RFM$,
 \begin{equation*}
\cl{xAU}= xL\cap\op{RF}_+\M
\end{equation*}
where  $x L$ is a closed orbit of  some $L\in \cal{L}_U$.\end{enumerate}
\item Let $x_0\widehat L$ be a  closed orbit for some $\widehat L\in\cal{L}_U$ and $x_0\in \RFM$. Let $x_iL_i 
\subset x_0\widehat L$ be a sequence of  closed orbits intersecting $\RFM$  where $x_i\in\RFPM$, $L_i\in\cal{Q}_U$. 
Assume that  no infinite subsequence of $x_iL_i$ is contained in a subset
of the form $y_0L_0D$ where  $y_0L_0$ is a closed orbit of $L_0\in \cal L_U$ with $\op{dim}L_0<\op{dim} \hat L$  and $D\subset \op{N}(U)$ is a compact subset.
 Then 
$$\lim\limits_{i\to\infty}\text{ }( x_i L_i \cap\op{RF}_+\M)=x_0\widehat L\cap \op{RF}_+\M.$$
\end{enumerate}
\end{theorem}


We will prove (1), (2), and (3) of Theorem \ref{thm.H'UMIN} by induction on the co-dimension of $U$ in $N$
and the co-dimension of $U$ in $\widehat L \cap N$, respectively.

For simplicity, let us say $(1)_m$ holds, if $(1)$ is true for all $U$ satisfying $\op{co-dim}_N(U)\leq m$.
We will say $(2)_m$ (resp. $(2.a)_m$, $(2.b)_m$)  holds, if $(2)$  (resp. $(a)$ of $(2)$, $(b)$ of $(2)$) is true for all $U$ and $\widehat{L}$ satisfying $\op{co-dim}_{\widehat L \cap N}(U)\leq m$ and similarly for $(3)_m$.

\subsection*{Base case of $m=0$}
Note that the bases cases $(1)_0$, and $(3)_0$ are trivial, and that $(2)_0$ follows from Theorem \ref{tm}.

We will deduce $(1)_{m+1}$ from $(2)_m$ and $(3)_m$  in section \ref{s:1}, and  $(2)_{m+1}$
 from $(1)_{m+1}$, $(2)_m$, and  $(3)_m$ in section \ref{s:2}, and finally deduce
 $(3)_{m+1}$ from  $(1)_{m+1}$, $(2)_{m+1}$ and $(3)_m$ in section \ref{s:3}.

\begin{remark}\label{checkh}
When $\op{co-dim}_{\widehat L \cap N}(U)\geq 1$ and $\widehat L\in \cal L_U$, we may assume
without loss of generality that
$$U\subset \widehat L\cap N \cap \check H$$
 by replacing $U$ and $\widehat L$ by their conjugates using an element $m\in M$.
\end{remark}

\begin{remark}\label{rfpm}
In the case when $x\in \partial F_{H(U)}$, Theorem \ref{mainth} (1) and (2) follow from 
 Theorem \ref{thm.ratner}, and if $x_0\in \partial F_{H(U)}$, (3) follows from the work of Mozes-Shah \cite{MS}.
So the main new cases of Theorem \ref{mainth} are when $x, x_0\in F_{H(U)}^*$.
\end{remark}

We will use following observation:
\subsection*{Singular $U$-orbits under the induction hypothesis} 
Recall the notation $\mathscr S(U, x\widehat L)$ and
$\mathscr G(U, x\widehat L)$ from \eqref{gxol}.

\begin{lemma}\label{lem.gp}

Suppose  that $(2.a)_m$ is true and that for $x\in \RFM$, $xU$ is contained in a closed orbit $x\widehat L$
for some $\widehat L\in\cal{L}_U$.
 \begin{enumerate}
\item   If $\op{co-dim}_{\widehat L\cap N}(U)\le m+1$, then 
for any $x_0\in \mathscr S(U, x\widehat L)\cap \RFPM$,
$$\overline{x_0U}=x_0 L\cap \RF_+\M $$
where  $x_0 L$ is a closed orbit of  some subgroup $L<\widehat L$ contained in $\cal{Q}_U$, satisfying
 $\op{dim} L_{nc} <\op{dim}\widehat L_{nc}$.

\item  If $\op{co-dim}_{\widehat L\cap N}(U)\le m$, then 
for any $x_0\in \mathscr G(U, x\widehat L)$,
$$\overline{x_0U}=x_0\widehat L\cap \RF_+\M .$$

\end{enumerate}
\end{lemma}
\begin{proof} 
Suppose that $\op{co-dim}_{\widehat L\cap N}(U)\le m+1$ and that  $x_0\in \mathscr S(U, x\widehat L)\cap \RFPM$.
By Proposition \ref{explain}, we
get $$\overline{x_0U}\subset x_0 Q$$
for some closed orbit $x_0Q$ where $Q\in \mathcal Q_U$ satisfies $\op{dim} Q_{nc} <\op{dim} \widehat L_{nc}$.

Now $Q= vL_0 v^{-1}$ for some $L_0\in \mathcal L_U$ and $v\in U^\perp$.
We have  $x_0 U v =x_0vU \subset x_0 v L_0$.
Since  $\op{co-dim}_{N\cap L_0}(U)=\op{co-dim}_{N\cap Q}(U)\le  m$,
 by applying $(2)_m$, we get
$$\overline{x_0vU}=x_0 v L\cap \RF_+\M $$
for some  closed orbit $x_0vL$ where $L\in \mathcal Q_U$ is contained in $L_0$.
Therefore
$$\overline{x_0U}=x_0 v Lv^{-1}\cap \RF_+\M .$$
As $vLv^{-1}\in \mathcal Q_U$ and  $\op{dim} L_{nc} \le \op{dim} Q_{nc}
<\op{dim}\widehat L_{nc}$,
 the claim (1) is proved.

To prove (2), note that
by $(2.a)_m$, we get $\overline{x_0U}=x_0L\cap \RF_+\M$ for some closed orbit $x_0L$ with
$L\in \mathcal Q_U$ such that $L\subset \widehat L$.
Since $x_0\in \mathscr G(U, x\widehat L)$, we have 
$\op{dim} L_{nc}=\op{dim} \widehat L_{nc}$.

Since $L\subset \widehat  L$,  $L\cap N$ is a horospherical subgroup of $\widehat L$.
By Theorem \ref{tm}, $L\cap N$ acts minimally on $x\widehat L$, and hence $L=\widehat L$.
\end{proof}

\section{Generic points, uniform recurrence and additional invariance}\label{s:li}\label{s:ge2}
The primary goal of this section is to prove Propositions \ref{cor.lin} and \ref{lem.lin}
in obtaining additional invariances
using a sequence converging to a generic point of an intermediate closed orbit; the main ingredient is  Theorem \ref{lin2} (Avoidance theorem II). The results in this section are main tools in the enlargement steps of the proof of Theorem \ref{thm.H'UMIN}.

In this section,  we let $U<N$ be a non-trivial connected closed subgroup. We suppose that \begin{itemize}
\item $(2)_m$ and $(3)_m$ are true;
\item $x\widehat L$ is a closed orbit for some $x\in \RFM$, and
$\widehat L\in\cal{L}_U$;
\item $\op{co-dim}_{\widehat L \cap N}(U)\le m+1$.
\end{itemize}

 We let $\{U^{(i)}\}$ be a finite collection of one-parameter subgroups generating $U$.
In the next two propositions, we let $X$ be a closed $U$-invariant subset of $x_0\widehat L$ such that
 $$X\supset xL\cap\op{RF}_+\M$$ for some closed orbit $xL$ where  $L\in \cal Q_U$ is a proper subgroup of $\widehat L$
 and  $$x\in \bigcap_{i}\mathscr{G}(U^{(i)},xL)\cap \RFM.$$

\begin{prop}[Additional invariance I] \label{cor.lin}
 Suppose that there exists a sequence $x_i\to x$ in $X$ where $x_i=x\ell_ir_i$ with $x\ell_i\in xL \cap\RFM$ and $r_i
\in \exp \frak l^\perp-\op{N}(U)$.

Then there exists a sequence $v_n\to\infty$ in $ ( L\cap N )^\perp$ such that 
\begin{equation*}
xL v_n\cap\op{RF}_+\M \subset X.
\end{equation*}
\end{prop}

\begin{proof}
Since $r_i\notin \op{N}(U)$, we can fix a one-parameter subgroup  $U_0=\{u_t:t\in \br\}$ in the family $\{U^{(i)}\}$ such that
$r_i\notin \op{N}(U_0)$ by passing to a subsequence.

Let $E_j$,  $j\in \mathbb N$, be  a sequence of compact subsets
in $\mS(U_0, xL)\cap \RFM$
given by Theorem \ref{lin2}.
  Set $z_i:=x\ell_i\in x L\cap \RFM$. Fix $j\in \bb N$ and $n\gg 1$.
  Since  $z_i\to x$ and $x\in \mG(U_0, xL)$,   there exist $i_j\ge 1$  and an open neighborhood $\cal O_j$
   of $E_j$ such that
 for each $i\ge i_j$, the set \begin{equation*}
\mathsf T_i=\{t\in\bb{R} : z_iu_t\in \op{RF}\M -\cal O_j\},
\end{equation*}
is  $2k$-thick by loc. cit.
We apply Lemma \ref{lem.QR2} to the sequence $\mathsf T_i$. We can find a sequence
${t_i=t_i(n)}\in \mathsf T_i$, $i\ge i_j$ and elements $y_j=y_j(n), v_j=v_j(n)$ satisfying that  as $i\to\infty$,
\begin{itemize}
\item
$z_i u_{t_i}\to y_j\in ( \RFM\cap xL)- \cal O_j $;
\item
$u_{t_i}^{-1}r_iu_{t_i}\to v_j\in (L\cap N)^\perp$ with $n \leq\norm{v_j}\leq (2k^2) n$.
\end{itemize}

So as $i\to \infty$, $$ x_i u_{t_i} =z_i r_i u_{t_i}\to y_jv_j\quad \text{in } X.$$

Note that since $L$ is a proper subgroup of $\widehat L$, we have
$\op{co-dim}_{ L\cap N}(U)\le m$ by Lemma \ref{sin3}.

If $y_j$ belongs to $\mG(U, xL)$, then $ \overline{y_jU}v_j=x L \cap \RFPM$
by Lemma \ref{lem.gp}(2), and hence 
$$X\supset \overline{y_j v_j U}=\overline{y_jU}v_j=x L v_j\cap \RFPM .$$
Hence the claim follows if $y_j(n) \in \mG(U, xL)$ for an infinite subsequence of $n$'s.

Now we may suppose that  for all $n\ge 1$ and $j\ge 1$, $y_j(n) \in \mS(U, xL)\cap \RFPM$, after passing to a subsequence.
Fix $n$, and set $y_j=y_j(n)$ and $v_j=v_j(n)$. Then, since $\op{dim}_{L\cap N}U\le m$,  by $(2)_m$, 
we have
\be\label{yju} \cl{y_jU}=y_jL_j\cap\op{RF}_+\M\ee for some 
closed $y_jL_j$ where $L_j\in\cal{Q}_U$
is  contained in $\widehat L$ and $\op{dim}(L_j)_{nc} <\op{dim} \widehat L_{nc}$.
Write $L_j= w_j^{-1} L_j' w_j$ for $L_j'\in\cal L_U$ and
$w_j\in  U^\perp$.
We claim that
 the sequence $y_jL_j=y_j w_j^{-1} L_j' w_j$ satisfies the hypothesis of $(3)_m$.
It follows from the condition $ y_j\in ( \RFM\cap xL)- \cal O_j $ for all $j$ that no infinite subsequence of $y_jL_j$ is contained in a subset
of the form $y_0L_0D\subset \mathscr S(U, x L)$  where $y_0L_0$ is closed, $L_0\subset \cal Q_U$ and $D\subset \op{N}(U)$ is a compact subset.
Hence, by $(3)_m$, we have
\begin{equation*}
\limsup_j  y_j L_j \cap\op{RF}_+\M =x L\cap \RF_+\M .
\end{equation*} Therefore for each fixed $n\gg 1$ and $y_j=y_j(n)$,
 $$\limsup_j \cl{y_jU}=x L\cap \RF_+\M .$$
By passing to a subsequence, there exists  $u_j\in U$ such that
$y_j u_j$ converges to $x$. As $n\le \| v_j(n) \| \le (2k^2) n$, the sequence $v_j(n) $ converges to some $v_n\in 
(L\cap N)^\perp$ as $j\to \infty$,
after passing to a subsequence.
Therefore
$$ \limsup_j \overline{y_j(n) v_j(n) U}=\limsup_j \overline{y_j(n) U}v_j(n) \supset  \overline{xU}v_n=xL v_n\cap \RFPM $$
where the last equality follows from Lemma \ref{lem.gp}(2), since $\op{co-dim}_{ L\cap N}(U)\le m$.
\end{proof}

Note that in the above proposition, $y_i=x\ell_ir_i$ is not necessarily in $\RFM$, and hence we cannot apply the avoidance
theorem \ref{lin2}
 to the sequence $y_i$ directly.
We instead applied it to the sequence $x\ell_i$.

In the proposition below, we will consider a sequence $x_i\to y$ inside $\RFM$, and apply Theorem \ref{lin2} to the sequence $x_i$.
\begin{prop}[Additional invariance II] \label{lem.linearization2}\label{lem.lin}
Suppose that there exists a sequence  $x_i\in X\cap\op{RF}\M-xL\cdot\op{N}(U)$, converging to $x$ as $i\to\infty$.
Then there exists a sequence $v_j\to \infty$ in $(N\cap L)^\perp$ such that
\begin{align*}
xL v_j \cap\op{RF}_+\M \subset X\quad\text{ and}\quad 
xLv_j  \cap\op{RF}\M \neq\emptyset .
\end{align*}
The same works for a sequence $x_i\in \op{RF}\M-xL\cdot\op{N}(U)$ such that $\limsup x_i U\subset X$.
\end{prop}

\begin{proof} Let $x_i\in \op{RF}\M-xL\cdot\op{N}(U)$ be a sequence converging to $x$ 
 such that $\limsup x_i U\subset X$. Write $x_i=xg_i$ for $g_i\to e$ in $\widehat L$.
Since $L$ is reductive, we can write $g_i=\ell_ir_i$ where
$\ell_i\to e$ in $L$ and $r_i\to e$ in $\exp \frak l^\perp$ as $i\to\infty$. By the assumption on $x_i$,
 there exists  a one-parameter subgroup  $U_0=\{u_t:t\in \br\}$ among $U^{(i)}$ such that $r_i\notin \op{N}(U_0)$ by passing to a subsequence.

For $R>0$, we set 
$B(R):=\{v\in (L\cap N)^\perp\cap \widehat L : \norm{v}\leq R\}$. Fix $j$ and $ n\in\bb{N}$. Let $E_j, \cal O_j$ be  given by Theorem \ref{lin2} for $xL$ with respect to $U_0$.
Then $E_j$ is of the form $$E_j=\bigcup_{i\in \La_j}\Gamma\ba \Gamma H_i D_i\cap \RFM$$
where $H_i\in \mathscr H^\star$ satisfies
$\op{dim} (H_i)_{nc} <\op{dim} L_{nc}$ and $ D_i$ is a compact subset of $X(H_i, U_0)\cap L$.
As $B(2k^2 n)\subset \op{C}(U_0)$, we have $D_j^* :=D_j  B(2k^2 n)$ is a compact subset of  $X(H_i, U_0)$.
Hence the following set 
$$\tilde E_j:= \bigcup_{i\in \La_j}\Gamma\ba \Gamma H_i D_i^*\cap \RFM$$
belongs to $\mathcal E_{U_0}$ and is associated to the family $\{H_i: i\in \La_j\}$, as defined in \eqref{defeu}.

Let $\tilde E_j' \in \mathcal E_{U_0}$ be a compact subset given by Theorem \ref{avoid1}, which is also associated to
the same family $\{H_i: i\in \La_j\}$.
Note that for any $z\in  \tilde E_j'$, the closure $\overline{zU_0}$ is contained in $ \Gamma\ba \Gamma H_i D_i^*$ for some $i\in \La_j$.
In particular, $\tilde E_j'$ is a compact subset disjoint from $ \mathscr{G}(U_0,xL)$.
Since $x_i\to x$ and $x \in \mathscr{G}(U_0,xL)$,
there exists $i_j\ge 1$ such that $x_i\notin  \tilde E_j'$ for all $i\ge i_j$.
By Theorem \ref{avoid1}, there exists a neighborhood $\tilde {\cal O}_j$ of $\tilde E_j$
such that for each $i\ge i_j$, the set
\begin{equation*}
\mathsf T_i=\{t\in\bb{R} : x_iu_t\in\op{RF}\M- \tilde{\cal O}_j\}
\end{equation*}
is $2k$-thick. Applying Lemma \ref{lem.QR2} to $\mathsf T_i$, and $r_i\to e$, we can find ${t_i}=t_i(n)\in \mathsf T_i$ such that $u_{t_i}^{-1}r_iu_{t_i}\to v_j$ for 
some $v_j=v_j(n)\in (L\cap N)^\perp$, with $n\leq\norm{v_j}\leq 2k^2\cdot n$.
Passing to a subsequence, $x_iu_{t_i}$ converges to some $ \tilde x_j(n)\in \op{RF}\M- \tilde{\cal O}_j $ as $i\to\infty$.
Set $$z_i:=x\ell_i,\;\text{ and }\; \cal O_j:=\tilde{\cal O}_j B(2k^2 n)\cap xL.$$
Since
$x_iu_{t_i}=z_i u_{t_i} (u_{t_i}^{-1}r_iu_{t_i})$, we have
$$z_i u_{t_i}\to y_j \in (\RFPM\cap x L )- \cal O_j $$
where $y_j= y_j(n):=  \tilde x_j(n) v_j^{-1}$. 

We check that $E_j\subset \cal O_j$ as $B(2k^2 n) B(2k^2 n)$ contains $e$. It follows that $y_j\notin E_j$.
Since
$ \tilde x_j(n)\in \cl{y_jU}v_j\subset X, $
we have  $\cl{y_jU}v_j\cap \RFM\ne \emptyset$.
Given these, we can now
 repeat verbatim the proof of Proposition \ref{cor.lin} to complete the proof.
 \end{proof}

Theorem \ref{mtl} in the introduction can be proved similarly to the proof of Proposition \ref{cor.lin}.

\noindent{\bf Proof of Theorem \ref{mtl}}
Let $E_j$, $j\in \mathbb N$, be  a sequence of compact subsets of $\mS(U_0)\cap \RFM$
given by Thoerem \ref{lin2}.
Fix $j\in \mathbb N$. Then there exist $i_j\ge 1$ and a neighborhood $\cal O_j$ of $E_j$ such that 
$$\{t\in \br: x_{i} u_t\in \RFM-\cal O_j\}$$ is $2k$-thick for all $i\ge i_j$. Hence
 we can find a sequence $t_i\in  [-2kT_{i},-T_{i}]\cup [T_{i},2kT_{i}]$
such that $x_{i}u_{t_i}\in\op{RF}\M-\cal O_j$ for all $i\ge i_j$.
Hence, by passing to a subsequence,
 $x_iu_{t_i}$ converges to some
$y_j\in\op{RF}\M-\cal O_j$ as $i\to\infty$.
If $y_j\in \mathscr{G}(U)$ for some $j$,
then $(2)_m$ and Lemma \ref{lem.gp}(2)
imply that $\overline{y_jU}=  \RFPM$, proving the claim.

Now, we assume that $y_j\in\mathscr{S}(U,x \widehat L)$ for all $j$.
Then by $(2)_m$  and Lemma \ref{lem.gp}(1), we have
$$ \cl{y_jU}=y_jL_j\cap\op{RF}_+\M $$ for some 
closed $y_jL_j$ where $L_j\in\cal{Q}_U$ is a proper subgroup of $G$.
Similarly to the proof of Theorem \ref{cor.lin}, we can show that  the sequence $y_jL_j$ satisfies the hypothesis of $(3)_m$. Hence, by applying $(3)_m$ to the sequence $y_jL_j$, we get
\begin{equation*}
\limsup \;(y_jL_j\cap\op{RF}_+\M) = \RF_+\M .
\end{equation*} 
Therefore $\limsup\; y_jU=\limsup\; \cl{y_jU}=\RFPM$.
This, together with Theorem \ref{thm.new5}(4), finishes the proof.

\section{$H(U)$-orbit closures: proof of $(1)_{m+1}$} \label{subsec.a}\label{s:1}
We fix a non-trivial connected 
proper subgroup $U<N$. Without loss of generality, we may assume
$$U<N\cap \check H$$
using a conjugation by an element of $M$.
We set
 $$H=H(U), \;\; H'=H'(U), \;\;  F=F_{H(U)}, \; \; F^*=F^*_{H(U)},\;\; \text{and}\;\; \partial F=\partial F_{H(U)}.$$
By the assumption $U<N\cap \check H$, we have
$$\partial F\cap \RFM= \BFM\cdot  \op{C}(H) .$$

We will be using the following observation:
\begin{lem}\label{LLL} Let $x_1L_1$ and $x_2L_2$ be  closed orbits where $x_1, x_2\in \RFM$, $L_1\in \cal Q_U$ and
 $L_2\in \mathcal L_U$.
If $x_1L_1\cap \RFM \subset x_2L_2$,
then $L_1\subset L_2$ and $x_1L_1\subset x_2L_2$.
\end{lem}
\begin{proof} 
Since $L_2$ contains $H$, we get
that  $x_1L_1\cap \RFM \cdot H \subset x_2L_2$.
Suppose that  $x_1L_1\cap F^*\ne \emptyset$. We may assume $x_1\in F^*$. 
Recall from (3.2) that $F^*\subset\RFM\cdot H$. Hence, we have $x_1L_1\cap F^*\subset x_2L_2$.
 Since $F^*$ is open, there exist $g_1, g_2\in G$ such that
 $[g_i]=x_i$, and $g_1L_1\cap \cal O\subset g_2L_2$ for some open neighborhood $\cal O$ of $g_1$.
 It follows that $L_1\cap g_1^{-1}\cal O\subset g_1^{-1}g_2 L_2$. Since $e\in g_1^{-1}g_2 L_2$,
we have  $g_1^{-1}g_2L_2=L_2$.
 Since $L_1$ is topologically generated by $L_1\cap g_1^{-1}\cal O$, we deduce $L_1\subset L_2$.
 Since $x_1L_1\cap x_2L_2\ne \emptyset$, it follows that $x_1L_1\subset x_2L_2$.

 Now consider the case when $x_1L_1\cap F^*=\emptyset$. In this case, $x_1L_1\cap \RFM \subset \RFM\cap  \partial F$. 
By Theorem \ref{tm}(4), we can assume that $\overline{x_1U}=x_1L_1\cap \RFPM$.
 As $x_1$ is contained in $ \BFM \cdot \op{C}(H)$, so is $\overline{x_1U}$.
 It follows that $x_1L_1$ is compact and hence is contained in $ \RFM$. Hence the hypothesis implies
 that $x_1L_1\subset x_2L_2$, which then implies $L_1\subset L_2$ by the same argument in the previous case.
 \end{proof}

\begin{lemma}\label{lem.inc}
Let $y_1L_1$ and $y_2L_2$ be closed orbits where $y_1\in\RFM$, $y_2\in\RFPM$,
 $L_1\in\cal Q_U$ and $L_2\in\cal L_U$. 
If $y_1L_1\subset y_2L_2 D$ for some subset $D\subset\op{N}(U)$, then there exists $d\in D$
such that $L_1\subset d^{-1}L_2d$ and $y_1L_1\subset y_2L_2d$.
\end{lemma}
\begin{proof}
By Theorem \ref{thm.new5}(4), we may assume $\cl{y_1U}=y_1L_1\cap\RFPM$. By the assumption, $y_1=y_2\ell_2 d$
for some $\ell_2\in L_2$ and $d\in D$. Since $y_2\ell_2 =y_1d^{-1}$ and { {$\op{N}(U)$ preserves $\RFPM$,}
$y_2\ell_2\in \RFPM$. Hence we may replace $y_2$ by $y_2\ell_2$, and hence
 assume that $y_1=y_2d$.}
Since
\begin{equation}\label{eq.qqu}
y_1L_1\cap\RFPM=\cl{y_2dU}=\cl{y_2U}d\subset y_2L_2d,
\end{equation}
and $F^*\subset\RFPM  \cdot H$, we get $y_1L_1d^{-1}\cap F^*\subset y_2L_2$.

If $y_1L_1d^{-1}\cap F^*\neq\emptyset$,  using the openness of $F^*$, the conclusion follows as in the first part of the proof of Lemma \ref{LLL}.
Now consider the case when $y_1L_1d^{-1}\cap F^* =\emptyset$. In particular, $y_2=y_1d^{-1}$ belongs to 
\begin{equation*}
\RFPM-F^*\subset\op{BF}\M\cdot\op{N}(U)
\end{equation*}
by \eqref{bfmz}.
 It follows from Theorem \ref{thm.ratner} that $\cl{y_2U}=y_2L_2'$ for some $L_2'\in\cal Q_U$ contained in $L_2$.
In view of \eqref{eq.qqu}, we get $y_1L_1\cap\RFPM= y_1d^{-1}L_2'd$.
Therefore $d^{-1}L_2'd\subset L_1$.
Since $y_1L_1\cap\RFPM$ is $A(L_1\cap N)$-invariant, it follows that $d^{-1}L_2'd\in\cal L_U$ and $d^{-1}L_2'd\cap N=L_1\cap N$.
As a result, $(L_1)_{nc}=d^{-1}(L_2')_{nc}d$.
By Lemma \ref{sin3}, we get that
$L_1=d^{-1}L_2'd\subset d^{-1}L_2d$ and that $y_1L_1=y_2L_2'd\subset y_2L_2d$.
\end{proof}

 The following proposition says that the  classification of $H'$-orbit closures
  yields the classification of $H$-orbit closures:
\begin{prop}\label{lem.W'toW}\label{hpu}
Let $x\in\RFM$, and assume that there exists $U<\til{U}<N$ such that $xH'(\til{U})$ is closed, and
\begin{equation*}
\cl{xH'}=xH(\til{U})\cdot\op{C}(H)\cap F.
\end{equation*}
Then there exists a closed subgroup $C<\op{C}(H(\tilde U))$ such that 
\begin{equation*}
\cl{xH}=xH(\til{U})C\cap F.
\end{equation*}
\end{prop}
\begin{proof}
By Proposition \ref{cor.HUminimal} and Theorem \ref{tm}(2), there exists a closed subgroup $C <\op{C}(H(\tilde U))$ such that
$H(\til{U})C\in \mathcal L_U$ and $X:= xH(\til{U})C $
is a closed $H(\til U)$-minimal subset.
In particular, $\cl{xH}\subset X\cap F.$
Now, by Theorem \ref{lem.denseAorbit}(3), there exists $y\in X$ such that $\cl{yA}=X\cap\RFM.$
Since $C$ is contained in $ \op{C}(H)$ and
$$
\cl{xH}\cdot\op{C}(H)=\cl{xH'}=xH(\til{U})\cdot\op{C}(H)\cap F,$$
there exists $c_0\in\op{C}(H)$ such that $yc_0\in\cl{xH}$.
Since $\cl{yA}c_0=\cl{yc_0A}\subset \cl{xH}$ and $c_0\in \op{C}(H)$,  it follows
$Xc_0 \cap\op{RF}\M \subset\cl{xH}\subset X.$
Applying Lemma \ref{LLL}, we get $Xc_0=\overline{xH}=X$.
\end{proof}

In the rest of this section,  fix $m\in \mathbb N\cup\{0\}$ and assume
that $$1\le \op{co-dim}_N(U)=m+1 .$$ 
In order to describe the closure of $xH(U)$,  in view of Theorem \ref{thm.ratner},
we assume that  $$x\in F^*\cap \RFM.$$
 By Proposition \ref{hpu},
it suffices to show that
\begin{equation}\label{ff} \cl{xH'}= x L \op{C}(H)\cap F\end{equation}
for some closed orbit $xL$ for some $L\in \mathcal L_U$.

 In the rest of this section, we set $$X:=\cl{xH'} \text{ and assume that $xH'$ is not closed, i.e., } X\ne xH'.$$
 We also assume that $(2)_m$ holds in the entire section.
\begin{lemma} [Moving from $\mathcal Q_U$ to 
$\mathcal L_U$] \label{lem.maximal}\label{alpha}
If $x_0L\cap\op{RF}_+\M\subset X$ for some  closed orbit $x_0L$ with
$x_0\in \RFM$, and $L \in\cal{Q}_U -\mathcal L_U$, then
\begin{equation*}
x_1\widehat{L}\cap\op{RF}_+\M\subset X
\end{equation*}
for some closed orbit $x_1\widehat L$ with $x_1\in \RFM$, and $\widehat{L}\in\cal{L}_U$ with $\op{dim} (\widehat L\cap N) > \op{dim} (L\cap N)$. 
Moreover, $x_1$ can be taken to be any element of the set $\limsup_{t\to +\infty} x_0u a_{-t}$ for any $u\in U$.
\end{lemma}
\begin{proof}
By Lemma \ref{qu}, we can write $L=v^{-1}\widehat Lv$ for some $\widehat L\in\cal{L}_U$ and $v\in (\widehat L\cap N)^\perp$.
As $L\notin \mathcal L_U$, we have $v\neq e$.
Set $\widehat U:=\widehat L\cap N$.
Note that  $x_0v^{-1}\widehat UAv\subset x_0 L\cap\op{RF}_+\M$, as $\widehat U A<\widehat L$.
Since $X$ is $A$-invariant, $x_0v^{-1}\widehat UAvA\subset X$.
Let $V^+$ be the  unipotent one-parameter subsemigroup contained in $ AvA$, and let $V$ be the one-parameter subgroup containing $V^+$.
Then $x_0v^{-1}V^+\widehat U\subset X$. Since $x_0A\subset \RFM$ and $\RFM$ is compact,
$\limsup_{t\to +\infty} x_0a_{-t}$ is not empty.
Now let $x_1$ be any limit of $x_0ua_{-t_n}$ for some sequence $t_n \to \infty$ 
and $u\in U.$
Since
 $v^{-1}V^+$ is an open neighborhood of $e$ in $V$,
 $\liminf_{n\to \infty} a_{t_n}v^{-1}V^+a_{-t_n} =V$.
 Note that as $u\in \widehat U$,
\begin{equation*}
x_0ua_{-t_n}(a_{t_n}v^{-1}\widehat{U}V^+a_{-t_n})=x_0v^{-1}\widehat{U}V^+a_{-t_n}\subset X.
\end{equation*}
As a result,  we obtain that $x_1\widehat{U}V \subset X$ and hence
$x_1\widehat{U}V A\subset X$. Since $\op{co-dim}_N(\widehat UV)\leq m$, 
the claim follows from by $(2.a)_m$.
\end{proof}

\begin{prop}\label{accdd}
If $R:=X\cap F^*\cap \RFM$ accumulates on $\partial F$, i.e., there exists
$x_n\in R$ converging to a point in $\partial F$, then
$$X\supset x_0L\cap \RFPM$$
for some closed orbit $x_0L$ with $x_0\in F^*\cap \RFM$ and $L\in \mathcal L_U$ such that $\op{dim}(L\cap N)>\op{dim} U$.
\end{prop}

\begin{proof}
 There exists $x_n\in R$ which converges to some $z\in\op{BF}\M\cdot\op{C}(H)$ as $n\to\infty$.
We may assume $z\in \BFM$ without loss of generality, since $R$ is $\op{C}(H)$-invariant.
We claim that $R\subset X$ contains $z_1v$ where $z_1\in\op{BF}\M$ and $v\in \check V-\{e\}$.
Write $x_n=zh_nr_n$ for some $h_n\in\check H$ and $r_n\in \exp \check{\frak h}^\perp $,
where $\check{ \frak h}^\perp$ denotes the $\op{Ad}(\check H)$-complementary subspace to $\op{Lie}(\check H)$ in $\frak g$.
Since $x_n\in F^*$ and $z\in
\BFM $, it follows that $r_n \notin \op{C}(H)$ for all large $n$.
By \eqref{nu} and \eqref{perp}, we have $$ \op{N}(U) \cap \exp  (\check{ \frak h}^\perp \cap \cal O)\subset \check V \op{C}(H)$$
for a small neighborhood $\cal O$ of $0$ in $\frak g$.
Therefore, if $r_n\in \op{N}(U)$ for some $n$,
then the $\check V$-component of $r_n$ should be non-trivial.
Hence by Theorem \ref{thm.ratner},
$X\supset \cl{zh_n U} r_n =zh_n L r_n$ for some $L\in \mathcal Q_U$ contained in $\check H$.
Note that $x_n=zh_n r_n\in F^*$ and that $r_n^{-1}Lr_n\in \mathcal Q_U-\mathcal L_U$,
since $r_n\in \check V-\{e\}$. 
Hence the claim follows from Lemma \ref{alpha}.

Now suppose that $r_n\not\in\op{N}(U)$ for all $n$. 
Then there exists a
one-parameter subgroup $U_0=\{u_t\}<U$ 
such that
$r_n\not\in\op{N}(U_0)$. 
Applying Lemma \ref{lem.QR2}, with a sequence of $k$-thick subsets
$$ \mathsf T(x_n):=\{t\in \br : x_nu_t\in\RFM\},$$
we get a sequence $t_n\in \mathsf T(x_n)$ such that $u_{t_n}^{-1} r_n u_{t_n}$ converges to  non-trivial element $v\in \check V$.
 Since $zh_n u_{t_n}\in z\check H$ and $z\check H$ is compact, the sequence
 $zh_nu_{t_n}$ converges to some $z_1\in z\check H$, after passing to a subsequence.
Then \be \label{zone} z_1v=\lim (zh_n u_{t_n})(u_{t_n}^{-1} r_n u_{t_n}) \in X\cap\RFM .\ee
Since $z_1\in \BFM$ and $v\in \check V-\{e\}$, $z_1v\in \RFM$ implies that $z_1v\in F^*$, and hence $z_1v\in R$.
This proves the claim. 

Now by Theorem \ref{thm.ratner}, $\cl{z_1U}=z_1 L $ for some $L\in \mathcal Q_U$ contained in $\check H$,
and hence $$X\supset \overline{z_1vU}=\overline{z_1U}v =(z_1 v)(v^{-1} L v).$$
Since $v\in \check V-\{e\}$, $v^{-1}Lv\notin \mathcal L_U$. Therefore, by Lemma \ref{alpha}, it suffices to prove
that there exists $u\in U$ such that
 \be\label{fff} (F^*\cap \RFM) \cap \limsup_{t\to +\infty} z_1uv a_{-t} \ne \emptyset.\ee
 Let $g_1\in G$ be such that $z_1=[g_1]$, and set $A_{(-\infty, -t]}:=\{a_{-s}: s\ge t\}$ for $t>0$.
Since  $z_1 v\in F^*\cap \RFM$,
the sphere $(gvU)^-\cup g^+$ intersects $\La -\bigcup_{i} \cl {B_i}$ non-trivially.
Let $u\in U$ be an element such that $(gv u)^-\in \La -\bigcup_{i} \cl {B_i}$. As $z_1vu\in \RFM$,
$\pi(zuvA)\subset \core {\M}$.
Take $\e>0$ small enough so that
 the $\epsilon$-neighborhoods of $\hull B_j$'s are mutually disjoint.
 If \eqref{fff} does not hold for $z_1uv$, then there exists $t>1$ such that
 the geodesic ray $\pi(z_1vuA_{(-\infty, -t]})$ is contained in the $\epsilon$-neighborhood of
 $\partial \core {\M}$ (cf. proof of Lemma \ref{corea}).  
As $\pi(g_1uvA_{(-\infty, -t]})$ is connected, there exists $B_j$ such that
$\pi(g_1uvA_{(-\infty, -t]})$ is contained in the $\epsilon$-neighborhood of $\op{hull} B_j$.
This implies that $(g_1uv)^-\in \partial B_j$, yielding a contradiction.
This proves \eqref{fff}.
\end{proof}

\begin{proposition}\label{prop.CFtoC2}\label{cs}
The orbit $xH'$ is not closed in $F^*$.
\end{proposition}
\begin{proof} 
Suppose that $xH'$ is closed in $F^*$. Since we are assuming that $xH'$ is not closed in $F$,
$\cl{xH'}$ contains some point $y\in \partial F$.
Since $\partial F=\BFM \check V^+ \op{C}(H)$,
we may assume $y\in \BFM \cdot \check V^+$.
Write $y=zv$ where $z\in \BFM$ and $v\in \check V^+$.
If $v\ne e$, $\cl{zvH'}$ intersects $\BFM$ by Theorem \ref{b-ratner}.
Therefore $\cl{xH'}$ always contains a point of $\BFM$, say $z$.
Let $x_n\in xH'$ be a sequence converging to a point $z$.
Since $xH'\subset F^*$, there exist $k_n\in H\cap K$ converging to some $k\in H\cap K$
such that $x_nk_n\in xH'\cap \RFPM$
and $x_n k_n \to zk$. 
Then $zk\in \BFM \cdot H'=\BFM \op{C}(H)$. Since $x_nk_n\in \RFM \cdot U$ by Lemma \ref{ru},
there exists $u_n\in U$ such that $x_nk_nu_n$ belongs to $\RFM$ and converges to a point in $\partial F$ by Lemma \ref{geo}.
Hence $X\cap F^*\cap \RFM$ accumulates on $\partial F$.
Now the claim follows from  Proposition \ref{accdd}.
\end{proof}

This proposition implies that
  \be\label{really} (X-xH')\cap (F^*\cap \RFM)\ne \emptyset.\ee

Roughly speaking,
our strategy in proving $(1)_{m+1}$ is first to find a closed $L$-orbit $x_0L$ such that $x_0L\cap F$ is contained in $X$ for some $L\in \mathcal L_U$. If $X\ne x_0L\op{C}(H) \cap F$, then we  enlarge $x_0L$ to a bigger closed orbit $x_1 \widehat L$ for some $\widehat L\in\cal L_{\widehat U}$ for some $\widehat U$ properly containing $U$,
such that $x_1\widehat L\cap F$ is contained in $X$.

It is in the enlargement step where Proposition \ref{cor.lin} (Additional invariance I) is a crucial ingredient of the arguments.
In order to find a sequence $x_i$ accumulating on a generic point of $x_0L$ satisfying the hypothesis of the proposition, 
we find a closed orbit $x_0L$ with a base point $x_0$ {\it in $ F^*\cap \RFM$}, and enlarge it to a bigger closed orbit, again based at a point {\it in} $F^*\cap \RFM$. The advantage of having a closed orbit $xL$ with $x\in F^*\cap \RFM$ is that any $U_0$-generic
 point in $xL\cap\RFM$ can be  approximated by a sequence of $\RFM$-points in $F^*\cap xL$ by Lemma \ref{lem.frameshift}.
The enlargement process must end after finitely many steps because of dimension reason.  

\subsection*{Finding a closed orbit of $L\in \mathcal L_U$ in $X$}

\begin{proposition}\label{prop.R6}\label{closed}
There exists a  closed orbit $x_0L$ with $x_0\in F^*\cap \RFM$ and $L\in \mathcal L_U$ such that
 \begin{equation*}
x_0L\cap\op{RF}_+\M\subset X .
\end{equation*} 

\end{proposition}
\begin{proof}
Let $R:=X\cap F^*\cap \op{RF}\M$. 
If $R$ is non-compact, the claim follows from Proposition \ref{accdd}.
Now suppose that $R$ is compact.
By $(2.a)_m$, it is enough to show that $X$ contains an orbit $z\widehat{U}$, and hence
$z\widehat U A$, for some $\widehat{U}<N$ properly containing $U$ and $z\in R$.
By Proposition \ref{prop.YVIX}, it suffices to find a $U$-minimal subset $Y\subset X$ with respect to $R$ and a point $y\in Y\cap R$ such
  that $X-yH'$ is not closed.

If $xH'$ is not locally closed, then take
  any $U$-minimal subset $Y$ of $X$ with respect to $R$. If $Y\cap R\subset xH'$, then choose any $y\in Y\cap R$.
  Then $X-yH'=X-xH'$ cannot be closed, as $xH'$ is not locally closed.
  If $Y\cap R\not\subset xH'$, then choose $y\in (Y\cap R)-xH'$. Then $X-yH'$ contains $xH'$
  and hence cannot be closed.
  
If $xH'$ is locally closed, then $X-xH'$ is a closed $H'$-invariant subset which intersects $R$ non-trivially.
So we can take a $U$-minimal subset $Y\subset X-xH'$ with respect to $R$. Take any $y\in Y\cap R$.
Then $X-yH'$ is not closed.
 \end{proof}

\subsection*{Enlarging a closed orbit of $L\in \mathcal L_U$ in $X$}
\begin{proposition}\label{prop.B2'}\label{prop.baseptinR}\label{enlarge}
Assume that $(3)_m$ holds as well.
Suppose that there exists a closed orbit $x_0L$ for some $x_0\in F^*\cap \RFM$ 
and $L\in\cal{L}_U$ such that 
\begin{equation}\label{eq.v1}
x_0L\cap\op{RF}_+\M\subset X  \text{ and }  X\ne  x_0L\cdot\op{C}(H)\cap F.
\end{equation}
Then there exists a closed orbit $x_1\widehat{L}$ for some $x_1\in F^*\cap \RFM$, and $\widehat{L}\in\cal{L}_{\widehat U}$ for some $\widehat U<N$ with $\op{dim}\widehat U>\op{dim}(L\cap N)$  such that
\begin{equation*}
x_1\widehat{L}\cap\op{RF}_+\M\subset X.
\end{equation*}
\end{proposition}

\begin{proof} 
Note that if $ X\subset x_0L\cdot\op{C}(H)$, then $X= x_0L\cdot\op{C}(H) \cap F$.
Indeed, this can be seen from the identity $x_0L\cdot \op{C}(H)\cap F=(x_0L\cap \RFPM)\op{C}(H)$.
Therefore we assume that $X\not\subset x_0L\cdot\op{C}(H)$.
First note that the hypothesis implies that $L\ne G$, and hence $\op{co-dim}_{L\cap N}(U)\le m$.
Let $U_-^{(1)},\cdots,U_-^{(\ell)}$ be one-parameter subgroups generating $U$.
Similarly, let $U^{(1)}_+,\cdots,U^{(\ell)}_+$ be one-parameter subgroups generating $U^+$.
By Theorem \ref{thm.new5}, 
$\bigcap_{i=1}^{\ell}\mathscr{G}(U^{(i)}_\pm,x_0L)\neq\emptyset.$
Therefore without loss of generality, we can assume 
\be\label{assume}x_0\in \bigcap_{i=1}^\ell\mathscr{G}(U^{(i)}_\pm,x_0L).\ee
Let us write $L=H(\til{U}) C$ for some $\til U<N$ and a closed subgroup $C$ of $\op{C}(H(\til U))$.
Note from the hypothesis that we have
\begin{equation*}
(x_0L\cap\op{RF}_+\M)\cdot H'\subset X.
\end{equation*}
Observe that \eqref{eq.v1} implies that $x\not\in x_0L\cdot H'=x_0L\cdot\op{C}(H)$. Since $C<\op{C}(H)$, we have $x\not\in x_0H(\til{U})$.
Now choose a sequence $w_i\in H'$ such that $xw_i\to x_0$, as $i\to\infty$.
Write $xw_i=x_0g_i$ where $g_i\to e$ in $G-LH'$. Let us write $g_i=\ell_ir_i$ where $\ell_i\in L$, and $r_i\in \exp \mathfrak l^\perp$.
In particular, $r_i\not\in\op{C}(H)$. Let $x_i=x_0\ell_i$, so that $x_ir_i\in X$.

 We  claim that we can assume that 
$x_i\in \RFM \cap x_0L$, $r_i\not\in\op{C}(H)$, and $x_ir_i\in X$.
Since $x_0\in F^*$, by Lemma \ref{lem.frameshift}, we can find $w_i'\to w'\in H$ such that $x_0\ell_iw_i'\in\op{RF}\M$, 
 and $x_0w' \in \bigcap_{i=1}^\ell \mathscr{G}(U^{(i)}_\pm,x_0L)$;
 hence
\begin{equation*}
\cl{x_0w'U}=x_0L\cap\op{RF}_+\M.
\end{equation*}
Writing $x_i'=x_0\ell_iw_i'$ and $r_i'=w_i'^{-1}r_iw_i'$, we have
\begin{equation*}
x_i'r_i'=xw_iw_i'\in X,
\end{equation*}
where $x_i'\to x_0w'$ in $x_0L\cap\op{RF}\M$, and $r_i'\to e$ in $ \exp \mathfrak l^\perp$.
Since $F^*$ is $H'$-invariant, we have $x_0w'\in F^*$.
Since $F^*$ is open and $x_0w'\in F^*$, it follows that $x_i'\in X\cap\op{RF}\M\cap F^*$ for sufficiently large $i$.
Note that $r_i'\not\in\op{C}(H)$, as $r_i\not\in\op{C}(H)$.
This proves the claim.
 
We may assume $r_i\notin \op{N}(U)$ for all $i$, up to switching the roles of $U$ and $U^+$, by Lemma \ref{hen}.
Note that  $x_i\to x_0$ in $\RFM\cap x_0L$ and $x_0$ satisfies \eqref{assume}. As we are assuming $(2)_m$, and $(3)_m$, we may now apply 
Proposition \ref{cor.lin} to the sequence $x_0\ell_ir_i\to x_0$ to obtain
 a non-trivial element  $v\in \til{U}^\perp$ such that
$$x_0L v \cap\op{RF}_+\M \subset X.$$ 
Since $x_0\in F^*\cap \RFM$, it follows from Lemma \ref{onev2} that
 there exist $x_2\in F^*\cap \RFM$ and a connected closed subgroup
$\widehat U<N$ properly containing $L\cap N$ such that
$$x_2\widehat UA\subset X .$$ 
Since $\op{co-dim}_N(\widehat U)\leq m$, it remains to apply $(2.a)_m$ to finish the proof of the proposition.
\end{proof}

\subsection*{Proof of $(1)_{m+1}$}
Combining Propositions \ref{closed} and \ref{enlarge}, we now prove:
\begin{theorem}\label{lem.indRA1}
If $(2)_m$ and $(3)_m$ are true,
then $(1)_{m+1}$ is true.
\end{theorem}
\begin{proof} Recall that we only need to consider the case
$X=\overline{xH'}$ where $x\in F^*$ and $xH'$ is not closed in $F^*$.
By Proposition \ref{prop.R6},
 there exists $x_0\in F^*\cap \RFM$ and $L\in\cal{L}_U$ such that $x_0L$ is closed and
\begin{equation*}
x_0L\cap\op{RF}_+\M\subset X.
\end{equation*}
Since $X$ is $H'$-invariant, it follows
\begin{equation}\label{eq.v2}
(x_0L\cap\op{RF}_+\M)\cdot H'\subset X.
\end{equation}
Note that $(x_0L\cap\op{RF}_+\M)\cdot H'=x_0L\cdot\op{C}(H)\cap F$ is a closed set.
We may assume the inclusion in \eqref{eq.v2} is proper, otherwise we have nothing further to prove.
Then by Proposition \ref{prop.B2'}, there exists $\widehat L\in\cal{L}_{\widehat U}$ for some $\widehat U<N$ properly containing $L\cap N$, and a closed orbit $x_1\widehat L$ with $x_1\in F^*\cap \RFM$ such that
$
x_1\widehat L\cap\op{RF}_+\M\subset X.
$ If
\begin{equation*}
(x_1\widehat L\cap\op{RF}_+\M)\cdot\op{C}(H)\neq X,
\end{equation*}
then we can apply Proposition \ref{prop.B2'} on
\begin{equation*}
x_1\widehat L\cap\op{RF}_+\M\subset X,
\end{equation*}
as $\cal{L}_{\widehat U}\subset\cal{L}_U$.
Continuing in this fashion, the process terminates in a finite step for a dimension reason, and hence
\begin{equation*}
X=(x_1\til L\cap\op{RF}_+\M)\cdot H'=x_1\til L\cdot\op{C}(H)\cap F
\end{equation*}
for some $\til L\in\cal L_U$, completing the proof.
\end{proof}

\section{$U$ and $AU$-orbit closures: proof of $(2)_{m+1}$}\label{s:2}
In this section, we fix
  a closed orbit $x_0\widehat L$  for $x_0\in F^*$ and $\widehat L\in\cal L_U$.
   Let $U< \widehat L\cap N$ be  a connected closed subgroup with $ \op{co-dim}_{\widehat L\cap N}U\le m+1$.
By replacing $U$ and $\widehat L$ by their conjugates using an element $m\in M$,
we may assume that $$U\subset \widehat L\cap  \check H \cap N.$$

We keep the same notation $H, F, \partial F, F^*$ etc from section \ref{s:1}.
If $x\in \RFPM\cap \partial F$ (resp. if $x\in \RFM\cap \partial F$), then $(2.a)$ (resp. $(2.b)$) follows from Theorem \ref{thm.ratner}.

We fix $x\in \RFM \cap x_0\widehat L\cap F^*$, and
set \be\label{as3} X:=\overline{xU}\text{ and assume that } X \ne x_0\widehat L\cap \RFPM.\ee
This assumption implies that $U$ is a proper connected closed subgroup of $\widehat L\cap N$ and hence
 $\op{dim}(\widehat L\cap N)>\op{dim} U \ge 1$.

By Proposition \ref{SAVE},
either $x_0\widehat L$ is compact or $ \mathscr{S}(U,x_0\widehat L)$ contains
a compact orbit $zL_0$ with $L_0\in \mathcal L_U$.
If $x_0\widehat L$ is compact, then $(2)_{m+1}$ follows from Theorem \ref{thm.ratner}. Therefore
we assume  in the rest of the section that
\be\label{compact} \text{ $\mathscr{S}(U,x_0\widehat L)$ contains
a compact orbit $zL_0$ with $L_0\in \mathcal L_U$. }\ee

\begin{lemma}\label{lem.AUbarmeetsSU}
Assume that $(1)_{m+1}$ and $(2)_m$ hold.
Then $$\cl{xAU}\cap \mathscr{S}(U,x_0\widehat L)\ne \emptyset .$$
\end{lemma}
\begin{proof} 
Since $(1)_{m+1}$ is true, we have
$$\cl{xH}= x Q\cap F$$
for some $Q\in \mathcal L_U$ such that $xQ$ is closed.
By Lemma \ref{LLL}, $Q<\widehat L$. It follows from Lemma \ref{sin3} that
either $Q=\widehat L$ or $\op{dim} (Q\cap N )<\op{dim}(\widehat L\cap N)$.
Suppose that $Q=\widehat L$. By \eqref{compact}, there exists a compact orbit $zL_0\subset \mS(U, x_0\widehat L)$ 
for some $L_0\in \mathcal L_U$. 
On the other hand, $x_0\widehat L\cap F=\cl{xH}=\overline{xAU} (K\cap H)$. 
Hence for some $k\in K\cap H$, $zk\in \overline{xAU}$.
Since $H\subset L_0$, $zk\in zL_0$.  So $\overline{xAU}$ intersects $zL_0$, proving the claim.
If $\op{dim} (Q\cap N) <\op{dim}(\widehat L\cap N)$.
 then $\overline{xAU}\subset xQ\subset \mS(U, x_0\widehat L)$.
\end{proof}

\begin{lemma}\label{xUone}
Assume that $(1)_{m+1}$ and $(2)_m$ hold.
Then $$\cl{xU}\cap \mathscr{S}(U,x_0\widehat L) \ne \emptyset .$$ \end{lemma}
\begin{proof}
Since  \be\label{nothing}\left(x_0\widehat L\cap\op{RF}_+\M\right) -F^*\subset\mathscr{S}(U,x_0\widehat L),\ee it suffices to consider
the case when $X:=\cl{xU}\subset F^*$.
Let $Y\subset X$ be a $U$-minimal set with respect to $\RFM$.
Since $Y\subset F^*$, by Proposition \ref{prop.YLY}, there exists an unbounded one-parameter subsemigroup $S$ inside $AU^\perp \op{C}_2(U)\cap\widehat L$ such that $YS\subset Y$.
In view of Lemma \ref{rmk.1psg},
we could remove $\op{C_2}(U)$-component of $S$
so that $S$ is either of the following  \begin{itemize}
  \item $v^{-1}A^+v$ for a one-parameter semigroup $A^+\subset A$ and $v\in U^\perp\cap\widehat L$; 
  \item
   $V^+$ for a one-parameter semigroup $V^+\subset U^\perp\cap\widehat L$,
  \end{itemize}
  and
  $$YS\subset X(\op{C}_2(U)\cap \widehat L).$$
   Since $\mathscr{S}(U,x_0\widehat L)$ is invariant by $N\op{C}_2(U)\cap\widehat L$,
  it suffices to show that $$X(N \op{C}_2(U)\cap \widehat L)\cap \mathscr{S}(U,x_0\widehat L)\ne \emptyset .$$
If $S=v^{-1}A^+v$,  then $Yv^{-1}A^+\subset Xv^{-1}(\op{C}_2(U)\cap\widehat L)$.
Choose $y\in Y$.
We may assume that $yv^{-1}\in F^*$ by \eqref{nothing}.
Then, replacing $y$ with an element in $yU$ if necessary, we may assume $yv^{-1}\in\op{RF}\M\cap F^*$.
Choose a sequence $a_n\to\infty$ in $A^+$. Then $yv^{-1}a_n $ converges to some $y_0\in \RFM$
by passing to a subsequence. Since $
\liminf a_{n}^{-1}A^+=A$, and
\begin{equation*}
(yv^{-1} a_n)(a_n^{-1}A^+)\subset Xv^{-1}(\op{C}_2(U)\cap\widehat L),
\end{equation*}
we obtain that
 $$y_0A\subset Xv^{-1}(\op{C}_2(U)\cap \widehat L).$$
Since  $\cl{y_0AU}\subset Xv^{-1} (\op{C}_2(U)\cap\widehat  L)$ and $\cl{y_0AU}$ meets $\mathscr{S}(U,x_0\widehat L)$ by  Lemma \ref{lem.AUbarmeetsSU}, the claim follows.

Next, assume that  $S=V^+$, so that $YV^+\subset\ X\op{C}_2(U)\cap\widehat L$.
Let $v_n\to\infty$ be a sequence in $V^+$.
We have  $Yv_n\subset X\subset F^*$.
Together with the fact $Yv_n$ is $U$-invariant, this implies $Yv_n$ meets $\RFM$.
Note that $$Yv_n(v_n^{-1}V^+)\subset X (\op{C}_2(U)\cap\widehat L).$$
Choose $y_n\in Yv_n\cap\RFM$.
As $\RFM$ is compact, $y_n$ converges to some $y_0\in\RFM$, by passing to a subsequence, and hence 
$$y_0UV\subset X (\op{C}_2(U)\cap\widehat L) .$$
Since $\op{co-dim}_N(UV)\leq m$, the conclusion follows from $(2)_m$.
\end{proof}

\begin{lemma}\label{lem.UbarmeetsSU}\label{onetwo}
Assume that $(1)_{m+1}$ and $(2)_m$ hold.
Then  $$\cl{xU}\cap \mathscr{S}(U,x_0\widehat L)\cap F^*\ne \emptyset .$$
  \end{lemma}
\begin{proof}
By Lemma \ref{xUone}, there exists $y\in\cl{xU}\cap\mathscr{S}(U,x_0\widehat L)$.
Hence by $(2)_m$, 
$$\overline{yU}=yL\cap \RFPM\subset \overline{xU}$$
for some  $L\in\cal{Q}_U$ properly contained in $\widehat L$.
Consider the collection of all subgroups $L\in \cal Q_U$ such that
$yL\subset \overline{xU}$ for some $y\in \RFPM$. Choose $L$ from this collection so that
$L\cap N$ has  maximal dimension.
If $yL\cap F^*\neq\emptyset$, then the claim follows.

Now suppose that $yL\subset \partial F$.
As $y\in \RFPM\cap \partial F$, we have
$$y=zv_0 c_0$$ for some $z\in \BFM$, $v_0\in \check V^+$ and $c_0\in \op{C}(H)$.
Since $y\in\cl{xU}$, there exists $u_i\in U$ such that $xu_i$ converges to $ y$ as $n\to\infty$.
Set $$z_i:=xu_ic_0^{-1}v_0^{-1}\in \cl{xU} c_0^{-1}v_0^{-1}$$
so, $z_i\to z.$
 As $v_0\in \check V^+$ and hence $v_0^{-1}\in \check V^-$ and $xu_i\in F^*$, we have $z_i\in F^*\cap \RFPM \subset \RFM \cdot U$.
By Lemma \ref{geo}, we may modify $z_i$ by elements of $U$ so that $z_i \in \RFM$ and $z_i$ converges to some $z_0\in z\check H$.
Write $z_i=z_0\ell_ir_i$ for some $\ell_i\in \check H$ and $r_i\in \exp \check{\mathfrak h}^\perp $ converging to $e$.
Since $z_i\in F^*$ and $z_0\ell_i \in \partial F$, we have $r_i \ne e$.
By Theorem \ref{thm.ratner},  we have $\cl{z_0\ell_i U} =z_0\ell_i L_i$ for some $L_i \in \mathcal Q_U $ contained in $\check H$.

\noindent{\bf Case 1: $r_i\in\op{N}(U)$ for some $i$}.
Then 
$$\cl{xU}=\cl{z_0 \ell_i r_i v_0c_0 U}=\cl{z_0\ell_i U} (r_i v_0 c_0)={z_0\ell_i L_i} (r_i v_0 c_0) .$$
As $\cl{xU}\ne x_0\widehat L$ by the hypothesis,
it follows that $x\in \mS(U, x_0\widehat L)\cap F^*$, proving the claim.

\noindent{\bf Case 2: $r_i\not\in\op{N}(U)$ for all $i$}.
Then
 there exists a one-parameter subgroup $U_0<U$  such that $r_i\not\in\op{N}(U_0)$ for all $i$, by passing to a subsequence.

By Lemma \ref{lem.QR2}, we can find $u_{t_i}\to \infty$ in $U_0$ so that
$z_i u_{t_i}\in \RFM$ and $u_{t_i}^{-1} r_i u_{t_i}$ converges to a non-trivial element $v\in \check V$, whose size is strictly bigger than
$\|v_0\|$.
As $z_0\ell_i u_{t_i}$ is contained in the compact subset $z_0\check H$, we may assume that $z_0\ell_i u_{t_i}$ converges to
some $z'\in z_0\check H$.
Hence $$z_i u_{t_i}= z_0\ell_i u_{t_i} (u_{t_i}^{-1} r_i u_{t_i}) \to z' v\in \RFM \cap\cl{xU} c_0^{-1}v_0^{-1}.$$
Since $z'\in \BFM$ and $z'v\in \RFM$, we have $v\in \check V^-$.

By Theorem \ref{thm.ratner}, $\overline{z' U}=z'Q_1$ for some $Q_1\in \cal Q_U$.
Since $z' vv_0c_0\in \cl{xU}$, we get $$\cl{xU}\supset z' Q_1 (v v_0) c_0. $$
Since the size of $v$  is larger than the size of $v_0$,
then $v v_0$ is a non-trivial element of $ \check V^-$. Since $z'Q_1\subset \BFM$,
 the closed orbit $z' Q_1 (v v_0) c_0$ meets $F^*$.
 Hence the claim follows.
\end{proof}

\begin{thm}\label{lem.indRA3}\label{twom}
Assume that $(1)_{m+1}$, $(2)_m$, and $(3)_m$ are true.
Then $(2)_{m+1}$ is true.
\end{thm}
\begin{proof} We first show $(2.a)_{m+1}$ holds for $X=\cl{xU}$. 
By Lemma \ref{lem.UbarmeetsSU} and  $(2)_m$, there exists a closed orbit
$yL$ with $y\in F^*$ and $L\in\cal Q_U$ 
such that
$$\cl{xU}\supset yL\cap\op{RF}_+\M$$ and
 $L\cap N\neq \widehat L\cap N$. We choose $L\in \cal Q_U$ so that
 $\op{dim}(L\cap N)$ is maximal.
 Note that $\op{co-dim}_{L\cap N}U\le m$. By Theorem \ref{thm.new5}, we can assume that
\be\label{ygeneric} y\in\bigcap_{i=1}^{\ell}\mathscr{G}(U^{(i)},yL)\cap F^*\cap \RFM\ee where $U^{(1)},\cdots, U^{(\ell)}$ are one-parameter subgroups generating $U$.
As $y\in\cl{xU}$, there exists $u_i\in U$ such that $xu_i\to y$ as $i\to\infty$.
Since $y\in F^*$, we can assume $xu_i\in\RFM$ after possibly modifying $u_i$ by Lemma \ref{lem.geometric}.
We will write $xu_i=y\ell_ir_i$ where $\ell_i\in L$ and $ r_i\in \exp \mathfrak l^\perp \cap\widehat L$.

\noindent{\bf Case 1: $r_i\in\op{N}(U)$ for some $i$}.
Then $y\ell_i\in \RFPM$ and $X=\overline{xu_iU}=\overline{y\ell_i U}r_i$.
Since $y\ell_i U\subset yL$, and $\op{co-dim}_{L\cap N} (U)\le m$, we have
$$X= \overline{y\ell_i  Ur_i}=y\ell_i L' r_i\cap \RFPM$$ for some $L'\in \mathcal Q_U$, proving the claim.

\noindent{\bf Case 2: $r_i\not\in\op{N}(U)$ for all $i$}.
By \eqref{ygeneric}, we can apply Proposition \ref{lem.lin} to the sequence $xu_i\to y$ and obtain 
a sequence $v_j\to\infty$ in $(L\cap N)^\perp$ such that
$$yL v_j \cap\op{RF}_+\M\subset X.$$
Since $y\in F^*$,  by Lemma \ref{lem.epsiloncore2}, there exists a one-parameter subgroup
$V\subset (L\cap N)^\perp$ such that $y_1(L\cap N) V\subset X$ for some $y_1\in F^*\cap \RFM$.
Hence, by $(2)_m$, we get a contradiction to the maximality of $L\cap N$; this proves $(2.a)_{m+1}$.

Now we show $(2.b)_{m+1}$ for  the closure $\cl{xAU}$. By $(1)_{m+1}$, we have $\cl{xH}=xL\cap F$ for some $L\in\cal L_U$ contained in $\widehat L$.
Hence $\cl{xAU}\subset xL\cap \RFPM$.
It suffices to show that
\be\label{final} \cl{xAU}=xL\cap\RFPM .\ee
 If $U=L\cap N$,
then $\cl{xU}=xL\cap \RFPM$ by Theorem \ref{tm}, which implies \eqref{final}.
So, suppose that $U$ is a proper closed subgroup of $L\cap N$.
Since $\cl{xAU} (K\cap H)= \cl{xH}=xL\cap F$, it follows from Lemma \ref{lem.SHU}  that we can choose  $y\in\cl{xAU}\cap\mathscr{G}(U,xL)$.
By $(2.a)_{m+1}$ and
Lemma \ref{lem.gp},
we have $\cl{yU}=xL\cap\RFPM$, finishing the proof. \end{proof}

\section{Topological equidistribution: proof of $(3)_{m+1}$}\label{s:3}
In this section, we prove $(3)_{m+1}$.
  Let $U< N$ be  a non-trivial connected closed subgroup.
Let $x_0\widehat L$ be a closed orbit for $x_0\in F^*\cap \RFM$ and $\widehat L\in\cal L_U$ such that $\op{co-dim}_{\widehat L \cap N}(U)=m+1$.
As before we may assume that $U\subset \widehat L\cap  \check H \cap N.$

Let $x_i L_i\subset x_0\widehat L$ be a sequence of closed orbits intersecting $\RFM$ where
$x_i\in\RFPM$, $L_i\in\cal{Q}_U$. We write $x_iL_i$ as $y_iL_iv_i$ where
$y_i\in\RFPM$, $L_i\in\cal{L}_U$, and $v_i\in 
(L_i\cap N)^\perp \cap \widehat L$.  
Assume that no infinite subsequence of $y_iL_iv_i$ is contained in a subset
of the form $y_0L_0D\subset \mathscr S(U, x_0\widehat L) $ where $y_0L_0$ is a closed orbit for some $L_0\in \cal L_U$ and $D\subset \op{N}(U)$
is a compact subset.
Let
\begin{equation*}
E=\limsup\limits_{i\to\infty}\text{ }(y_i L_iv_i\cap\op{RF}_+\M).
\end{equation*}
Note that $\liminf_{i\to \infty} (y_i L_iv_i\cap\op{RF}_+\M)$ coincides with the intersection 
of the subsets  $\limsup (y_{i_k} L_{i_k}v_{i_k}\cap\op{RF}_+\M)$ for all infinite subsequences $\{i_k:k\in\bb N\}$ of $\bb N$. 
If the hypothesis of $(3)_{m+1}$ holds for a given sequence $y_i L_iv_i$, then it also holds for all subsequences.
Hence to prove $(3)_{m+1}$, it suffices to show that $$E=\RFPM\cap x_0\widehat L.$$
 We note that by $(3)_m$, we may assume that 
$$L_i\cap N= U\quad\text{for all $i$}.$$ This in particular implies
that  each $y_i L_iv_i\cap \RFPM$ is $U$-minimal by Theorem \ref{tm}.

\begin{lemma}\label{lem.limsup}
Assume that  $(1)_{m+1}$, $(2)_{m+1}$
and $(3)_m$ are true.
Then there exist $y\in  F^*\cap \RFM$ and
 $L\in \mathcal Q_U$ with $\op{dim}(L\cap N)>\op{dim} U$ such that $yL$ is closed and
 $$E\supset yL\cap \RFPM.$$ 
 \end{lemma}
\begin{proof}
By $(2)_m$, it suffices to show that
there exist $y_0\in F^*\cap \RFM$ and
$\widehat U<N$ properly containing $U$ such that $$E\supset y_0 \widehat U.$$
Suppose that $y_iL_iv_i\subset \partial F$ for  infinitely many $i$.
Since $y_iL_iv_i\cap \RFM\ne \emptyset$, we may assume $y_iv_i\in z_i\check H\op{C}(H)$ for some $z_i\in \BFM$ by \eqref{bfmz}.
Since $L_i\cap N=U$, we get $y_iL_iv_i=\cl{y_iU}\subset z_i\check H\op{C}(H)$ by Theorem \ref{thm.ratner}.
This contradicts the hypothesis on $y_iL_iv_i$'s.

Therefore by passing to a subsequence, for all $i$, 
 $$ y_iL_i v_i\cap \RFPM\cap F^*\ne \emptyset .$$
Since $AU<L_i$ for all $i$, it follows that
\begin{equation*}
E=\limsup\limits_{i\to\infty}(y_iL_i v_i\cap\op{RF}_+\M) (v_i^{-1}AUv_i)
\end{equation*}
By Lemma \ref{lem.epsiloncore3}, there exists $y_0\in\limsup_i (y_iL_iv_i \cap\op{RF}_+\M)\cap F^*$.
Hence
\begin{equation}\label{yxxx}
y_0\limsup_{i\to\infty} (v_i^{-1}AUv_i)\subset E,
\end{equation}
after passing to a subsequence.

If $v_i\to \infty$, then 
  $\limsup_i(v_i^{-1}AUv_i)$ contains $A\widehat U$ for some $\widehat U$ properly containing $U$ by Lemma \ref{vAv}.
Therefore, we get the conclusion $y_0\widehat U\subset E$ from \eqref{yxxx}.
Now suppose that, by passing to a subsequence, $v_i$ converges to some $v\in N\cap \widehat L$.
Then \eqref{yxxx} gives \begin{equation*}
y_0 v^{-1} AUv \subset E.
\end{equation*}
Then by $(2)_{m+1}$, $\cl{y_0v^{-1} AU}$ is of the form $y_0v^{-1} L_0\cap\RFPM$ for some $L_0\in\cal L_U$.
Hence \be\label{Indd} E\supset y_0L\cap \RFPM\ee
where  $L:=v^{-1}L_0v$.
If $L\cap N$ contains $U$ properly, this proves the claim. So we suppose that $L\cap N=U$.
By Theorem \ref{tm}, we can assume that
$y_0\in\bigcap_{i=1}^{\ell}\mathscr{G}(U^{(i)},y_0L)\cap F^*\cap \RFM$, where $U^{(1)},\cdots, U^{(\ell)}$ are one-parameter subgroups generating $U$.
By replacing $y_i$ by an element of $y_iL\cap \RFPM$,
we may assume that $y_iv_i \to y_0$. Furthermore, as $y_0\in F^*\cap \RFM$,
for all $i$ sufficiently large,
$y_iv_i\in F^*\cap \RFPM\subset \RFM \cdot U$ (as $F^*$ is open). Hence
we can also assume $y_iv_i\in \RFM$ by Lemma \ref{xfss}.
Therefore we may write
$$y_iv_i =y_0\ell_ir_i$$ for some $\ell_i\to e$ in $L$ and non-trivial $r_i\to e$ in $\exp \mathfrak l^\perp$.

Suppose that $r_i$ belongs to $\op{N}(U)$ for infinitely many $i$.
Then 
\begin{align*} y_iL_iv_i\cap \RFPM &=\overline{y_iv_iU} =\overline{y_0\ell_i U}r_i =
 y_0L r_i\cap \RFPM .\end{align*}
 Hence $y_iL_iv_i r_i^{-1}\cap \RFPM = y_0L\cap \RFPM$.
 In particular, $y_iL_iv_i r_i^{-1}\cap \RFM$ is non-empty (as it contains $y_0$) and is contained in $y_0L$.
 By  Lemma \ref{LLL}, this implies that
 $y_iL_iv_i \subset y_0Lr_i$.
 As $r_i\to e$, this contradicts the hypothesis on $y_iL_iv_i$'s.

Therefore $r_i\not\in \op{N}(U)$ for all $i$ but finitely many.
We may now apply Proposition \ref{lem.lin} and Lemma \ref{onev} to deduce that $E$ contains an orbit
 $z_0\widehat U $ for some $\widehat U<\widehat L\cap N$ containing $U$ properly and for some
 $z_0\in \RFPM\cap F^*$. This proves the claim. \end{proof}

\begin{thm}\label{prop.indRA2}
If $(1)_{m+1}$, $(2)_{m+1}$, and $(3)_m$ are true, then $(3)_{m+1}$ is true.
\end{thm}
\begin{proof}
 We claim that \be\label{yw} x_0\widehat L\cap\op{RF}_+\M= E .\ee
 By Lemmas \ref{lem.limsup},
 we can take a maximal $\widehat U$ such that
 $E\supset y\widehat U$ for some $y\in F^*\cap \RFM$. 
  By $(2)_{m}$, we get a closed orbit $y L$ for some  $L\in\cal{Q}_{\widehat U}$ such that
\begin{equation}\label{eq.k2}
y  L\cap\op{RF}_+\M\subset E.
\end{equation}

If $L=\widehat L$, then the claim  \eqref{yw} is clear. Now suppose that $L$ is a proper subgroup
of $\widehat L$. This implies that $L\cap N$ is a proper subgroup of $\widehat L\cap N$, since $\widehat L\cap N$ acts minimally on $x_0\widehat L\cap \RFPM$ as $\widehat L\in \mathcal L_U$.
By Theorem \ref{thm.new5}, we can assume that
$y\in\bigcap_{i=1}^{\ell}\mathscr{G}(U^{(i)},yL)\cap F^*\cap \RFM$, where $U^{(1)},\cdots, U^{(\ell)}$ are one-parameter subgroups generating $U$. As $y\in E$, there exists a
sequence $x_i\in y_iL_iv_i \cap \RFPM$
converging to $y$, by passing to a subsequence.
Since $U=v_i^{-1}L_iv_i\cap N$, we have $x_i\in \RFM \cdot U$.
By Lemma \ref{xfss}, 
by replacing $x_i$ with $x_iu_i$ for some $u_i\to e$ in $U$, we may assume $x_i\in \RFM$.

We claim that
  $$ x_i\notin yL \op{N}(U).$$
Suppose not, i.e., $x_i=y\ell_i r_i$ for some $\ell_i\in L$ and $r_i\in\op{N}(U)$.
Then $$y_iL_iv_i\cap\RFPM=\cl{x_iU}=\cl{y\ell_i U}r_i\subset yLr_i.$$
 By the assumption on $y_iL_iv_i$'s, this cannot happen as $r_i$'s are bounded.

 On the other hand,
 $\op{dim}(L_i\cap N)$ is strictly smaller
than $\op{dim} (L\cap N)$, since $L_i\cap N=U$ and $\widehat U<L\cap N$, yielding a contradiction.
Hence $x_i\notin yL \op{N}(U)$. 

We can now apply Proposition \ref{lem.lin} and Lemma  \ref{onev} and deduce that $E$ contains $y_1\widehat UV$ for some $y_1\in F^*\cap \RFM$. This is a contradiction to the maximality assumption on 
$\op{dim} \widehat U$.
\end{proof}

\subsection*{Proof of Theorem \ref{geo-intro2}}
We explain how to deduce this theorem from Theorem \ref{mainth}(3).
For (1),  we may first assume that $P_i$ have all same dimension so that for some fixed connected closed
subgroup $U<N$, $P_i=\pi(x_i H'(U))$
where $x_iH'(U)$ is a closed orbit of some $x_i\in \RFM$.
Then there exists $L_i\in \cal L_U$ such that $x_iL_i$ is closed  and $P_i=\pi(x_iL_i)$ by Proposition \ref{count}.
We claim that  the sequence $x_iL_i$ satisfies the hypothesis of Theorem \ref{mainth}(3).
Suppose not. Then there exists a closed orbit $y_0L_0$
with $L_0\in \cal L_U$, $L_0\neq G$ and a compact subset $D\subset \op{N}(U)$ such that $x_iL_i\subset y_0L_0D$ for infinitely many $i$.
By Lemma \ref{lem.inc}, this can happen only when $L_i\subset d_i^{-1}L_0d_i$ and $x_iL_i\subset y_0L_0d_i$ for some $d_i\in D$.
Since $D\subset\op{N}(U)\subset L_0(L_0\cap N)^\perp M$, we may assume that  $d_i\in (L_0\cap N)^\perp M$.
Since $A\subset L_i\subset d_i^{-1}L_0d_i$, we have $d_i\in M$.
This implies that $P_i=\pi(x_iL_i)\subset\pi(y_0L_0d_i)=\pi(y_0L_0)$. By the maximality assumption on $P_i$'s,
it follows that $P_i$ is a constant sequence, yielding a contradiction.
Hence by Theorem \ref{mainth}(3), $\lim (x_iL_i\cap \RFPM)=\RFPM$.
Since $\pi(\RFPM)=\Gamma\ba \bH^d$, the claim follows.
(2) follows from Corollary \ref{rcount}.
For (3), if there are infinitely many bounded properly immersed $P_i$'s, then $\lim P_i=M$ by
(1). On the other hand, $P_i\subset \core \M $; because any bounded $H'(U)$ orbit should be inside $\RFM$.
Since $\core\M$ is a proper closed subset of $M$, as $\op{Vol} (\M)=\infty$, this gives a contradiction.

\begin{Rmk}\label{nochange}
In fact, when $\M$ is any convex cocompact hyperbolic manifold of infinite volume, there are only finitely many bounded maximal
closed $H'(U)$-orbits, and hence only finitely many maximal properly immersed bounded geodesic planes.
  The reason is that if not, we will be having infinitely many maximal closed orbits $x_iL_i$ {\it contained in $\RFM$}
for some $L_i\in \mathcal L_U$, and for any $U$-invariant subset $E$ contained in $\RFM$, 
the $1$-thickness for points in $E$ holds for any one-parameter subgroup of $U$ for the trivial reason, which makes
our proof of Theorem \ref{mainth} work with little modification (in fact, much simpler)
for a general $\M$.
\end{Rmk}

\newpage  
\section{Appendix: Orbit closures for $\Gamma\ba G$ compact case}
In this section we give an outline of the proof of  the orbit closure theorem for the actions of $H(U)$ and $U$,
 assuming that $\Gamma\ba G$ is compact and there exists at least one closed orbit of $\SO^\circ (d-1,1)$.
 We hope that giving an outline of the proof of Theorem \ref{mainth} in this special case will help readers understand 
the whole scheme of the proof better and see the differences with the infinite volume case more clearly.

Note that in the case at hand, $$\RFM=F_{H(U)}^*=\RFPM=\Gamma\ba G.$$
Without loss of generality, we assume that $U\subset \SO^\circ(d-1,1)\cap N$.
\begin{theorem} \label{mainthlattice}
Let $x\in\Gamma\ba G$.

\begin{enumerate}
\item
There exists  $L\in \mathcal L_U$ such that
\begin{equation*}
\cl{xH(U)}=xL .
\end{equation*}

\item
There exists  $L\in \mathcal Q_U$ such that
\begin{equation*}
\cl{xU}=x  L.
\end{equation*}
\end{enumerate}
\end{theorem}

 The base case $(2)_0$ follows from a special case of Theorem \ref{tm}. 
  For $m\ge 0$, we will show that $(2)_m$ implies $(1)_{m+1}$, and that  $(1)_{m+1}$ and $(2)_m$ together imply $(2)_{m+1}$.

We note that when $\Gamma\ba G$ is compact, we don't need the topological equidistribution statement, which is Theorem \ref{mainth}(3)
to run the induction argument, thanks to \eqref{dmg}.  In order to prove $(1)_{m+1}$,
it  suffices to use $(2)_m$ only
when the ambient space is $\Gamma\ba G$; in
the proof of Theorem \ref{mainth}, we needed to use $(2)_m$
whenever $\op{co-dim}_{N\cap \widehat L} U\le m$ for  any closed orbit $x_0\widehat L$ containing $xU$ (this was needed in order to use results in section \ref{s:ge2}).

\begin{Rmk} Theorem \ref{mainthlattice}
is  proved by Shah \cite{Sh0} by topological arguments. Our proof presented in this appendix
is somewhat different from Shah's in that we prove that $(1)_m$ implies $(2)_m$  using the existence of a closed $\SO^\circ (d-1,1)$-orbit, while
 he shows that $(2)_m$ implies $(1)_m$.
  \end{Rmk}

\subsection*{Proof of $(1)_{m+1}$}
We assume that $1\le \op{co-dim}_N U=m+1$.
By Proposition \ref{hpu},
it suffices to show that $X:=\overline{xH'(U)}=xL \op{C}(H(U))$ for some $L\in \mathcal L_U$.
Assume that $xH'(U)$ is not closed in the following.

\subsection*{Step 1: Find a closed orbit inside $X$}
We claim that $X$ contains a $U$-minimal subset $Y$ such that
$X-yH'$ is not closed for some $y\in Y$ (cf. the case when $R$ is compact in the proof of Proposition \ref{closed}).
If $xH'(U)$ is not locally closed, then any $U$-minimal subset $Y\subset X$  does the job.
If $xH'(U)$ is locally closed, then 
any $U$-minimal subset $Y$ of $X-xH'(U)$ does the job; note that the set $X-xH'(U)$ is a compact $H'(U)$-invariant subset
and hence contains a $U$-minimal subset.

Hence, by Lemma \ref{7.6}, $X$ contains an orbit $x_0\widehat U$
with $\op{dim}\widehat U >\op{dim} U$.
By $(2)_m$ and Lemma \ref{alpha}, $X$ contains a closed orbit $zL$ for some $L\in \mathcal L_U$.
 We may assume that $X\ne zL \op{C}(H(U))$; otherwise, we are done.
\subsection*{Step 2: Enlarge a closed orbit inside $X$}

Since $zL$ is compact, by Theorem \ref{tm}, we can assume that $zU^{(i)}_{\pm}$ is dense in $zL$
where $U^{(1)}_\pm,\cdots  ,U^{(k)}_\pm$ are one-parameter subgroups of $U^\pm$ generating $U^\pm$.
Note that there exists $g_i\to e$ in $G- L \op{C}(H(U))$ such that $zg_i\in X$.
We can write $g_i=\ell_i r_i$ where $r_i\in \exp \mathfrak l^\perp $ and $\ell_i\in L$. Then $r_i\notin \op{C}(H(U))$.
Since $\bigcap_{i=1}^k (\op{N}(U^{(i)}_+)\cap \op{N}(U^{(i)}_-))\cap \exp \frak l^\perp$ is locally contained in $ \op{C}(H(U))$, we have $r_i\notin \op{N}(U_0)$
where $U_0$ is one of the subgroups $U^{(i)}_{\pm}$. If $U_0\in \{U^{(i)}_+\}$, then replace $U$ by $U^+$.

Fix any $k>1$. Applying \eqref{dmg} to the sequence $z_i:=z\ell_i \to z$,
the set 
\begin{equation}\label{tzi} \mathsf T(z_i):=\{t\in \br: z_i u_t\in \Gamma\ba G -\bigcup_{j=1}^{i}\cal O_j\}\end{equation}
is a $k$-thick subset (take $0<\epsilon<1-1/k$).
By Lemma \ref{lem.QR2}, there exists  ${t_i}\in \mathsf T(z_i)$ such that
$u_{t_i}^{-1} r_iu_{t_i}$ converges to a non-trivial element $v\in (L\cap N)^\perp$.
Now the sequence $z_iu_{t_i}$ converges to $z_0\in \mG(U_0, zL)$.
Since $zg_i u_{t_i}$ converges to $z_0v$,
we deduce
$$zL v=\overline{z_0 v U_0 }\subset X\text{ and hence } zLV^+\subset zL(AvA)\subset X$$
where $V^+$ is the one-parameter unipotent subsemigroup contained in $AvA$. Take any sequence $v_i\to \infty$ in $V^+$ such that $zv_i$ converges to some $x_0$.
Then $x_0V\subset \limsup (zv_i)(v_i^{-1}V^+)\subset X$ and hence $X$ contains $x_0(L\cap N)V$.
By the induction hypothesis $(2)_m$ and Lemma \ref{alpha}, $X$ contains a closed orbit of $\widehat L$ for some $\widehat L\in \mathcal L_{\widehat U}$. This process of enlargement must end after finitely many steps.

\subsection*{Proof of $(2)_{m+1}$} Set $X:=\overline{xU}$.
We assume that $X\ne \G\ba G$. Since the co-dimension of $U$ in $ N$ is at least $1$,
we may assume without loss of generality that $U< N\cap \SO^\circ (d-1,1)$ using conjugation by an element of $M$.

\subsection*{Step 1: Find a closed orbit inside $X$}
 By the hypothesis on the existence of a closed $L_0:=\SO^\circ (d-1,1)$-orbit, $\mS( U)\ne \emptyset$.
It  follows from $(1)_{m+1}$, $(2)_m$,  and
  the cocompactness of $AU$ in $H'(U)$ 
 that any ${AU}$-orbit closure intersects $\mS( U)$ (cf. proof of Lemma \ref{lem.AUbarmeetsSU}).

We claim that $X$ intersects $\mS( U)$. Since $\mS(U)$ is $N\op{C}_2(U)$-invariant,
it suffices to show that $X N\op{C}_2(U)$ intersects $\mS(U)$.  Let $Y\subset X$ be a $U$-minimal subset.
Then  there exists a one-parameter subgroup $S<AU^\perp \op{C}_2(U)$ such that
$Yg=Y$ for all $g\in S$ by Lemma \ref{YLY}.
Strictly speaking, the cited lemma gives $Yg\subset Y$ for $g$ in a semigroup $S$, but in the case at hand,
$Yg\subset Y$ implies $Yg=Y$, since $Yg$ is $U$-minimal again, and hence $Yg^{-1}=Y$ as well.
In view of Lemma \ref{rmk.1psg}, we get $YA\subset XN \op{C}_2(U)$ or $YvA\subset XN \op{C}_2(U)$ for some $v\in N$.
In either case, $XN \op{C}_2(U)$ contains an $AU$-orbit and hence intersects $\mS(U)$. So
the claim follows. Since $X$ intersects $\mS( U)$, by applying $(2)_m$,
$X$ contains a closed orbit $zL$ for some $L\in \mathcal Q_U$. 

\subsection*{Step 2: Enlarge a closed orbit inside $X$}
Suppose $L\ne G$ and $X\ne zL$.
It suffices to show that $X$ contains a closed orbit of $\widehat L$ for some $\widehat L\in \mathcal L_{\widehat U}$
for some $\widehat U$ properly containing $L\cap N$.
We may assume $X\not\subset zL \op{C}(H(U))$; otherwise, the claim follows from $(2)_m$.
We may assume $z\in \bigcap_{i=1}^{\ell} \mG(U^{(i)}, yL)$ where $U^{(i)}$'s are one-parameter generating subgroups of $U$.
Take a sequence $xu_i\to z$ where $u_i\in U$, and write $xu_i=z\ell_i r_i$
where $\ell_i\in L$ and $r_i\in \exp ( \mathfrak l^\perp)$.
The case of $r_i\in \op{N}(U)$ for some $i$ follows from $(2)_m$ (cf. proof of Lemma \ref{lem.indRA3}). Hence 
we may assume $r_i\notin \op{N}(U)$, and by passing to a subsequence, $r_i\notin \op{N}(U_0)$
for some $U_0\in \{U^{(i)}\}$.

Fix any $k>1$. Then $\mathsf T(z_i)$ as in \eqref{tzi}
is a $k$-thick subset.  We now repeat the same argument of a step in the proof of $(1)_{m+1}$.
By Lemma \ref{lem.QR2}, there exists  ${t_i}\in \mathsf T(z_i)$ such that
$u_{t_i}^{-1} r_iu_{t_i}$ converges to a non-trivial element $v\in U^\perp$.
Now the sequence $z_iu_{t_i}$ converges to $z_0\in \mG(U_0, zL)$.
Hence $X\supset \overline{z_0 (L\cap N)} v=zL v$. Moreover, by  Lemma \ref{lem.QR2}, such $v$ can be made of arbitrarily large size,
so we get $X\supset zLv_j$ for a sequence $v_j\in (L\cap N)^\perp$ tending to $\infty$.
The set 
$\limsup_{j\to \infty} v_j^{-1}A v_j$
contains a one-parameter subgroup $V\subset (L\cap N)^\perp$ by Lemma \ref{vAv}. 
Passing to a subsequence, there exists $y\in\liminf zLv_j$ and hence
 $$X\supset \limsup_{j\to \infty} (zLv_j) \supset
 y(L\cap N)\limsup_{j\to \infty}  (v_j^{-1} A v_j )\supset y(L\cap N) V.$$
Hence $X$ contains $y(L\cap N)V$, and hence the claim follows from $(2)_m$.

\end{document}